\documentclass[a4paper, 11pt]{amsart}
\usepackage{amssymb,amsmath,amsthm,stmaryrd,mathrsfs}
\usepackage{epsfig,color}

\newtheorem{theorem}{Theorem}
\newtheorem{lemma}[theorem]{Lemma}
\newtheorem{corollary}[theorem]{Corollary}
\newtheorem{proposition}[theorem]{Proposition}

\newtheorem{definition}[theorem]{Definition}
\newtheorem{remark}[theorem]{Remark}

\numberwithin{equation}{section}
\numberwithin{theorem}{section}

\newcommand{\op}[1]{\operatorname{\text{\rm #1}}}

\def\A{\mathcal A}
\def\B{\mathcal B}
\def\K{\mathcal K}
\def\Mass{\mathbf M}
\def\R{\mathbb R}
\def\ball{\mathbb{B}^{m+1}}
\def\sphere{\mathbb{S}^m}

\begin{document}

\title[Relative isoperimetric inequality]{Higher codimension relative isoperimetric inequality outside a convex set}
\author{Brian Krummel} 

\begin{abstract}
We consider an isoperimetric inequality for $(m+1)$-dimensional area minimizing submanifolds of arbitrary codimension which lie outside a convex set $\K \subset \R^{n+1}$ and are bounded by a submanifold of $\R^{n+1} \setminus \K$ and the convex set $\K$.  We show that the least value of the isoperimetric ratio is attained for an $(m+1)$-dimensional flat half-disk of $\R^{n+1}_+$.  This extends prior work of Choe, Ghomi, and Ritor\'{e} in codimension one and proves a conjecture of Choe in the case of relative area minimizers.
\end{abstract}

\maketitle

\setcounter{tocdepth}{1}
\tableofcontents

\section{Introduction}

\subsection{Overview}  We consider isoperimetric inequalities for submanifolds in open domains of Euclidean space.  In particular, let $1 \leq m \leq n$ be integers and $\K \subset \R^{n+1}$ be a closed convex subset with nonempty interior and smooth boundary.  Consider an $(m+1)$-dimensional area minimizing submanifold $R$ in $\R^{n+1} \setminus \K$ whose boundary $\partial R$ consists of both a portion $T = \partial R \setminus \K$ in $\R^{n+1} \setminus \K$ and a portion lying on $\partial \K$.  We want to minimize the \textit{relative isoperimetric ratio} 
\begin{equation} \label{iso_rat_fig_ref}
	\frac{|T|^{\frac{m+1}{m}}}{|R|} 
\end{equation}
where $|T|$ and $|R|$ denote the $m$ and $m+1$ dimensional areas of $T$ and $R$ respectively, see 
\begin{figure}
  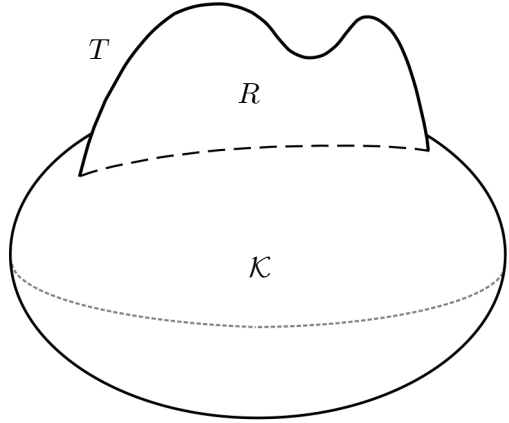\caption{{\small $(m+1)$-dimensional area minimizing submanifold $R$ lying outside a closed convex set $\K$ and bounded by an $m$-dimensional submanifold $T$ and a boundary portion on $\K$ (possibly not rectifiable).  The relative isoperimetric inequality concerns the minimum value of a scale-invariant ratio of the areas of $T$ and $R$ as in \eqref{iso_rat_fig_ref}.}}\label{introfigref}
\end{figure}
Figure \ref{introfigref}.  Note that the isoperimetric ratio is invariant under homotheties and rigid motions.  

One natural setting to consider such isoperimetric inequalities is the space of smooth submanifolds.  Unfortunately, the space of smooth submanifolds is not compact under smooth or weak measure theoretic limits, and area minimizers are known to admit singularities.  We will prove an isoperimetric inequality for integral currents of $\R^{n+1} \setminus \K$.  Integral currents are a certain generalization of oriented smooth submanifolds-with-boundary which allow for singularities and multiplicities greater than one.  Each $m$-dimensional integral current $T$ has an $(m-1)$-dimensional boundary $\partial T$ defined via Stoke's theorem and an area, or mass, $\Mass(T)$ which is the $m$-dimensional Hausdorff measure of $T$ counting multiplicity.  The space of integral currents is a natural setting for studying geometric variational problems such as isoperimetric inequalities and in particular has the advantage of a compactness theorem due to Federer and Fleming~\cite{FF60} which can be used to prove the existence of minimizers.  We will discuss the basic theory of integral currents in Subsection \ref{sec:prelims_currents}. 

The isoperimetric inequality for integral currents of $\R^{n+1}$ of higher codimension was proven by Almgren in~\cite{Alm86} and can be stated as follows.  

\noindent {\bf Isoperimetric inequality.}  {\it Let $1 \leq m \leq n$ be integers.  Let $T$ be an $m$-dimensional integral current of $\R^{n+1}$ with $\partial T = 0$ in $\R^{n+1}$ and $R$ be an $(m+1)$-dimensional area minimizing integral current of $\R^{n+1}$ with $\partial R = T$ in $\R^{n+1}$.  Then 
\begin{equation} \label{iso ineq}
	\frac{\Mass(T)^{\frac{m+1}{m}}}{\Mass(R)} 
		\geq \frac{\mathcal{H}^m(\partial \ball)^{\frac{m+1}{m}}}{\mathcal{H}^{m+1}(\ball)}, 
\end{equation}
where $\ball$ is the unit ball in $\R^{m+1}$ centered at the origin.  Equality holds true in \eqref{iso ineq} if and only if $R$ is a multiplicity one $(m+1)$-dimensional flat disk.}  

Now let us carefully consider the setting of the relative isoperimetric inequality.  Let $\K$ be any closed subset of $\R^{n+1}$.  Let $T$ is an $m$-dimensional integral current with zero boundary in $\R^{n+1} \setminus \K$, i.e.~$T$ lies in $\R^{n+1} \setminus \K$, $\partial T$ lies on $\K$, and $\partial T$ is possibly not rectifiable.  Consider the $(m+1)$-dimensional integral currents $R$ such that $\partial R = T$ in $\R^{n+1} \setminus \K$, i.e.~$R$ lies in $\R^{n+1} \setminus \K$ and is bounded by $T$ and a free boundary portion in $\K$ which is possibly not rectifiable, see Figure \ref{introfigref}.  For topological reasons, it is not true for a general closed set $\K$ and integral current $T$ that such a current $R$ exists.  In the special case that $\K$ is a closed convex subset of $\R^{n+1}$, we will show in Lemma \ref{isoper lower bound lemma} that there is at least one $(m+1)$-dimensional integral current $R$ with $\partial R = T$ in $\R^{n+1} \setminus \K$ and moreover we can choose $R$ so that $\Mass(R) \leq C(m,n,\K) \,\Mass(T)^{\frac{m+1}{m}}$.  Using the Federer-Fleming compactness theorem and semi-continuity of mass, one can show that if $\K$ is a closed subset of $\R^{n+1}$ and $T$ is an $m$-dimensional integral current bounding some $(m+1)$-dimensional integral current in $\R^{n+1} \setminus \K$, then there always exists at least one $(m+1)$-dimensional integral current $R$ with $\partial R = T$ in $\R^{n+1} \setminus \K$ which minimizes area.  It is known that an area minimizing integral current $R$ can admit interior singularities; for instance, consider the holomorphic variety $R = \{ (z,w) \in \mathbb{C}^2 : w^2 = z^3 \}$.  The interior of an $(m+1)$-dimensional area minimizing integral current $R$ is a locally real-analytic submanifold away from a singular set which has Hausdorff dimension at most $m-1$ due to a famous result of Almgren~\cite{Alm83}, also see~\cite{DeLSpa1}~\cite{DeLSpa2}~\cite{DeLSpa3}.  Little known about the boundary regularity of area minimizing integral currents in general. 

In the special case that $\K$ is a half-space, one can reflect $R$ across $\partial \K$ and apply the isoperimetric inequality in $\R^{n+1}$ to obtain 
\begin{equation} \label{rel iso ineq}
	\frac{\Mass(T)^{\frac{m+1}{m}}}{\Mass(R)} 
		\geq 2^{-\frac{1}{m}} \,\frac{\mathcal{H}^m(\partial \ball)^{\frac{m+1}{m}}}{\mathcal{H}^{m+1}(\ball)} 
\end{equation}
with equality if and only if $R$ is a multiplicity one $(m+1)$-dimensional flat half-disk with hemispherical boundary $T$ in $\R^{n+1} \setminus \K$ and meeting $\partial \K$ orthogonally.  It is natural to conjecture that a similar relative isoperimetric inequality might hold true for other closed sets $\K$.  In particular, Choe~\cite{ChoeSurvey} conjectured that \eqref{rel iso ineq} holds true in higher codimension for any closed convex set $\K$ with nonempty interior.  

In the codimension one setting $m = n$, the relative isoperimetric inequality reduces to the case where $R = \llbracket \Omega \rrbracket$ is a multiplicity one integral current associated with a Lebesgue measurable subset $\Omega \subset \R^{n+1} \setminus \K$ with finite perimeter.  Partial results for the codimension one relative isoperimetric inequality outside a convex set were previously obtained by Kim~\cite{Kim00} in the special case that $\K = U \times \R$ where $U$ is the epigraph of a convex $C^2$-function and by Choe~\cite{Choe03} in the special case that $\partial R \cap \partial \K$ is symmetric about $n$ hyperplanes of $\R^{n+1}$.  The codimension one relative isoperimetric inequality outside a convex set was ultimately proved by Choe, Ghomi, and Ritor\'{e} in~\cite{CGR07}. 

In higher codimension $m < n$, the relative isoperimetric inequality outside of a convex set was conjectured by Choe in~\cite[Open Problem 12.6]{ChoeSurvey}.  (Note that~\cite{ChoeSurvey} primarily concerns isoperimetric inequalities such as \eqref{iso ineq} and \eqref{rel iso ineq} in the case that $R$ is a minimal submanifold, not necessarily area minimizing.  Even the isoperimetric inequality \eqref{iso ineq} in $\R^{n+1}$ remains largely open outside of the case when $R$ is area minimizing, see~\cite[Open Problem 8.5]{ChoeSurvey}.)  Prior to~\cite{ChoeSurvey}, a partial result was obtained by Kim in~\cite{Kim98} in the special case of two-dimensional piece-wise smooth surfaces of $\R^{n+1}$, and also for two-dimensional piece-wise smooth surfaces of hyperbolic space $\mathbb{H}^{n+1}$.  Otherwise, the relative isoperimetric inequality in higher codimension has remained open.  We prove the sharp relative isoperimetric inequality for integral currents of arbitrary codimension in $\R^{n+1} \setminus \K$ for a convex set $\K$.  We thereby extend the main result of~\cite{CGR07} to higher codimension and prove the conjecture of Choe from~\cite{ChoeSurvey} in the special case that $R$ is relatively area minimizing.

\noindent {\bf Theorem A (Relative isoperimetric inequality).}  {\it Let $1 \leq m \leq n$ be integers.  Let $\K$ be a proper convex subset of $\R^{n+1}$, i.e.~a convex subset which has nonempty interior and is not equal to $\R^{n+1}$.  Let $T$ be an $m$-dimensional integral current of $\R^{n+1} \setminus \K$ with $\partial T = 0$ in $\R^{n+1} \setminus \K$ and $R$ be an $(m+1)$-dimensional area minimizing integral current in $\R^{n+1} \setminus \K$ with $\partial R = T$ in $\R^{n+1} \setminus \K$.  Then \eqref{rel iso ineq} holds true. 

If we additionally assume $\K$ is bounded and has $C^2$-boundary, equality holds true in \eqref{rel iso ineq} if and only if $T$ is a multiplicity one $m$-dimensional hemisphere and $R$ is a multiplicity one $(m+1)$-dimensional flat half-disk which meets $\partial \K$ orthogonally.}  

Note that the lower bound in \eqref{rel iso ineq} is sharp and, for a general proper convex set $\K$, is attained in the limit case of $T$ and $R$ concentrating at a boundary point of $\K$, as we will discuss below.

\subsection{Outline of the method} \label{sec:intro method}  To describe our approach, it will be useful to have the following notation and terminology. 

\begin{definition} \label{gamma defn}
Let $1 \leq m \leq n$ be integers and $\K$ be any closed subset of $\R^{n+1}$ with $\K \neq \R^{n+1}$.  We define  
\begin{equation*}
	\gamma_{m,n}(\K) = \inf_{T,R} \frac{\Mass(T)^{\frac{m+1}{m}}}{\Mass(R)}
\end{equation*}
where the infimum is over all $m$-dimensional integral currents $T$ and $(m+1)$-dimensional integral currents $R$ of $\R^{n+1} \setminus \K$ such that $\partial R = T$ in $\R^{n+1} \setminus \K$ and $R$ is relatively area minimizing in $\R^{n+1} \setminus \K$. 
\end{definition}

\begin{definition} 
Let $1 \leq m \leq n$ be integers and $\K$ be any closed subset of $\R^{n+1}$ with $\K \neq \R^{n+1}$.  Given an $m$-dimensional integral current $T_0$ and $(m+1)$-dimensional integral current $R_0$ of $\R^{n+1} \setminus \K$, we say the ordered pair of currents $(T_0,R_0)$ is a relative isoperimetric minimizer if $\partial R_0 = T_0$ in $\R^{n+1} \setminus \K$, $R_0$ is area minimizing in $\R^{n+1} \setminus \K$, and 
\begin{equation*}
	\frac{\Mass(T_0)^{\frac{m+1}{m}}}{\Mass(R_0)} \leq \frac{\Mass(T)^{\frac{m+1}{m}}}{\Mass(R)}
\end{equation*}
for every $m$-dimensional integral current $T$ and $(m+1)$-dimensional integral current $R$ of $\R^{n+1} \setminus \K$ such that $\partial R = T$ in $\R^{n+1} \setminus \K$ and $R$ is area minimizing in $\R^{n+1} \setminus \K$.
\end{definition}

Our approach to proving Theorem A follows a similar strategy as~\cite{Alm86} and~\cite{CGR07}.  In particular, one proves the existence of an isoperimetric minimizer $(T,R)$.  Then one proves an area-mean curvature inequality and uses it to characterize all isoperimetric minimizers $(T,R)$. 

Notice that there might not exist currents $T$ and $R$ which minimize the relative isoperimetric ratio in $\R^{n+1} \setminus \K$.  In particular, if $\partial \K$ is a $C^2$-submanifold with strictly positive principal curvatures, then it is impossible to attain the equality case from Theorem A.  Rather, the least value of the relative isoperimetric ratio is attained in the limit case of currents $T$ and $R$ concentrating at a point in $\partial \K$.  Let $\K$ be a proper convex set and $x_0 \in \partial \K$ be a regular point of $\K$ at which $\partial \K$ has a unique tangent plane.  For each $\rho > 0$ let $D(x_0,\rho)$ be an $(m+1)$-dimensional flat disk such that $D(x_0,\rho)$ has center $x_0$ and radius $\rho$ and $D(x_0,\rho)$ is orthogonal to the tangent space to $\partial \K$ at $x_0$.  As $\rho \downarrow 0$, $\rho^{-1} (\K - x_0)$ blows up to a half-space and $\rho^{-1} (D(x_0,\rho) \setminus \K - x_0)$ converges weakly to an $(m+1)$-dimensional unit flat disk with center $0$ and radius one and meeting the boundary of the half-space orthogonally.  Moreover, 
\begin{equation*}
	\lim_{\rho \downarrow 0} \frac{\mathcal{H}^m((\partial D(x_0,\rho)) \setminus \K)^{\frac{m+1}{m}}}{
		\mathcal{H}^{m+1}(D(x_0,\rho) \setminus \K)} 
		= 2^{-\frac{1}{m}} \,\frac{\mathcal{H}^m(\partial \ball)^{\frac{m+1}{m}}}{\mathcal{H}^{m+1}(\ball)}. 
\end{equation*}
Thus in the limit as $\rho \downarrow 0$ we obtain a relative isoperimetric minimizer for a half-space.  To rule out this limit case, we will assume by way of contradiction that \eqref{existence hyp} below holds true.

Let $\K$ be as in Theorem A and by approximation assume that $\K$ is bounded and has $C^2$-boundary.  Our existence result is as follows.

\noindent {\bf Theorem B (Existence result).}  {\it Let $1 \leq m \leq n$ be integers and $\K$ be a bounded proper convex subset of $\R^{n+1}$ with $C^2$-boundary.  Assume the least value $\gamma_{m,n}(\K)$ of the relative isoperimetric ratio (see Definition \ref{gamma defn} above) satisfies 
\begin{equation} \label{existence hyp}
	\gamma_{m,n}(\K) < 2^{-\frac{1}{m}} \,\frac{\mathcal{H}^m(\partial \ball)^{\frac{m+1}{m}}}{\mathcal{H}^{m+1}(\ball)}.
\end{equation}
Then there exists an $m$-dimensional integral current $T_0$ and an $(m+1)$-dimensional integral current $R_0$ of $\R^{n+1} \setminus \K$ such that $\partial R_0 = T_0$ in $\R^{n+1} \setminus \K$, $R_0$ is relatively area minimizing in $\R^{n+1} \setminus \K$, and 
\begin{equation} \label{existence concl}
	\frac{\Mass(T_0)^{\frac{m+1}{m}}}{\Mass(R_0)} \leq \frac{\Mass(T)^{\frac{m+1}{m}}}{\Mass(R)} 
\end{equation}
for every $m$-dimensional integral current $T$ and an $(m+1)$-dimensional integral current $R$ of $\R^{n+1} \setminus \K$ such that $\partial R = T$ in $\R^{n+1} \setminus \K$ and $R$ is relatively area minimizing in $\R^{n+1} \setminus \K$.}

The proof of Theorem B proceeds by the direct method.  In other words, we take a sequence of $m$-dimensional integral currents $T_j$ and an $(m+1)$-dimensional integral currents $R_j$ such that $R_j$ is relatively area minimizing with boundary $T_j$ in $\R^{n+1} \setminus \K$ and 
\begin{equation*}
	\lim_{j \rightarrow \infty} \frac{\Mass(T_j)^{\frac{m+1}{m}}}{\Mass(R_j)} = \gamma_{m,n}(\K). 
\end{equation*}
After passing to a subsequence, it suffices to consider the following three cases: 
\begin{enumerate}
	\item[(a)] $\lim_{j \rightarrow \infty} \Mass(R_j) = \infty$, 
	\item[(b)] $\lim_{j \rightarrow \infty} \Mass(R_j) = 0$, and  
	\item[(c)] $\lim_{j \rightarrow \infty} \Mass(R_j)$ exists as a positive real number. 
\end{enumerate}
In each case, we rescale $\K$, $T_j$, and $R_j$ to normalize the mass of $R_j$ before taking limits of $\K$, $T_j$, and $R_j$.  In case (a), we translate so that the origin is a point in $\K$.  Then we rescale, causing $\K$ to collapse to the origin as $j \rightarrow \infty$.  Thus after scaling $T_j$ and $R_j$ converge to integral currents $\widetilde{T}_0$ and $\widetilde{R}_0$ of $\R^{n+1}$.  We will show that $\widetilde{R}_0$ is area minimizing with boundary $\widetilde{T}_0$ in $\R^{n+1}$ and 
\begin{equation} \label{intro existence eqn}
	\frac{\Mass(\widetilde{T}_0)^{\frac{m+1}{m}}}{\Mass(\widetilde{R}_0)} = \gamma_{m,n}(\K), 
\end{equation}
which by \eqref{existence hyp} and the isoperimetric inequality in $\R^{n+1}$ this is impossible.  In case (b), we argue as in~\cite{Alm83} using the deformation theorem that $T_j$ concentrates near some boundary point of $x_j \in \partial \K$, see Lemma \ref{local concentration lemma} below, and we translate $x_j$ to the origin.  Then we rescale, causing $\K$ to blow up to a half-space as $j \rightarrow \infty$.  After scaling, $T_j$ and $R_j$ converge integral currents $\widetilde{T}_0$ and $\widetilde{R}_0$ of the half-space such that $\widetilde{R}_0$ is area minimizing with boundary $\widetilde{T}_0$ in the half-space and \eqref{intro existence eqn} holds true.  By \eqref{existence hyp} and the relative isoperimetric inequality in a half-space, \eqref{intro existence eqn} is impossible.  Finally in case (c), one does not need to translate or rescale, we simply let $T_j \rightarrow T_0$ and $R_j \rightarrow R_0$ to obtain the desired isoperimetric minimizer $(T_0,R_0)$. 

A key step is showing that after scaling $R_j$ converges to a relatively area minimizing integral current and the respective masses converge.  This requires showing that the mass of $R_j$ does not concentrate on the boundary of $\K$.  In the case of codimension one multiplicity one sets with finite perimeter as in~\cite{CGR07}, this is an obvious consequence of the compactness of $BV$-functions.  In higher codimension, sequences of relative area minimizers can have mass concentrating on the boundary of $\K$; consider for $j = 1,2,3,\ldots$ the multiplicity $j$ relatively area minimizing annulus $j \,\llbracket (B^m_{1+1/j}(0) \setminus B^m_1(0)) \times \{0\} \rrbracket$ in $\R^{n+1} \setminus \overline{B_1(0)}$ (see Subsection \ref{sec:prelims_currents} for notation).  In Lemma \ref{boundary rectifiability lemma2}, we will adapt arguments of Gr\"{u}ter in~\cite{Gru85} to show that if the relative isoperimetric ratio of $(T,R)$ is close to $\gamma_{m,n}(\K)$  -- as is true with the minimizing sequence $(T_j,R_j)$ above -- we have monotonicity formulas $s^{-1} \Mass(T \llcorner \{ x : \op{dist}(x,\K) < s \})$ and $s^{-1} \Mass(R \llcorner \{ x : \op{dist}(x,\K) < s \})$.  Consequently we obtain new estimates showing that the masses $T$ and $R$ cannot concentrate along $\partial \K$.  In the special case that $(T,R)$ is in fact a relative isoperimetric minimizer, these estimates tell us that $T$ and $R$ have rectifiable boundaries with finite mass along $\partial \K$.  With slight modification, Lemma \ref{boundary rectifiability lemma2} applies in the more general setting where $\K$ has a $C^2$-boundary, provided we remain in a tubular neighborhood of $\K$, see Corollary \ref{boundary rectifiability remark2}.  In addition to non-concentration of mass along $\partial \K$, in Section \ref{sec:nonconinf_sec} we adapt an argument of Almgren in~\cite{Alm86} to show that the mass of $T_j$ and $R_j$ cannot concentrate at infinity.

Now let us assume a relative isoperimetric minimizer $(T,R)$ exists (but not necessarily that \eqref{existence hyp} holds true).  We want to show that $T$ is a multiplicity one $m$-dimensional hemisphere and $R$ is an $(m+1)$-dimensional flat half-disk.  Then if \eqref{existence hyp} were true, by Theorem B there would exists a relative isoperimetric minimizer $(T,R)$ and such an $R$ would be a flat half-disk, contradicting \eqref{existence hyp}.  Moreover, equality holds true in \eqref{rel iso ineq} of Theorem A precisely when $R$ is a flat half-disk.  

To characterize relative isoperimetric minimizers $(T,R)$, we first compute the first variation of the relative isoperimetric ratio, showing that $T$ has bounded mean curvature with respect to variational vector fields tangent to $\partial \K$. 

\noindent {\bf Theorem C (First variational inequality).}  {\it Let $1 \leq m \leq n$ be integers and $\K$ be the closure of an open subset of $\R^{n+1}$ with $C^2$-boundary.  Let $T$ be an $m$-dimensional integral current and $R$ be an $(m+1)$-dimensional integral current of $\R^{n+1} \setminus \K$ such that $\partial R = T$ in $\R^{n+1} \setminus \K$, $R$ is relatively area minimizing in $\R^{n+1} \setminus \K$, and $(T,R)$ is a relative isoperimetric minimizer in $\R^{n+1} \setminus \K$.  Then $T$ has distributional mean curvature $\mathbf{H}_T \in L^{\infty}(\|T\|;\R^{n+1})$ in the sense that 
\begin{equation} \label{first variation T concl1}
	\int \op{div}_T \zeta(x) \,d\|T\|(x) = \int \mathbf{H}_T(x) \cdot \zeta(x) \,d\|T\|(x)
\end{equation}
for all $\zeta \in C^1_c(\R^{n+1};\R^{n+1})$ such that $\zeta(x)$ is tangent to $\partial \K$ at each $x \in \partial \K$, where $\op{div}_T \zeta(x)$ denotes the divergence of $\zeta$ computed with respect to the approximate tangent plane of $T$ at $x$ for $\|T\|$-a.e.~$x$.  Moreover, $\mathbf{H}_T(x)$ is orthogonal to the approximate tangent plane of $T$ at $x$ and satisfies 
\begin{equation} \label{first variation T concl2}
	|\mathbf{H}_T(x)| \leq H_0 
\end{equation}
for $\|T\|$-a.e.~$x$, where 
\begin{equation} \label{H0 defn}
	H_0 = \frac{m}{m+1} \,\frac{\Mass(T)}{\Mass(R)}. 
\end{equation} }

\begin{remark}
In the codimension one case $m = n$, by a straightforward modification of the proof of Theorem C we know that $T$ has constant scalar mean curvature $H_0$ as in \eqref{H0 defn}, see Remark \ref{first variation T codim one}.
\end{remark}

As an important consequence of Theorem C, by the work of Gr\"{u}ter and Jost~\cite{GJ86} the area of $T$ satisfies a monotonicity formula.  Using this monotonicity formula and a local version Lemma \ref{boundary rectifiability lemma4} of our non-concentration estimates of mass along $\partial \K$, we can show that $T$ has a relatively area minimizing tangent cone at each point of its support, including points on $\partial \K$.  (Note that Theorem C and monotonicity does does not require convexity of $\K$.)  Tangent cones will play a small but important role in our proof of Theorem D below.  Again notice that in the codimension one case of~\cite{CGR07} concentration of mass along $\partial \K$ is not an issue, whereas in higher codimension nonconcentration of mass along $\partial \K$ is essential, in particular for the existence of tangent cones.

Finally we prove the following area-mean curvature inequality, which we will use to show that $(T,R)$ is isoperimetric minimizing if and only if $R$ is a flat half-disk. 

\noindent {\bf Theorem D (Area-mean curvature characterization of hemispheres).}  {\it Let $1 \leq m \leq n$ be integers and $\K$ be a bounded proper convex subset of $\R^{n+1}$ with $C^2$-boundary.  Let $T$ be an $m$-dimensional integral current and $R$ be an $(m+1)$-dimensional integral current of $\R^{n+1} \setminus \K$ such that $T$ and $R$ have compact support in $\R^{n+1}$, $\partial R = T$ in $\R^{n+1} \setminus \K$, $R$ is relatively area minimizing in $\R^{n+1} \setminus \K$, and $(T,R)$ is a relative isoperimetric minimizer in $\R^{n+1} \setminus \K$.  Assume that $T$ has distributional mean curvature $\mathbf{H}_T \in L^{\infty}(\|T\|;\R^{n+1})$ in the sense that \eqref{first variation T concl1} holds true for every vector field $\zeta \in C^1_c(\R^{n+1};\R^{n+1})$ which is is tangent to $\partial \K$ and assume that 
\begin{equation*} 
	|\mathbf{H}_T(x)| \leq m
\end{equation*}
for $\|T\|$-a.e.~$x$.  Then 
\begin{equation} \label{AH halfspheres concl}
	\Mass(T) \geq \frac{1}{2} \,\mathcal{H}^m(\sphere),  
\end{equation}
where $\sphere = \partial \ball$.  Equality holds true in \eqref{AH halfspheres concl} if and only if $T$ is an multiplicity one $m$-dimensional unit hemisphere (lying in an $(m+1)$-dimensional affine plane) and $T$ meets $\partial \K$ orthogonally.}

The proof of Theorem D is similar to Almgren's proof of the area-mean curvature characterization of spheres in~\cite{Alm86}, in which one considers the convex hull $\A$ of the support of $T$ and computes the area of the set of outward unit normals to $\A$ to obtain a lower bound on $\Mass(T)$.  The main change from~\cite{Alm86} is that the support hyperplanes of $\A$ can touch the support of $T$ at both interior points of $T$, i.e.~points in $\op{spt} T \setminus \K$, or a boundary point of $T$, i.e.~points in $\op{spt} T \cap \partial \K$.  In~\cite{CGR07}, the boundary points are dealt with by noting that $T$ is a smooth hypersurface-with-boundary away from a singular set of Hausdorff dimension at most $m-7$ and using an analytical computation from the appendix of~\cite{CGR06}.  In codimension $> 1$ little is known about the boundary regularity of isoperimetric minimizers.  Instead we apply the concept of restricted support hyperplanes from~\cite{CGR06}.  Using tangent cones of $T$ we will show that the restricted support hyperplanes touch the support of $T$ only at interior points, where the arguments of~\cite{Alm86} apply.  By~\cite{CGR06} (with some modification) the area of the unit normals to restricted support hyperplanes is $\geq \tfrac{1}{2} \,\mathcal{H}^n(\mathbb{S}^n)$, which is what is needed to obtain \eqref{AH halfspheres concl}.

\subsection{Organization of the paper}  We discuss notation and the basic facts about convex sets and integral currents in Section \ref{sec:preliminaries}.  Section \ref{sec:weak rel iso sec} contains the proof of existence of relative area minimizers outside a convex set, including a non-sharp relative isoperimetric inequality, and Section \ref{sec:halfspace rel iso sec} contains a rigorous proof of the sharp relative isoperimetric inequality for a half-space.  The arguments in both sections are geometrically straightforward but we need to be careful about the boundaries of $T$ and $R$ possibly not being rectifiable along $\partial \K$.   Section \ref{sec:bdry_rect_sec} will concern the non-concentration of mass along $\partial \K$ and Section \ref{sec:nonconinf_sec} will concern the non-concentration of mass at infinity.  Having shown non-concentration of mass, in Section \ref{sec:convergence_sec} we will argue that any convergent sequence of relative area minimizers $R_j$ has a relative area minimizing limit $R$ and the respective masses converge.  In Section \ref{sec:existence_sec} we put this all together to prove the existence result Theorem B.  In Section \ref{sec:variation sec} we compute the first variation of the relative isoperimetric ratio as in Theorem C.  Then in Section \ref{sec:monotonicity sec} we use Theorem C and~\cite{GJ86} to establish a monotonicity formula and the existence of tangent cones for $T$.  In Section \ref{sec:restricted normal sec} we discuss the unit normal cone and restricted support hyperplanes from~\cite{CGR06}, which we use in Section \ref{sec:AH sec} to prove Theorem D.  Finally, in Section \ref{sec:main proof sec} we use Theorems B, C, and D to prove Theorem A.

\subsection*{Acknowledgements}

This work was supported by the NSF thought the DMS FRG Grant 1361122 and DMS FRG Grant 1361185.  The author would like to thank Francesco Maggi for his mentorship and encouragement in pursuing this project.

\section{Preliminaries and notation}  \label{sec:preliminaries}

In this section, we cover some basic notation and facts.  In particular, in Subsection \ref{sec:prelims_notation} we establish some general notation.  In Subsection \ref{sec:prelims_sets} we will discuss convex subsets of $\R^{n+1}$ and the distance function, projection map, and Gauss map associated with a convex subset.  In Subsection \ref{sec:prelims_currents} we will discuss currents and integral currents.  We refer the reader to~\cite[Chapter 4]{Fed69} or~\cite[Chapter 6]{Sim83} for a more detailed discussion of the theory of currents.

\subsection{Basic notation} \label{sec:prelims_notation}

$n \geq 1$ is a fixed integer.  $x = (x_1,x_2,\ldots,x_{n+1})$ denotes a point of the $(n+1)$-dimensional Euclidean space $\R^{n+1}$. 

For each integer $k \geq 1$, $\R^k_+ = \{ (x_1,x_2,\ldots,x_k) \in \R^k : x_k > 0 \}$ is the open upper half-space in $\R^k$ and $\R^k_- = \{ (x_1,x_2,\ldots,x_k) \in \R^k : x_k < 0 \}$ is the open lower half-space in $\R^k$. 

For each integer $k \geq 1$, we let $B^k_{\rho}(y) = \{ x \in \R^k : |x-y| < \rho \}$ denote the open ball of $\R^k$ with center $y \in \R^k$ and radius $\rho > 0$.  When $k = n+1$ we let $B_{\rho}(y) = B^{n+1}_{\rho}(y)$. 

For each $y \in \R^{n+1}$ and $\rho > 0$, $\eta_{y,\rho} : \R^{n+1} \rightarrow \R^{n+1}$ is the map defined by $\eta_{y,\rho}(x) = (x-y)/\rho$ for all $x \in \R^{n+1}$. 

For each $k = 0,1,2,\ldots,n+1$, $\mathcal{H}^k$ denotes the $k$-dimensional Hausdorff measure on $\R^{n+1}$. 

$\mathcal{L}^{n+1}$ denotes the $(n+1)$-dimensional Lebesgue measure on $\R^{n+1}$ and $\mathcal{L}^1$ denotes the one-dimensional Lebesgue measure on $\R$.

For each integer $k \geq 1$, $\mathbb{B}^k = B^k_1(0)$, $\mathbb{S}^k = \partial B^{k+1}_1(0)$, and $\omega_k = \mathcal{H}^k(\mathbb{B}^k)$. 

For $A \subseteq \R^{n+1}$, $\op{int} A$ denotes the interior of $A$ in $\R^{n+1}$ and $\overline{A}$ denotes the closure of $A$ in $\R^{n+1}$.  

For $x \in \R^{n+1}$ and $A \subseteq \R^{n+1}$, $\op{dist}(x,A) = \inf_{y \in A} |x - y|$.  When $A = \emptyset$, $\op{dist}(x,\emptyset) = \infty$.  

For sets $A,B \subseteq \R^{n+1}$, 
\begin{equation*}
	\op{dist}_{\mathcal{H}}(A,B) = \max\left\{ \sup_{x \in A} \op{dist}(x,B), \,\sup_{x \in B} \op{dist}(x,A), \,0 \right\} 
\end{equation*}
is the Hausdorff distance between $A$ and $B$.  $\op{dist}_{\mathcal{H}}(A,\emptyset) = \infty$ if $A \neq \emptyset$ and $\op{dist}_{\mathcal{H}}(\emptyset,\emptyset) = 0$. 

Given $A_i,A \subseteq \R^{n+1}$, we say $A_i \rightarrow A$ in Hausdorff distance if $\op{dist}_{\mathcal{H}}(A_i,A) \rightarrow 0$. 

Given $A_i,A \subseteq \R^{n+1}$, we say $A_i \rightarrow A$ locally in Hausdorff distance if for every $r > 0$
\begin{equation*}
	\lim_{i \rightarrow \infty} \max\left\{ \sup_{x \in A_i \cap B_r(0)} \op{dist}(x,A), \,\sup_{x \in A \cap B_r(0)} \op{dist}(x,A_i), \,0 \right\} = 0. 
\end{equation*}
In particular, for each $r > 0$, $A \cap B_r(0) = \emptyset$ if and only if $A_i \cap B_r(0) = \emptyset$ for all large $i$. 

$\llbracket 0,1 \rrbracket$ denotes the one-dimensional integral current associated with the open interval $(0,1) \subset \R$.

\subsection{Convex sets} \label{sec:prelims_sets}  Throughout we consider the following type of convex set.

\begin{definition} \label{proper convex defn}
We say $\K \subseteq \R^{n+1}$ is a proper convex set if $\K$ is convex, $\K$ has nonempty interior in $\R^{n+1}$, and $\K \neq \R^{n+1}$. 
\end{definition}

To each closed convex subset $\K \subset \R^{n+1}$ we associate the following functions.  We define the \textit{distance function} $d_{\K} : \R^{n+1} \rightarrow [0,\infty)$ by $d_{\K}(x) = \op{dist}(x,\K)$ for each $x \in \R^{n+1}$.  Since $\K$ is convex, there is a well-defined \textit{projection map} $\xi_{\K} : \R^{n+1} \rightarrow \K$ such that $\xi_{\K}(x)$ is the closest point to $x$ on $\K$ for each $x \in \R^{n+1} \setminus \K$ and $\xi_{\K}(x) = x$ for each $x \in \K$.  We define the \textit{Gauss map} $\nu_{\K} : \R^{n+1} \setminus \K \rightarrow \mathbb{S}^n$ by $\nu_{\K}(x) = \frac{x-\xi_{\K}(x)}{d_{\K}(x)}$.  We have $d_{\K} \in C^{0,1}(\R^{n+1}) \cap C^{1,1}_{\rm loc}(\R^{n+1} \setminus \K)$ with $\op{Lip} d_{\K} = 1$ and $\nabla d_{\K}(x) = \nu_{\K}(x)$ for all $x \in \R^{n+1} \setminus \K$, $\xi_{\K} \in C^{0,1}(\R^{n+1};\R^{n+1})$ with $\op{Lip} \xi_{\K} = 1$, and $\nu_{\K} \in C^{0,1}_{\rm loc}(\R^{n+1} \setminus \K;\R^{n+1})$ with $\op{Lip}_{\{d_{\K} \geq s\}} \nu_{\K} \leq 3/s$ for each $s > 0$. 

Suppose $\K$ is a bounded proper convex subset with $C^2$-boundary.  Then the condition that $\K$ is convex is equivalent to the condition that at each $p \in \partial \K$ the principal curvatures of $\partial \K$ at $p$ computed with respect to the outward unit normal to $\K$ are all nonnegative.  Moreover, $\nu_{\K}$ extends to a continuous function on $\R^{n+1} \setminus \op{int} \K$ such that $\nu_{\K} |_{\partial \K}$ is the outward unit normal to $\K$ and $\nu_{\K}(x) = \nu_{\K}(\xi_{\K}(x))$ for all $x \in \R^{n+1} \setminus \K$.  We have $d_{\K} \in C^2(\R^{n+1} \setminus \op{int} \K)$ and $\xi_{\K}, \nu_{\K} \in C^1(\R^{n+1} \setminus \op{int} \K;\R^{n+1})$ with 
\begin{gather} \label{grad d xi nu}
	\nabla d_{\K}(x) = \nu_{\K}(x), \quad \nabla_{e_i} \xi_{\K}(x) = \frac{1}{1 + \kappa_i \,d_{\K}(x)} \,e_i, \quad 
		\nabla_{e_i} \nu_{\K}(x) = \frac{\kappa_i}{1 + \kappa_i \,d_{\K}(x)} \,e_i, \\
	\nabla_{\nu_{\K}(x)} \,\xi_{\K}(x) = \nabla_{\nu_{\K}(x)} \,\nu_{\K}(x) = 0 \nonumber 
\end{gather}
for each $x \in \R^{n+1} \setminus \K$, where $e_1,e_2,\ldots,e_n$ are principal directions of $\partial \K$ at $\xi_{\K}(x)$ with corresponding principal curvatures $\kappa_1,\kappa_2,\ldots,\kappa_n$. 

Suppose $\K$ is the closure of a bounded open convex subset with $C^2$-boundary but $\K$ is not necessarily convex.  In this case the closest point projection map $\xi_{\K}$ onto $\partial \K$ may not be well-defined on all of $\R^{n+1} \setminus \op{int} \K$.  Let $\kappa_0 \geq 0$ be the infimum of $1/\rho$ over all radii $\rho > 0$ such that $\partial \K$ satisfies an interior and exterior sphere condition of radius $\rho$ at every point of $\partial \K$, that is 
\begin{equation} \label{kappa0 defn} 
	\kappa_0 = \inf \left\{ \frac{1}{\rho} : \forall y \in \partial \K, \,B_{\rho}(y - \rho \nu_{\K}(y)) \cap \partial \K 
		= B_{\rho}(y + \rho \nu_{\K}(y)) \cap \partial \K = \{y\} \right\} ,
\end{equation}
where $\nu_{\K}(y)$ is the outward unit normal to $\K$ at $y$.  Notice that $|\kappa_i| \leq \kappa_0$ whenever $\partial \K$ has principal curvatures $\kappa_1,\kappa_2,\ldots,\kappa_n$ at $y$.  Let 
\begin{equation*} 
	U = \{ x \in \R^{n+1} \setminus \op{int} \K : \op{dist}(x,\K) < 1/\kappa_0 \} .
\end{equation*}
Then there is a well-defined projection map $\xi_{\K} : U \cup \K \rightarrow \partial \K$ such that $\xi_{\K}(x)$ is the closest point to $x$ on $\K$ for each $x \in U$ and $\xi_{\K}(x) = x$ for each $x \in \K$.  We define the distance function $d_{\K} : \R^{n+1} \rightarrow [0,\infty)$ by $d_{\K}(x) = \op{dist}(x,\K)$ for each $x \in \R^{n+1}$.  We have a Gauss map $\nu_{\K} : U \rightarrow \mathbb{S}^n$ defined by $\nu_{\K} |_{\partial \K}$ being the outward unit normal to $\K$ and $\nu_{\K}(x) = \nu_{\K}(\xi_{\K}(x))$ for all $x \in U$.  It is well-known that $d_{\K} \in C^2(U)$ and $\xi_{\K}, \nu_{\K} \in C^1(U;\R^{n+1})$, $x = \xi_{\K}(x) + d_{\K}(x) \,\nu_{\K}(x)$ for all $x \in U$, and \eqref{grad d xi nu} holds true for all $x \in \op{int} U$.

\subsection{Currents} \label{sec:prelims_currents}  For each $m = 1,2,\ldots,n+1$, $\Lambda_m(\R^{n+1})$ denotes the space of $m$-vectors of $\R^{n+1}$ and $\Lambda^m(\R^{n+1})$ denotes the dual space of $m$-covectors of $\R^{n+1}$.  For each open subset $U \subseteq \R^{n+1}$ and $m = 1,2,\ldots,n+1$, we let $\mathcal{D}^m(U)$ denote the space of smooth $m$-forms $\omega : U \rightarrow \Lambda^m(\R^{n+1})$ with compact support and equip $\mathcal{D}^m(U)$ with the standard locally convex topology.  When $m = 0$, we let $\mathcal{D}^0(U) = C^{\infty}_c(U)$.  

Let $U$ be any open subset of $\R^{n+1}$.  An \textit{$m$-dimensional current} $T$ of $U$ is a continuous linear functional $T : \mathcal{D}^m(U) \rightarrow \R$.  We let $\mathcal{D}_m(U)$ denote the space of all $m$-dimensional currents $T$ of $U$.  For each $T \in \mathcal{D}_m(U)$, the \textit{support} $\op{spt} T$ of $T$ is the intersection of all closed sets $K$ such that $T(\omega) = 0$ for all $\omega \in \mathcal{D}^m(U)$ with $\op{spt} \omega \subset U \setminus K$.  Given an $m$-dimensional current $T \in \mathcal{D}_m(U)$, the \textit{boundary} $\partial T \in \mathcal{D}_{m-1}(U)$ is defined by 
\begin{equation*}
	\partial T(\omega) = T(d\omega) 
\end{equation*}
for all $\omega \in \mathcal{D}^{m-1}(U)$, where $d$ is the exterior derivative on differential forms.  Since $d^2 = 0$, $\partial^2 T = 0$ for every current $T$.  For each current $T \in \mathcal{D}_m(U)$ and open set $W \subseteq U$, the \textit{mass}, or area, $\Mass_W(T)$ is defined by 
\begin{align*}
	\Mass_W(T) = \sup \big\{ T(\omega) : \omega \in \mathcal{D}^m(U), \,\op{spt} \omega \subset W, \,\sup_U |\omega| \leq 1 \big\} .
\end{align*}
When $W = U$, we let $\Mass(T) = \Mass_{U}(T)$.  If $T \in \mathcal{D}_m(U)$ such that $\Mass_W(T) < \infty$ for all $W \subset\subset U$, there exists a unique Radon measure $\|T\|$ such that $\|T\|(W) = \Mass_W(T)$ for every open set $W \subset U$ and $\op{spt} T = \op{spt} \|T\|$. 

An important example of an $m$-dimensional current is the current $T = \llbracket M \rrbracket$ associated with smooth oriented $m$-dimensional submanifold-with-boundary $M$, which is given by 
\begin{equation*}
	\llbracket M \rrbracket(\omega) = \int_M \omega 
\end{equation*}
for all $\omega \in \mathcal{D}^m(U)$.  In this case, by Stoke's theorem 
\begin{equation*}
	\partial \llbracket M \rrbracket(\omega) = \int_M d\omega = \int_{\partial M} \omega = \llbracket \partial M \rrbracket(\omega)
\end{equation*}
for all $\omega \in \mathcal{D}^{m-1}(U)$ and thus $\partial \llbracket M \rrbracket = \llbracket \partial M \rrbracket$ is the $(m-1)$-dimensional current associated with the submanifold boundary $\partial M$ of $M$.  It is readily verified that the mass of $\llbracket M \rrbracket$ is the $m$-dimensional Hausdorff measure of $M$ 
\begin{equation*}
	\Mass_W(\llbracket M \rrbracket) = \mathcal{H}^m(M \cap W) 
\end{equation*}
for every open set $W \subset U$.  

We say an $m$-dimensional current $T \in \mathcal{D}_m(U)$ is \textit{locally integer-multiplicity rectifiable} if 
\begin{equation} \label{rectifiable current defn}
	T(\omega) = \int_M \langle \omega(x) , \xi(x) \rangle \,\theta(x) \,d\mathcal{H}^m(x) 
\end{equation}
for all $\omega \in \mathcal{D}^m(U)$, where $M \subseteq U$ is a countably $m$-rectifiable set, $\theta : M \rightarrow \mathbb{Z}_+$ is a locally $\mathcal{H}^m$-integrable function (called the \textit{multiplicity function}), and $\xi : M \rightarrow \Lambda^m(\R^{n+1})$ is a locally $\mathcal{H}^m$-measurable function (called the \textit{orientation}) such that for $\mathcal{H}^m$-a.e.~$x \in M$ we have $\xi(x) = \xi_1 \wedge \xi_2 \wedge \cdots \wedge \xi_m$ for some orthonormal basis $\xi_1,\xi_2,\ldots,\xi_m$ of the approximate tangent plane to $M$ at $x$.  We will often call the approximate tangent plane to $M$ at $x \in M$ the \textit{approximate tangent plane of $T$ at $x$}.  (See~\cite[Section 11]{Sim83} for a discussion of rectifiable sets and approximate tangent planes.)  For each locally integer-multiplicity rectifiable current $T \in \mathcal{D}_m(U)$, 
\begin{equation*}
	\Mass_W(T) = \int_{M \cap W} \theta \,d\mathcal{H}^m
\end{equation*} 
for every open set $W \subseteq U$ and in particular $\|T\| = \theta \,d\mathcal{H}^m$.  We let $\mathcal{I}_{m,{\rm loc}}(U)$ denote the space of all $m$-dimensional locally integer-multiplicity rectifiable currents of $U$.  We let $\mathbf{I}_{m,{\rm loc}}(U)$ be the space of \textit{locally integral currents} of $U$, which consists of all currents $T \in \mathcal{I}_{m,{\rm loc}}(U)$ with $\partial T \in \mathcal{I}_{m-1,{\rm loc}}(U)$.  We let $\mathbf{I}_m(U)$ be the space all currents $T \in \mathbf{I}_{m,{\rm loc}}(U)$ such that $\Mass(T) + \Mass(\partial T) < \infty$.  (Note that in~\cite{Fed69}, the space of integral currents is defined to contain all $T \in \mathcal{I}_{m,{\rm loc}}(U)$ with compact support in $U$.)

Observe that if $T \in \mathcal{I}_{m,{\rm loc}}(U)$ is given by \eqref{rectifiable current defn} and satisfies $\Mass(T) < \infty$, the formula \eqref{rectifiable current defn} also defines an integer-multiplicity rectifiable current of $\R^{n+1}$ and thus we may regard $T$ as a current in $\mathcal{I}_{m,{\rm loc}}(\R^{n+1})$.  Also, one readily checks that $\|T\|(\R^{n+1} \setminus U) = 0$.  To avoid confusion, when considering $T \in \mathcal{I}_{m,{\rm loc}}(U)$ we take $\op{spt} T$ to mean the support of $T$ as a current of $\R^{n+1}$ (as opposed to $\op{spt} T \cap U$) and we will be explicit about which portion of $\partial T$ we are discussing.  Notice that if $T \in \mathbf{I}_m(U)$ and $\op{spt} T$ is a relatively compact subset of $U$, then $T \in \mathbf{I}_m(\R^{n+1})$.  However, when $T \in \mathbf{I}_{m,{\rm loc}}(U)$ and $\Mass(T) < \infty$ it does not necessarily follow that $T \in \mathbf{I}_{m,{\rm loc}}(\R^{n+1})$ since the boundary of $T$ might not be rectifiable along $\overline{U} \setminus U$.   

We consider the following operations on integer-multiplicity rectifiable currents (see~\cite[Sections 4.1 and 4.3]{Fed69} or~\cite[Sections 26, 27, and 28]{Sim83}): 

For each $T \in \mathcal{I}_{m,{\rm loc}}(U)$ and $\|T\|$-measurable subset $A \subseteq U$, we let $T \llcorner A \in \mathcal{I}_{m,{\rm loc}}(U)$ denote the \textit{restriction} of $T$ to $A$. 

For each $S \in \mathcal{I}_{l,{\rm loc}}(U)$ and $T \in \mathcal{I}_{m,{\rm loc}}(V)$, we let $S \times T \in \mathcal{I}_{l+m,{\rm loc}}(U \times V)$ denote the \textit{cross product} of $S$ and $T$. 

For each $T \in \mathcal{I}_{m,{\rm loc}}(U)$ and a Lipschitz function $f : U \rightarrow V$ such that $f |_{\op{spt} T}$ is proper, i.e.~$f^{-1}(K) \cap \op{spt} T$ is compact whenever $K \subset V$ is relatively compact, we let $f_{\#} T \in \mathcal{I}_{m,{\rm loc}}(V)$ denote the \textit{pushforward} of $T$ by $f$. 

Let $T \in \mathbf{I}_{m,{\rm loc}}(U)$ and a Lipschitz function $f : U \rightarrow \R$.  For $\mathcal{L}^1$-a.e.~$t \in \R$, we let $\langle T, f, t \rangle \in \mathbf{I}_{m-1,{\rm loc}}(U)$ denote the \textit{slice} of $T$ by $f$ at $t$.  Note that by the coarea formula, for each $-\infty \leq a < b \leq \infty$ and open set $W \subseteq U$  
\begin{equation*}
	\int_a^b \Mass_W(\langle T, f, t \rangle) \,dt \leq \Mass_W(T \llcorner \{ a < f < b \}). 
\end{equation*}
Thus for every $-\infty < a < b < \infty$, $W \subseteq U$, and $\vartheta \in (0,1)$, 
\begin{equation} \label{meas_good_slices}
	\mathcal{L}^1\left\{ t \in (a,b) : \langle T, f, t \rangle \in \mathbf{I}_{m-1,{\rm loc}}(U), \, 
		\Mass_W(\langle T, f, t \rangle) > \frac{\Mass_W(T \llcorner \{ a < f < b \})}{\vartheta \,(b-a)} \right\} 
	< \vartheta \,(b-a). 
\end{equation}

There are three important ways one can define convergence for integral currents.  Let $T_j,T \in \mathcal{D}_m(U)$.  We say $T_j \rightarrow T$ in the \textit{mass norm topology} if $\Mass_W(T_j - T) \rightarrow 0$ for all $W \subset\subset U$.  We say $T_j \rightarrow T$ \textit{weakly} if $T_j(\omega) \rightarrow T(\omega)$ for all $\omega \in \mathcal{D}^m(U)$.  For each $W \subset\subset U$ we define the \textit{flat semi-norm} $\mathcal{F}_W$ on $\mathbf{I}_{m,{\rm loc}}(U)$ by 
\begin{equation*}
	\mathcal{F}_W(T) = \inf \big\{ \Mass_W(A) + \Mass_W(B) : A \in \mathbf{I}_{m+1,{\rm loc}}(W), \, B \in \mathbf{I}_{m,{\rm loc}}(W), \, 
		T = \partial A + B \text{ in } W \big\} 
\end{equation*}
for all $T \in \mathbf{I}_{m,{\rm loc}}(U)$.  Given $T_j,T \in \mathbf{I}_{m,{\rm loc}}(U)$, we say $T_j \rightarrow T$ in the \textit{flat norm topology} if $\mathcal{F}_W(T_j - T) \rightarrow 0$ for all $W \subset\subset U$.  Mass norm convergence is much strong than weak convergence or flat norm convergence.  For integral currents, weak convergence is the equivalent to flat norm convergence.  We have the following well-known compactness theorems, the first of which is an easy consequence of the Banach-Alaoglu theorem and the second is due to Federer and Fleming in~\cite{FF60}.

\begin{theorem}
For any sequence $T_j \in \mathcal{D}_m(U)$ such that $\sup_{\,j \geq 1} \Mass_W(T_j) < \infty$ for all $W \subset\subset U$, there exists a subsequence $\{j'\} \subset \{j\}$ and $T \in \mathcal{D}_m(U)$ such that $T_{j'} \rightarrow T$ weakly and 
\begin{equation} \label{intro semicontinuity mass}
	\Mass_W(T) \leq \liminf_{i \rightarrow \infty} \Mass_W(T_j) 
\end{equation}
for all $W \subset\subset U$. 
\end{theorem}

\begin{theorem}[Federer and Fleming compactness theorem]
For any sequence $T_j \in \mathbf{I}_{m,{ \rm loc}}(U)$ such that $\sup_{\,j \geq 1} (\Mass_W(T_j) + \Mass_W(T_j)) < \infty$ for all $W \subset\subset U$, there exists a subsequence $\{j'\} \subset \{j\}$ and $T \in \mathbf{I}_{m,{ \rm loc}}(U)$ such that $T_{j'} \rightarrow T$ in the flat norm topology.
\end{theorem}


\begin{definition} \label{rel area min defn}
Let $1 \leq m \leq n$ be integers and $\K$ be any closed subset of $\R^{n+1}$.  We say $R \in \mathbf{I}_{m+1}(\R^{n+1} \setminus \K)$ is relatively area minimizing in $\R^{n+1} \setminus \K$ if 
\begin{equation*}
	\Mass(R) \leq \Mass(Q)
\end{equation*}
for every $Q \in \mathbf{I}_{m+1}(\R^{n+1} \setminus \K)$ such that $\partial Q = \partial R$ in $\R^{n+1} \setminus \K$ and $R-Q$ has compact support as a current of $\R^{n+1}$.  (Note that we do not require $R$ and $Q$ to agree on an open neighborhood of $\K$.)
\end{definition}

By a standard application of the direct method using the Federer-Fleming compactness theorem and semi-continuity of mass \eqref{intro semicontinuity mass}, whenever $T \in \mathbf{I}_m(\R^{n+1} \setminus \K)$ is the boundary of some $(m+1)$-dimensional integral current in $\R^{n+1} \setminus \K$, there exists an integral current $R \in \mathbf{I}_{m+1}(\R^{n+1} \setminus \K)$ such that $\partial R = T$ in $\R^{n+1} \setminus \K$ and $R$ is relatively area minimzing in $\R^{n+1} \setminus \K$.  If additionally $T$ has compact support, then any relatively area minimizing current $R$ with boundary $T$ in $\R^{n+1} \setminus \K$ has compact support.  If $\K \subset \R^{n+1}$ is any closed set and $R \in \mathbf{I}_{m+1}(\R^{n+1})$ with $\|R\|(\mathcal{K}) = 0$, then $R$ being relatively area minimizing in $\R^{n+1} \setminus \K$ implies $R$ is area minimizing in $\R^{n+1}$, but $R$ being area minimizing in $\R^{n+1}$ does not imply $R$ is relatively area minimizing in $\R^{n+1} \setminus \K$; consider for $m = n-1$ and $\varepsilon > 0$ sufficiently small $\R^{n+1} \setminus \K = R^{n+1}_+$ and $R = \llbracket B^n(0) \times \{\varepsilon\} \rrbracket$.

\begin{definition}
We say $T$ is a multiplicity one $m$-dimensional hemisphere of $\R^{n+1}$ if 
$$T = \eta_{y,\rho \#} q_{\#} \llbracket (\{0\} \times \mathbb{S}^m) \cap \R^{n+1}_+ \rrbracket$$ 
for some $y \in \R^{n+1}$, $\rho > 0$, and orthogonal transformation $q$ of $\R^{n+1}$ (perhaps orientation reversing). 

Similarly, we say $R$ is a multiplicity one $(m+1)$-dimensional flat half-disk of $\R^{n+1}$ if 
$$R = \eta_{y,\rho \#} q_{\#} \llbracket (\{0\} \times \mathbb{B}^{m+1}) \cap \R^{n+1}_+ \rrbracket$$ 
for some $y \in \R^{n+1}$, $\rho > 0$, and orthogonal transformation $q$ of $\R^{n+1}$. 
\end{definition}

\section{Non-sharp relative isoperimetric inequality} \label{sec:weak rel iso sec}

Observe that given any closed subset $\K \subset \R^{n+1}$ and current $T \in \mathbf{I}_m(\R^{n+1} \setminus \K)$ with $\partial T = 0$ in $\R^{n+1} \setminus \K$ there might not exist any current $R \in \mathbf{I}_{m+1}(\R^{n+1} \setminus \K)$ with $\partial R = T$ in $\R^{n+1} \setminus \K$.  In the special case that $\K$ is a closed convex subset of $\R^{n+1}$, we will show that such a current $R$ does indeed exist.  Moreover, we can choose $R$ so that it satisfies the relative isoperimetric inequality $\Mass(R) \leq C \,\Mass(T)^{\frac{m+1}{m}}$ for some constant $C = C(m,n,\K) \in (0,\infty)$.  Note that we will later determine the optimal constant $C$ in Theorem A.

\begin{lemma} \label{isoper lower bound lemma}
Let $1 \leq m \leq n$ be integers.  Let $\K$ be a closed convex subset of $\R^{n+1}$ and $\K \neq \R^{n+1}$.  There exists a constant $C_0 = C_0(m,n,\K) \in (0,\infty)$ such that for each $T \in \mathbf{I}_m(\R^{n+1} \setminus \K)$ with $\partial T = 0$ in $\R^{n+1} \setminus \K$ there exists $R \in \mathbf{I}_{m+1}(\R^{n+1} \setminus \K)$ such that $\partial R = T$ in $\R^{n+1} \setminus \K$ and   
\begin{equation} \label{isoper lower bound}
	\Mass(R) \leq C_0 \,\Mass(T)^{\frac{m+1}{m}}. 
\end{equation} 
\end{lemma}

\begin{remark} \label{gamma zero rmk}
Lemma \ref{isoper lower bound lemma} asserts that $\gamma_{m,n}(\K) > 0$ whenever $\K$ is a closed convex set.  For general closed sets $\K$ one can certainly have $\gamma_{m,n}(\K) = 0$.  For instance when $m = n$, $\gamma_{m,n}(\K) = 0$ if $\K$ has a cusp, e.g.~$\K \cap B_1(0) = \{ (x',x_{n+1}) \in B_1(0) : x_{n+1} \leq |x'|^{1/2} \}$. 
\end{remark}

\begin{proof}[Proof of Lemma \ref{isoper lower bound lemma}]
Let us first look at the special case of $T \in \mathbf{I}_m(\R^{n+1})$ with $\|T\|(\mathcal{K}) = 0$ and $\partial T = 0$ in $\R^{n+1} \setminus \K$.  Consider $T - \xi_{\K \#} T \in \mathbf{I}_m(\R^{n+1})$.  Since $T$ has compact support, $\xi_{\K} |_{\op{spt} T}$ is proper and thus $T - \xi_{\K \#} T$ is well-defined.  By~\cite[4.1.15]{Fed69} (also see~\cite[Lemma 26.24]{Sim83}), since $\Mass(T) + \Mass_{\R^{n+1}}(\partial T) < \infty$ and $\xi_{\K}(x) = x$ for all $x \in \op{spt} \partial T \subseteq \K$, $\partial (T - \xi_{\K \#} T) = 0$ in $\R^{n+1}$.  Thus we can take $R$ to be an area minimizing integral current of $\R^{n+1}$ with $\partial R = T - \xi_{\K \#} T$ in $\R^{n+1}$.  In general, $T$ might not have compact support or the boundary of $T$ might not be rectifiable on $\K$, so we instead argue as follows. 

\noindent \textit{Step 1.}  Suppose $T \in \mathbf{I}_m(\R^{n+1} \setminus \K)$ with compact support and $\partial T = 0$ in $\R^{n+1} \setminus \K$.  For each $s > 0$, let $\K_s = \{ x : d_{\K}(x) \leq s\}$ and $T_s = T \llcorner \{d_{\K} > s\}$.  By slicing theory, $T_s \in \mathbf{I}_m(\R^{n+1})$ for $\mathcal{L}^1$-a.e.~$s > 0$.  Thus $\widetilde{T}_s = T_s - \xi_{\K_s \#} T_s \in \mathbf{I}_m(\R^{n+1})$ such that by~\cite[4.1.15]{Fed69} $\partial \widetilde{T}_s = 0$ in $\R^{n+1}$ and $\Mass(\widetilde{T}_s) \leq 2 \,\Mass(T)$.  By the Federer-Fleming compactness theorem there exists a sequence $s_j \downarrow 0$ and $\widetilde{T} \in \mathbf{I}_m(\R^{n+1})$ such that $\widetilde{T}_{s_j} \rightarrow \widetilde{T}$ flat norm topology of $\R^{n+1}$.  Clearly $\widetilde{T} = T$ in $\R^{n+1} \setminus \K$, $\partial \widetilde{T} = 0$ in $\R^{n+1}$, and by the semi-continuity of mass $\Mass(\widetilde{T}) \leq 2 \,\Mass(T)$.  Now take $R \in \mathbf{I}_{m+1}(\R^{n+1})$ to be the area minimizing current with $\partial R = \widetilde{T}$ in $\R^{n+1}$.  Hence $\partial R = \widetilde{T} = T$ in $\R^{n+1} \setminus \K$.  By the isoperimetric inequality in $\R^{n+1}$ 
\begin{equation*}
	\frac{\Mass(T)^{\frac{m+1}{m}}}{\Mass(R)} 
	\geq 2^{-\frac{m+1}{m}} \,\frac{\Mass(\widetilde{T})^{\frac{m+1}{m}}}{\Mass(R)} 
	\geq 2^{-\frac{m+1}{m}} \,\frac{\mathcal{H}^m(\sphere)^{\frac{m+1}{m}}}{\mathcal{H}^{m+1}(\ball)}. \qedhere
\end{equation*}
In particular, \eqref{isoper lower bound} holds true for some constant $C_0 = C_0(m) \in (0,\infty)$.  Fix this constant $C_0$.  

\noindent \textit{Step 2.}  Consider the general case of $T \in \mathbf{I}_m(\R^{n+1} \setminus \K)$ with $\partial T = 0$ in $\R^{n+1} \setminus \K$.  Since $T$ has finite mass and zero boundary in $\R^{n+1} \setminus \K$, for every integer $j \geq 1$ there exists $\rho_j \in [2^j,\infty)$ such that 
\begin{equation*}
	\Mass(T \llcorner \R^{n+1} \setminus B_{\rho_j}(0)) < 1/j
\end{equation*}
and thus by slicing theory using \eqref{meas_good_slices} with $\vartheta = 1/2$ there exists $\rho^*_j \in (\rho_j,\rho_j+1)$ such that $T_j = T \llcorner B_{\rho^*_j}(0) \in \mathbf{I}_m(\R^{n+1} \setminus \K)$ with compact support, $\partial T_j = \langle T, |\cdot|, \rho^*_j \rangle$ in $\R^{n+1} \setminus \K$, and 
\begin{gather*}
	\Mass(T_j - T) = \Mass(T \llcorner \R^{n+1} \setminus B_{\rho^*_j}(0)) < 1/j, \\
	\Mass_{\R^{n+1} \setminus \K}(\langle T, |\cdot|, \rho^*_j \rangle) \leq 2 \,\Mass(T \llcorner \R^{n+1} \setminus B_{\rho_j}(0)) \leq 2/j. 
\end{gather*}
By Step 1, there exists $S_j \in \mathbf{I}_m(\R^{n+1} \setminus \K)$ such that $S_j$ has compact support, $\partial S_j = \partial T_j = \langle T, |\cdot|, \rho^*_j \rangle$ in $\R^{n+1} \setminus \K$, and 
\begin{equation*}
	\Mass(S_j) \leq C(m) \,\Mass_{\R^{n+1} \setminus \K}(\langle T, |\cdot|, \rho^*_j \rangle)^{\frac{m}{m-1}} \leq C(m) \,j^{\frac{-m}{m-1}} \rightarrow 0. 
\end{equation*}
By Step 1, there exists $R_j \in \mathbf{I}_m(\R^{n+1} \setminus \K)$ such that $\partial R_j = T_j - S_j$ in $\R^{n+1} \setminus \K$ and $\Mass(R_j) \leq C_0 \,\Mass(T_j - S_j)^{\frac{m+1}{m}}$.  By the Federer-Fleming compactness theorem and using $\Mass(T_j - T) + \Mass(S_j) \rightarrow 0$, after passing to a subsequence $R_j$ converges to some $R \in \mathbf{I}_{m+1}(\R^{n+1} \setminus \K)$ in the flat norm topology such that $\partial R = T$ in $\R^{n+1} \setminus \K$ and by the semi-continuity of mass $\Mass(R) \leq C_0 \,\Mass(T)^{\frac{m+1}{m}}$.
\end{proof}

\section{Sharp relative isoperimetric inequality in a half-space} \label{sec:halfspace rel iso sec}

In this section we will formally prove the sharp relative isoperimetric inequality in a half-space.  As was discussed in the introduction, this readily follows from a reflection argument and the isoperimetric inequality in $\R^{n+1}$.  However, as we point out below, we have to be careful since $T$ and $R$ might not have rectifiable boundaries along the boundary of the half-space.  Thus we include a detailed proof of the theorem. 

\begin{proposition}[Relative isoperimetric inequality in a half-space] 
Let $1 \leq m \leq n$ be integers.  Let $T \in \mathbf{I}_m(\R^{n+1}_+)$ and $R \in \mathbf{I}_{m+1}(\R^{n+1}_+)$ such that $T = \partial R$ in $\R^{n+1}_+$ and $R$ is relatively area minimizing in $\R^{n+1}_+$.  Then 
\begin{equation*} 
	2^{-\frac{1}{m}} \,\frac{\mathcal{H}^m(\sphere)^{\frac{m+1}{m}}}{\mathcal{H}^{m+1}(\ball)} 
		\leq \frac{\Mass(T)^{\frac{m+1}{m}}}{\Mass(R)}. 
\end{equation*}
with equality if and only if $R$ is a multiplicity one $(m+1)$-dimensional flat half-disk which meets $\{x : x_{n+1} = 0\}$ orthogonally and is bounded by a multiplicity one $m$-dimensional hemisphere $T$ in $\R^{n+1}_+$. 
\end{proposition}

\begin{proof}
First let us consider the special case of $T \in \mathbf{I}_m(\R^{n+1})$ and $R \in \mathbf{I}_{m+1}(\R^{n+1})$ with $\|T\|(\overline{\R^{n+1}_-}) = \|R\|(\overline{\R^{n+1}_-}) = 0$.  Let $\iota : \R^{n+1} \rightarrow \R^{n+1}$ be the reflection map across $\{x : x_{n+1} = 0 \}$ given by $\iota(x',x_{n+1}) = (x',-x_{n+1})$ for all $x = (x',x_{n+1}) \in \R^{n+1}$.  Reflect $T$ and $R$ about $\{ x : x_{n+1} = 0 \}$ to obtain $T - \iota_{\#} T \in \mathbf{I}_m(\R^{n+1})$ and $R - \iota_{\#} R \in \mathbf{I}_{m+1}(\R^{n+1})$.  Since $\|T\|(\overline{\R^{n+1}_-}) = \|R\|(\overline{\R^{n+1}_-}) = 0$, 
\begin{equation*}
	\Mass(T - \iota_{\#} T) = 2\,\Mass(T), \quad \Mass(R - \iota_{\#} R) = 2\,\Mass(R).
\end{equation*}
By~\cite[4.1.15]{Fed69}, since $\Mass(T) + \Mass_{\R^{n+1}}(\partial T) + \Mass(R) + \Mass_{\R^{n+1}}(\partial R) < \infty$ and $\iota(x) = x$ for all $x \in \op{spt}(\partial R - T) \subset \{ x_{n+1} = 0 \}$, $\partial (R - \iota_{\#} R) = T - \iota_{\#} T$ in $\R^{n+1}$.  $R - \iota_{\#} R$ is an area minimizing since for any $Q \in \mathbf{I}_{m+1}(\R^{n+1})$ with $\partial Q = T - \iota_{\#} T$ in $\R^{n+1}$ and $\op{spt}(R - \iota_{\#} R - Q)$ compact, $\Mass(Q \llcorner \R^{n+1}_+) \geq \Mass(R)$ since $R$ is relatively area minimizing in $\R^{n+1}_+$ and $\Mass(Q \llcorner \R^{n+1}_-) \geq \Mass(R)$ by symmetry, hence $\Mass(Q) \geq 2\,\Mass(R) = \Mass(R - \iota_{\#} R)$.  By the isoperimetric inequality in $\R^{n+1}$, 
\begin{equation} \label{isoper halfspace eqn1}
	\frac{\Mass(T)^{\frac{m+1}{m}}}{\Mass(R)} = 2^{-\frac{1}{m}} \,\frac{\Mass(T - \iota_{\#} T)^{\frac{m+1}{m}}}{\Mass(R - \iota_{\#} R)} 
	\geq 2^{-\frac{1}{m}} \,\frac{\mathcal{H}^m(\sphere)^{\frac{m+1}{m}}}{\mathcal{H}^{m+1}(\ball)}. 
\end{equation}
Equality holds true in \eqref{isoper halfspace eqn1} if and only if $R - \iota_{\#} R$ is a multiplicity one $(m+1)$-dimensional flat disk.  Since $R - \iota_{\#} R$ is symmetric about $\{x : x_{n+1} = 0 \}$, $R$ must be a multiplicity one $(m+1)$-dimensional half-disk orthogonal to $\{x : x_{n+1} = 0\}$. 

Now let us consider the general case of $T \in \mathbf{I}_m(\R^{n+1}_+)$ and $R \in \mathbf{I}_{m+1}(\R^{n+1}_+)$ such that $R$ is area minimizing with $T = \partial R$ in $\R^{n+1}_+$.  Notice that $T$ or $R$ might not have rectifiable boundaries with finite mass along $\{x : x_{n+1} = 0 \}$ and thus we cannot directly apply~\cite[4.1.15]{Fed69} to obtain $\partial (R - \iota_{\#} R) = T - \iota_{\#} T$ in $\R^{n+1}$.  Instead, we translate $T$ and $R$ downward slightly, letting \begin{align*}
	T_s &= \eta_{(0,s),1\#} (T \llcorner \{x_{n+1} > s\}) = (\eta_{(0,s),1\#} T) \llcorner \R^{n+1}_+, \\
	R_s &= \eta_{(0,s),1\#} (R \llcorner \{x_{n+1} > s\}) = (\eta_{(0,s),1\#} R) \llcorner \R^{n+1}_+ 
\end{align*}
for $s > 0$, where $\eta_{(0,s),1}(x) = x - (0,s)$.  By slicing theory, $T_s \in \mathbf{I}_m(\R^{n+1})$ and $R_s \in \mathbf{I}_{m+1}(\R^{n+1})$ for $\mathcal{L}^1$-a.e.~$s > 0$.  Hence $T_s - \iota_{\#} T_s \in \mathbf{I}_m(\R^{n+1})$ and $R_s - \iota_{\#} R_s \in \mathbf{I}_{m+1}(\R^{n+1})$ with $\partial (R_s - \iota_{\#} R_s) = T_s - \iota_{\#} T_s$ in $\R^{n+1}$.  Clearly 
\begin{equation} \label{isoper halfspace eqn2}
	\Mass(T_s - \iota_{\#} T_s) \leq 2 \,\Mass(T_s) \leq 2\,\Mass(T), \quad 
	\Mass(R_s - \iota_{\#} R_s) \leq 2 \,\Mass(R_s) \leq 2\,\Mass(R). 
\end{equation}
By the Federer-Fleming compactness theorem using \eqref{isoper halfspace eqn2}, there exists $s_j \downarrow 0$, $\widetilde{T} \in \mathbf{I}_m(\R^{n+1})$, and $\widetilde{R} \in \mathbf{I}_{m+1}(\R^{n+1})$ such that $T_{s_j} - \iota_{\#} T_{s_j} \rightarrow \widetilde{T}$ and $R_{s_j} - \iota_{\#} R_{s_j} \rightarrow \widetilde{R}$ in the flat norm topology locally in $\R^{n+1}$.  Clearly  
\begin{equation} \label{isoper halfspace eqn3}
	\partial \widetilde{R} = \widetilde{T}, \quad \iota_{\#} \widetilde{T} = -\widetilde{T}, \quad \iota_{\#} \widetilde{R} = -\widetilde{R}, \quad 
	\widetilde{T} \llcorner \R^{n+1}_+ = T, \quad \widetilde{R} \llcorner \R^{n+1}_+ = R 
\end{equation}
in $\R^{n+1}$.  By the semi-continuity of mass $\Mass(\widetilde{T}) \leq 2 \,\Mass(T)$.  But by \eqref{isoper halfspace eqn3} we have $\widetilde{T} \llcorner \R^{n+1}_+ = T$ and $\widetilde{T} \llcorner \R^{n+1}_- = -\iota_{\#} T$ in $\R^{n+1}$ implying $\Mass(\widetilde{T} \llcorner \{x_{n+1} \neq 0\}) = 2 \,\Mass(T)$.  Hence $\Mass(\widetilde{T} \llcorner \{x_{n+1} = 0\}) = 0$, $\Mass(\widetilde{T}) = 2 \,\Mass(T)$, and $\widetilde{T} = T - \iota_{\#} T$ in $\R^{n+1}$.  Therefore, $T - \iota_{\#} T = \widetilde{T} \in \mathbf{I}_m(\R^{n+1})$ with $\Mass(\widetilde{T}) = 2 \,\Mass(T)$.  By the exact same argument, $R - \iota_{\#} R = \widetilde{R} \in \mathbf{I}_m(\R^{n+1})$ with $\Mass(\widetilde{R}) = 2 \,\Mass(R)$.  By \eqref{isoper halfspace eqn3}, $\partial (R - \iota_{\#} R) = T - \iota_{\#} T$ in $\R^{n+1}$.  Arguing as above, $R - \iota_{\#} R$ is area minimizing in $\R^{n+1}$ and $T$ and $R$ satisfy \eqref{isoper halfspace eqn1} with equality if and only if $R$ is a multiplicity one $(m+1)$-dimensional half-disk orthogonal to $\{x : x_{n+1} = 0\}$. 
\end{proof}

\section{Mass estimates and rectifiability at the boundary} \label{sec:bdry_rect_sec} 

Suppose $\K$ is a proper convex subset of $\R^{n+1}$.  Suppose $T \in \mathbf{I}_m(\R^{n+1} \setminus \K)$ and $R \in \mathbf{I}_{m+1}(\R^{n+1} \setminus \K)$ such that $\partial R = T$ in $\R^{n+1} \setminus \K$, $R$ is relatively area minimizing in $\R^{n+1} \setminus \K$, and the relative isoperimetric ratio of $(T,R)$ is close to $\gamma_{m,n}(\K)$.  In Lemma \ref{boundary rectifiability lemma2} below we obtain estimates telling us that $T$ and $R$ cannot concentrate near $\partial \K$.  Moreover, in the case that $(T,R)$ are in fact relative isoperimetric minimizing, we show in Corollary \ref{boundary rectifiability cor} that $\partial T$ and $\partial R$ are rectifiable along $\partial \K$ and obtain mass estimates on $\partial T$ and $\partial R$ on $\partial \K$.  In Lemma \ref{boundary rectifiability lemma4}, we obtain local mass estimates for $T$ and $R$ near $\partial \K$ and for $\partial T$ and $\partial R$ on $\partial \K$. 

First we will recall~\cite[Theorem 3.1 and Corollary 3.2]{Gru85} which under certain hypotheses established rectifiability of $\partial R$ along $\partial \K$ for a relative area minimizing current $R$.  Since~\cite{Gru85} states its results in the case that $\partial \K$ is any closed $C^2$-hypersurface and $T$ is a smooth submanifold meeting $\partial \K$ transversally, we will state and prove the result as it applies to our setting.  Note that~\cite[Theorem 3.1 and Corollary 3.2]{Gru85} applies with only obvious changes when $\K \subset \R^{n+1}$ is the closure of a bounded open set with $C^2$-boundary and $T \in \mathbf{I}_m(\R^{n+1} \setminus \K)$ satisfies \eqref{bdry rect1 hyp} (as opposed to $T$ smooth).

\begin{lemma} \label{boundary rectifiability lemma1}
Let $\K$ be a bounded proper convex subset of $\R^{n+1}$.  Let $T \in \mathbf{I}_m(\R^{n+1} \setminus \K)$ and $R \in \mathbf{I}_{m+1}(\R^{n+1} \setminus \K)$ such that $\partial R = T$ in $\R^{n+1} \setminus \K$ and $R$ is relatively area minimizing in $\R^{n+1} \setminus \K$.  Assume for some constants $\Gamma \in (0,\infty)$ and $0 < s_0 < s_1 < \infty$ 
\begin{equation} \label{bdry rect1 hyp}
	\frac{\Mass(T \llcorner \{d_{\K} < s\})}{s} \leq \Gamma < \infty
\end{equation}
for all $s \in [s_0,s_1]$.  Then 
\begin{equation} \label{bdry rect1 concl1}
	\frac{\Mass(R \llcorner \{d_{\K} < s\})}{s} + \Gamma \,s \leq \frac{\Mass(R \llcorner \{d_{\K} < t\})}{t} + \Gamma \,t
\end{equation}
for all $s_0 \leq s < t \leq s_1$. 

If additionally \eqref{bdry rect1 hyp} holds true for all $s \in (0,s_1]$, then $T \in \mathbf{I}_m(\R^{n+1})$ and $R \in \mathbf{I}_{m+1}(\R^{n+1})$ and in particular $\partial T$ and $\partial R$ are rectifiable along $\partial \K$ (as currents of $\R^{n+1}$) with $\Mass(\partial T \llcorner \partial \K) \leq 4 \,\Gamma < \infty$ and $\Mass(\partial R \llcorner \partial \K) < \infty$. 
\end{lemma}

\begin{proof}
For $\mathcal{L}^1$-a.e.~$s \in (s_0,s_1)$, compare $R$ with $R_s \in \mathbf{I}_{m+1}(\R^{n+1} \setminus \K)$ given by 
\begin{equation} \label{bdry rect1 eqn1}
	R_s = R \llcorner \{d_{\K} > s\} + h_{\#}(\llbracket 0,1 \rrbracket \times \langle R, d_{\K}, s \rangle) 
		+ h_{\#}(\llbracket 0,1 \rrbracket \times (T \llcorner \{d_{\K} < s\})) , 
\end{equation}
where $h : [0,1] \times (\R^{n+1} \setminus \K) \rightarrow \R^{n+1}$ is given by $h(t,x) = (1-t) \,\xi_{\K}(x) + t \,x$ for $t \in [0,1]$ and $x \in \R^{n+1} \setminus \K$.  (Recall from Subsection \ref{sec:prelims_notation} that $\llbracket 0,1 \rrbracket$ is the one-dimensional integral current associated with the interval $(0,1)$.)  To compute the boundary of $R_s$, using the homotopy formula, slicing theory, and $\xi_{\K}(\R^{n+1} \setminus \K) \subseteq \partial \K$ we obtain 
\begin{align} \label{bdry rect1 eqn2}
	&\partial (R \llcorner \{d_{\K} > s\}) = T \llcorner \{d_{\K} > s\} - \langle R, d_{\K}, s \rangle, \\
	&\partial h_{\#}(\llbracket 0,1 \rrbracket \times \langle R, d_{\K}, s \rangle) 
		= \langle R, d_{\K}, s \rangle + h_{\#}(\llbracket 0,1 \rrbracket \times \langle T, d_{\K}, s \rangle), \nonumber \\
	&\partial h_{\#}(\llbracket 0,1 \rrbracket \times (T \llcorner \{d_{\K} < s\})) 
		= T \llcorner \{d_{\K} < s\} - h_{\#}(\llbracket 0,1 \rrbracket \times \langle T, d_{\K}, s \rangle) \nonumber 
\end{align}
in $\R^{n+1} \setminus \K$  for $\mathcal{L}^1$-a.e.~$s > 0$.  Thus by \eqref{bdry rect1 eqn1} and summing terms in \eqref{bdry rect1 eqn2}, $\partial R_s = T$ in $\R^{n+1} \setminus \K$ for $\mathcal{L}^1$-a.e.~$s > 0$.  Since $\op{Lip} \xi_{\K} = 1$, $\op{Lip} h(t,\cdot) \leq 1$ for all $t \in [0,1]$.  Thus by \eqref{bdry rect1 eqn1} and \eqref{bdry rect1 hyp},  
\begin{align*}
	\Mass(R_s) &\leq \Mass(R \llcorner \{d_{\K} > s\}) 
		+ \Mass(h_{\#}(\llbracket 0,1 \rrbracket \times \langle R, d_{\K}, s \rangle)) 
		+ \Mass(h_{\#}(\llbracket 0,1 \rrbracket \times (T \llcorner \{d_{\K} < s\}))) 
	\\&\leq \Mass(R \llcorner \{d_{\K} > s\}) + s \,\Mass(\langle R, d_{\K}, s \rangle) 
		+ s\,\Mass(T \llcorner \{d_{\K} < s\}) 
	\\&\leq \Mass(R \llcorner \{d_{\K} > s\}) + s \,\Mass(\langle R, d_{\K}, s \rangle) + \Gamma \,s^2 
\end{align*}
for $\mathcal{L}^1$-a.e.~$s \in (s_0,s_1)$.  Since $R$ is relatively area minimizing in $\R^{n+1} \setminus \K$, $\Mass(R) \leq \Mass(R_s)$ and thus 
\begin{equation*}
	\Mass(R \llcorner \{d_{\K} < s\}) \leq s \,\Mass(\langle R, d_{\K}, s \rangle) + \Gamma \,s^2 
\end{equation*}
for $\mathcal{L}^1$-a.e.~$s \in (s_0,s_1)$.  By slicing theory $\frac{d}{ds} \,\Mass(R \llcorner \{d_{\K} < s\}) = \Mass(\langle R, d_{\K}, s \rangle)$ for $\mathcal{L}^1$-a.e.~$s \in (s_0,s_1)$, so 
\begin{equation*} 
	\Mass(R \llcorner \{d_{\K} < s\}) \leq s \,\frac{d}{ds} \,\Mass(R \llcorner \{d_{\K} < s\}) + \Gamma \,s^2
\end{equation*}
for $\mathcal{L}^1$-a.e.~$s \in (s_0,s_1)$.  Equivalently, 
\begin{equation} \label{bdry rect1 eqn3}
	\frac{d}{ds} \left( \frac{\Mass(R \llcorner \{d_{\K} < s\})}{s} + \Gamma s \right) \geq 0 
\end{equation}
for $\mathcal{L}^1$-a.e.~$s \in (s_0,s_1)$.  Since $\Mass(R \llcorner \{d_{\K} < s\})$ is an increasing function of $s$, we can integrate \eqref{bdry rect1 eqn3} over $[s,t]$ to obtain \eqref{bdry rect1 concl1}. 

Let $k$ be any positive integer.  By slicing theory using \eqref{meas_good_slices} with $\vartheta = 1/4$, \eqref{bdry rect1 hyp}, and \eqref{bdry rect1 concl1} for every positive integer $k$ there exists $s^*_k \in (0,2^{-k})$ such that $T \llcorner \{d_{\K} > s^*_k\} \in \mathbf{I}_m(\R^{n+1})$, $\langle T, d_{\K}, s^*_k \rangle \in \mathbf{I}_{m-1}(\R^{n+1})$, $R \llcorner \{d_{\K} > s^*_k\} \in \mathbf{I}_{m+1}(\R^{n+1})$, and $\langle R, d_{\K}, s^*_k \rangle \in \mathbf{I}_m(\R^{n+1})$ with 
\begin{equation} \label{bdry rect1 eqn4}
	\partial (T \llcorner \{d_{\K} > s^*_k\}) = -\langle T, d_{\K}, s^*_k \rangle , \quad 
	\partial (R \llcorner \{d_{\K} > s^*_k\}) = T \llcorner \{d_{\K}  > s^*_k\} - \langle R, d_{\K}, s^*_k \rangle 
\end{equation}
in $\R^{n+1}$ and 
\begin{align} \label{bdry rect1 eqn5}
	&\Mass(\langle T, d_{\K}, s^*_k \rangle) \leq \frac{4 \,\Mass(T \llcorner \{d_{\K} < 2^{-k}\})}{2^{-k}} \leq 4 \,\Gamma < \infty, \\
	&\Mass(\langle R, d_{\K}, s^*_k \rangle) \leq \frac{4 \,\Mass(R \llcorner \{d_{\K} < 2^{-k}\})}{2^{-k}} 
		\leq \frac{4\,\Mass(R)}{s_1} + 4\,\Gamma \,s_1 < \infty. \nonumber
\end{align}
$T \llcorner \{d_{\K} > s^*_k\} \rightarrow T$ and $R \llcorner \{d_{\K} > s^*_k\} \rightarrow R$ in the mass norm topology on compact subsets of $\R^{n+1}$.  But by the Federer-Fleming compactness theorem, \eqref{bdry rect1 eqn4}, and \eqref{bdry rect1 eqn5}, after passing to a subsequence $T \llcorner \{d_{\K} > s^*_k\}$ and $R \llcorner \{d_{\K} > s^*_k\}$ converge locally in the flat norm topology to some integral currents in $\R^{n+1}$ as $k \rightarrow \infty$ and thus $T \in \mathbf{I}_m(\R^{n+1})$ and $R \in \mathbf{I}_{m+1}(\R^{n+1})$.  Moreover, there exists $S \in \mathbf{I}_{m-1}(\R^{n+1})$ and $Z \in \mathbf{I}_m(\R^{n+1})$ such that after passing to a subsequence $\langle T, d_{\K}, s^*_k \rangle \rightarrow -S$ and $\langle R, d_{\K}, s^*_k \rangle \rightarrow -Z$ in the flat norm topology on compact subsets of $\R^{n+1}$ and $\partial T = S$ and $\partial R = T + Z$ in $\R^{n+1}$.  In particular, $S = \partial T \llcorner \partial \K$ and $Z = \partial R \llcorner \partial \K$ are integer-multiplicity rectifiable currents with $\Mass(S) \leq 4 \,\Gamma < \infty$ and $\Mass(Z) < \infty$. 
\end{proof}

Next we will prove an almost minimizing property for $T$ in the case that the isoperimetric ratio of $(T,R)$ is close to minimal.  Our aim is to construct competitors for $(T,R)$ of the form $(T + \partial X, R+X)$ and then use the fact that 
\begin{equation} \label{aim why linearize}
	\frac{\Mass(T)}{\Mass(R)^{\frac{m}{m+1}}} \lessapprox \frac{\Mass_{\R^{n+1} \setminus \K}(T + \partial X)}{(\Mass(R) - \Mass(X))^{\frac{m}{m+1}}}
\end{equation}
to obtain mass estimates on $T$.  Since the relative isoperimetric ratio is nonlinear, we will derive an almost minimizing property of $T$ that is a linearization of \eqref{aim why linearize}.  Note that the almost minimizing property will contain the term $H_0$ as in \eqref{H0 defn} in Theorem C. 

\begin{lemma} \label{almost minimizing lemma}
Let $\K$ be a closed subset of $\R^{n+1}$ with $\K \neq \R^{n+1}$.  Let $T \in \mathbf{I}_m(\R^{n+1} \setminus \K)$ and $R \in \mathbf{I}_{m+1}(\R^{n+1} \setminus \K)$ such that $\partial R = T$ in $\R^{n+1} \setminus \K$ and $R$ is relatively area minimizing in $\R^{n+1} \setminus \K$.  Suppose for $\varepsilon > 0$ that 
\begin{equation} \label{aim hyp}
	\frac{\Mass(T)^{\frac{m+1}{m}}}{\Mass(R)} \leq (1+\varepsilon) \,\gamma_{m,n}(\K). 
\end{equation}
Then 
\begin{equation*} 
	(1-\varepsilon) \,\Mass(T) \leq \Mass_{R^{n+1} \setminus \K}(T+\partial X) + 2H_0 \,\Mass(X) 
\end{equation*}
for all $X \in \mathbf{I}_{m+1}(\R^{n+1} \setminus \K)$ with compact support and with $\Mass(X) \leq \tfrac{1}{2} \,\Mass(R)$, where $H_0$ is as in \eqref{H0 defn}.
\end{lemma}

\begin{proof}
Let $X \in \mathbf{I}_{m+1}(\R^{n+1} \setminus \K)$ with compact support.  Since $R$ is relatively area minimizing in $\R^{n+1} \setminus \K$ and $\partial (R+X) = T + \partial X$ in $\R^{n+1} \setminus \K$, the relatively area minimizing current $\widetilde{R} \in \mathbf{I}_{m+1}(\R^{n+1} \setminus \K)$ with $\partial \widetilde{R} = T + \partial X$ in $\R^{n+1} \setminus \K$ satisfies 
\begin{equation*}
	\Mass(\widetilde{R}) \geq \Mass(R) - \Mass(X). 
\end{equation*}
Thus by \eqref{aim hyp}, 
\begin{equation*}
	\frac{\Mass(T)^{\frac{m+1}{m}}}{\Mass(R)} 
		\leq (1+\varepsilon) \,\frac{\Mass_{\R^{n+1} \setminus \K}(T + \partial X)^{\frac{m+1}{m}}}{\Mass(\widetilde{R})}
		\leq (1+\varepsilon) \,\frac{\Mass_{\R^{n+1} \setminus \K}(T + \partial X)^{\frac{m+1}{m}}}{\Mass(R) - \Mass(X)} . 
\end{equation*}
Taking the $\tfrac{m}{m+1}$ power of both sides, 
\begin{equation*}
	\frac{\Mass(T)}{\Mass(R)^{\frac{m}{m+1}}} \leq (1+\varepsilon)^{\frac{m}{m+1}} 
		\,\frac{\Mass_{\R^{n+1} \setminus \K}(T + \partial X)}{(\Mass(R) - \Mass(X))^{\frac{m}{m+1}}}. 
\end{equation*}
Rearranging terms, 
\begin{equation*}
	(1+\varepsilon)^{-\frac{m}{m+1}} \left( 1 - \frac{\Mass(X)}{\Mass(R)} \right)^{\frac{m}{m+1}} \leq \frac{\Mass_{\R^{n+1} \setminus \K}(T + \partial X)}{\Mass(T)}. 
\end{equation*}
By Taylor's theorem 
\begin{equation*}
	1 - \frac{2m}{m+1} \,x \leq 1 - \frac{m}{m+1} \,x \cdot 2^{\frac{1}{m+1}} 
		\leq 1 - \int_0^1 \frac{m}{m+1} \,x \,(1-tx)^{\frac{-1}{m+1}} \,dt = (1-x)^{\frac{m}{m+1}}
\end{equation*} 
for all $x \in [0,1/2]$, which with $x = \Mass(X)/\Mass(R)$ gives us 
\begin{equation*}
	(1+\varepsilon)^{-\frac{m}{m+1}} \left( 1 - \frac{2m}{m+1} \,\frac{\Mass(X)}{\Mass(R)} \right) 
		\leq \frac{\Mass_{\R^{n+1} \setminus \K}(T + \partial X)}{\Mass(T)}
\end{equation*}
provided $\Mass(X) \leq \tfrac{1}{2} \,\Mass(R)$.  Since $(1-\varepsilon) \leq (1+\varepsilon)^{-\frac{m}{m+1}} \leq 1$, 
\begin{equation*}
	1-\varepsilon - \frac{2m}{m+1} \,\frac{\Mass(X)}{\Mass(R)} \leq \frac{\Mass_{\R^{n+1} \setminus \K}(T + \partial X)}{\Mass(T)}
\end{equation*}
By multiplying by $\Mass(T)$ and rearranging terms, 
\begin{equation*}
	(1-\varepsilon) \,\Mass(T) \leq \Mass_{\R^{n+1} \setminus \K}(T + \partial X) + \frac{2m}{m+1} \,\frac{\Mass(T)}{\Mass(R)} \,\Mass(X) 
		= \Mass_{\R^{n+1} \setminus \K}(T + \partial X) + 2H_0 \,\Mass(X), 
\end{equation*}
where the last step follows from \eqref{H0 defn}. 
\end{proof}

Now we come to our main mass estimates for $T$ and $R$ near $\partial \K$.  The rectifiability of $\partial T$ and $\partial R$ along $\partial \K$ in the case $(T,R)$ is isoperimetric minimizing will follow. 

\begin{lemma} \label{boundary rectifiability lemma2}
There exists a constant $\theta = \theta(m) \in (0,1]$ such that the following holds true.  Let $\varepsilon \in (0,\theta)$.  Let $\K$ be a bounded proper convex subset of $\R^{n+1}$.  Let $T \in \mathbf{I}_m(\R^{n+1} \setminus \K)$ and $R \in \mathbf{I}_{m+1}(\R^{n+1} \setminus \K)$ such that $R$ is relatively area minimizing in $\R^{n+1} \setminus \K$ with $\partial R = T$ in $\R^{n+1} \setminus \K$ and 
\begin{equation} \label{bdry rect2 hyp}
	\frac{\Mass(T)^{\frac{m+1}{m}}}{\Mass(R)} \leq (1+\varepsilon) \,\gamma_{m,n}(\K). 
\end{equation}
Then 
\begin{align} 
	\label{bdry rect2 concl1} &\frac{\Mass(T \llcorner \{d_{\K} < s\})}{s} \leq C(m) \,\Mass(T)^{\frac{m-1}{m}} , \\
	\label{bdry rect2 concl2} &\frac{\Mass(R \llcorner \{d_{\K} < s\})}{s} \leq C(m) \,\Mass(T) 
\end{align}
for all $\varepsilon \,\Mass(R)^{\frac{1}{m+1}} \leq s \leq \theta \,\Mass(R)^{\frac{1}{m+1}}$. 
\end{lemma}

\begin{proof}
We will apply Lemma \ref{almost minimizing lemma} with 
\begin{equation} \label{bdry rect2 eqn1}
	X = X_s = -h_{\#}(\llbracket 0,1 \rrbracket \times (T \llcorner \{d_{\K} < s\})) 
\end{equation}
in $\R^{n+1} \setminus \K$ for $\mathcal{L}^1$-a.e.~$s > 0$, where $h(t,x) = (1-t) \,\xi_{\K}(x) + t \,x$ as in the proof of Lemma \ref{boundary rectifiability lemma1}.  By the homotopy formula and slicing theory, 
\begin{align} \label{bdry rect2 eqn2}
	T + \partial X_s &= T - T \llcorner \{d_{\K} < s\} + h_{\#}(\llbracket 0,1 \rrbracket \times \langle T, d_{\K}, s \rangle) 
		\\&= T \llcorner \{d_{\K} > s\} + h_{\#}(\llbracket 0,1 \rrbracket \times \langle T, d_{\K}, s \rangle) \nonumber 
\end{align}
in $\R^{n+1} \setminus \K$ for $\mathcal{L}^1$-a.e.~$s > 0$.  Thus, recalling from the proof of Lemma \ref{boundary rectifiability lemma1} that $\op{Lip} h(t,\cdot) \leq 1$ for all $t \in [0,1]$, the mass of $T+\partial X_s$ from \eqref{bdry rect2 eqn2} is bounded by 
\begin{align} \label{bdry rect2 eqn3}
	\Mass_{\R^{n+1} \setminus \K}(T + \partial X_s) 
		&\leq \Mass(T \llcorner \{d_{\K} > s\}) + \Mass(h_{\#}(\llbracket 0,1 \rrbracket \times \langle T, d_{\K}, s \rangle)) 
		\\&\leq \Mass(T \llcorner \{d_{\K} > s\}) + s \,\Mass(\langle T, d_{\K}, s \rangle) \nonumber 
\end{align}
and the mass of $X_s$ from \eqref{bdry rect2 eqn1} is bounded by 
\begin{equation} \label{bdry rect2 eqn4}
	\Mass(X_s) = \Mass(h_{\#}(\llbracket 0,1 \rrbracket \times (T \llcorner \{d_{\K} < s\}))) \leq s \,\Mass(T \llcorner \{d_{\K} < s\}) 
\end{equation}
for $\mathcal{L}^1$-a.e.~$s > 0$.  When applying Lemma \ref{almost minimizing lemma}, we will want to assume $\Mass(X_s) \leq s \,\Mass(T) \leq \tfrac{1}{2} \,\Mass(R)$, which is implied by $s < \tfrac{1}{4H_0}$.  Hence by Lemma \ref{almost minimizing lemma}, \eqref{bdry rect2 hyp}, \eqref{bdry rect2 eqn3}, and \eqref{bdry rect2 eqn4},  
\begin{equation*} 
	\Mass(T \llcorner \{d_{\K} < s\}) \leq s \,\Mass(\langle T, d_{\K}, s \rangle) + 2 H_0 s \,\Mass(T \llcorner \{d_{\K} < s\}) + \varepsilon \,\Mass(T) 
\end{equation*}
for $\mathcal{L}^1$-a.e.~$0 < s < 1/(4H_0)$.  By slicing theory $\frac{d}{ds} \,\Mass(T \llcorner \{d_{\K} < s\}) = \Mass(\langle T, d_{\K}, s \rangle)$ for $\mathcal{L}^1$-a.e.~$s \in (s_0,s_1)$, so 
\begin{equation} \label{bdry rect2 eqn5}
	\Mass(T \llcorner \{d_{\K} < s\}) \leq s \,\frac{d}{ds} \Mass(T \llcorner \{d_{\K} < s\})
		+ 2 H_0 s \,\Mass(T \llcorner \{d_{\K} < s\}) + \varepsilon \,\Mass(T) 
\end{equation}
for $\mathcal{L}^1$-a.e.~$0 < s < 1/(4H_0)$.  By subtracting $2 H_0 s \,\Mass(T \llcorner \{d_{\K} < s\})$ from both sides and dividing by $1 - 2 H_0 s$,  
\begin{align*}
	\Mass(T \llcorner \{d_{\K} < s\}) &\leq \frac{s}{1 - 2 H_0 s} \,\frac{d}{ds} \Mass(T \llcorner \{d_{\K} < s\}) + \frac{\varepsilon}{1 - 2 H_0 s} \,\Mass(T) 
	\\&\leq s \,(1 + 4 H_0 s) \,\frac{d}{ds} \Mass(T \llcorner \{d_{\K} < s\}) + 2 \,\varepsilon \,\Mass(T) 
\end{align*}
for $\mathcal{L}^1$-a.e.~$0 < s < 1/(4H_0)$, where in the last step we used $1/(1-2H_0s) \leq 1+4H_0s \leq 2$ for all $0 < s < 1/(4H_0)$.  Hence 
\begin{equation} \label{bdry rect2 eqn6}
	0 \leq \frac{d}{ds} \left( (1 + 4 H_0 s) \,\frac{\Mass(T \llcorner \{d_{\K} < s\})}{s} - 2 \,\varepsilon \,\frac{\Mass(T)}{s} \right) 
\end{equation}
for $\mathcal{L}^1$-a.e.~$0 < s < 1/(4H_0)$.  Noting that $\Mass(T \llcorner \{d_{\K} < s\})$ is increasing and integrating \eqref{bdry rect2 eqn6} over $[s,t]$, 
\begin{equation} \label{bdry rect2 eqn7}
	(1 + 4 H_0 s) \,\frac{\Mass(T \llcorner \{d_{\K} < s\})}{s} - 2 \,\varepsilon \,\frac{\Mass(T)}{s} 
	\leq (1 + 4 H_0 t) \,\frac{\Mass(T \llcorner \{d_{\K} < t\})}{t} - 2 \,\varepsilon \,\frac{\Mass(T)}{t}
\end{equation}
for all $0 < s < t \leq 1/(4H_0)$. 

By Lemma \ref{isoper lower bound lemma}, $\gamma_{m,n}(\K) \geq c_1(m) > 0$ for some constant $c_1(m) \in (0,\infty)$, and by the isoperimetric inequality in $\R^{n+1}$ applied to currents disjoint from $\K$, $\gamma_{m,n}(\K) \leq C_1(m)$ for some constant $C_1(m) \in (0,\infty)$.  Thus by \eqref{bdry rect2 hyp}, 
\begin{equation*}
	c_1(m) \leq \gamma_{m,n}(\K) \leq \frac{\Mass(T)^{\frac{m+1}{m}}}{\Mass(R)} \leq 2\,\gamma_{m,n}(\K) \leq 2 \,C_1(m). 
\end{equation*}
Equivalently, $\Mass(T)$ and $\Mass(R)$ are related by 
\begin{equation} \label{bdry rect2 eqn8}
	c_2(m) \,\Mass(R)^{\frac{1}{m+1}} \leq \Mass(T)^{\frac{1}{m}} \leq C_2(m) \,\Mass(R)^{\frac{1}{m+1}} 
\end{equation}
for some constants $0 < c_2(m) < C_2(m) < \infty$.  By \eqref{H0 defn} and \eqref{bdry rect2 eqn8},
\begin{equation} \label{bdry rect2 eqn9}
	c_3(m) \,\Mass(R)^{\frac{-1}{m+1}} \leq H_0 \leq C_3(m) \,\Mass(R)^{\frac{-1}{m+1}} 
\end{equation}
for some constants $0 < c_3(m) < C_3(m) < \infty$.  Set $\theta = \theta(m) = \min\{1/(4 \,C_3(m)),1\}$ for $C_3(m)$ as in \eqref{bdry rect2 eqn9} so that \eqref{bdry rect2 eqn7} holds true for all $0 < s < t \leq \theta \,\Mass(R)^{\frac{1}{m+1}}$. 

Let $\varepsilon \,\Mass(R)^{\frac{1}{m+1}} \leq s \leq \theta \,\Mass(R)^{\frac{1}{m+1}}$.  If $\Mass(T \llcorner \{d_{\K} < s\}) \leq 4\,\varepsilon \,\Mass(T)$, then using \eqref{bdry rect2 eqn8} 
\begin{equation*}
	\Mass(T \llcorner \{d_{\K} < s\}) \leq 4\,\varepsilon \,\Mass(T) \leq C(m) \, s \,\Mass(T)^{\frac{m-1}{m}} 
\end{equation*}
and thus \eqref{bdry rect2 concl1} holds true.  If instead $\Mass(T \llcorner \{d_{\K} < s\}) > 4\,\varepsilon \,\Mass(T)$, by \eqref{bdry rect2 eqn7} 
\begin{align*} 
	\frac{1}{2} \,\frac{\Mass(T \llcorner \{d_{\K} < s\})}{s} 
	&< \frac{\Mass(T \llcorner \{d_{\K} < s\})}{s} - 2 \,\varepsilon \,\frac{\Mass(T)}{s} 
	\\&\leq (1 + 4 H_0 t) \,\frac{\Mass(T \llcorner \{d_{\K} < t\})}{t} - 2 \,\varepsilon \,\frac{\Mass(T)}{t}
	\leq 2 \,\frac{\Mass(T \llcorner \{d_{\K} < t\})}{t} \nonumber
\end{align*}
for all $s \leq t \leq \theta \,\Mass(R)^{\frac{1}{m+1}}$.  Setting $t = \theta \,\Mass(R)^{\frac{1}{m+1}}$ and using \eqref{bdry rect2 eqn8} gives us \eqref{bdry rect2 concl1}.

By Lemma \ref{boundary rectifiability lemma1} and \eqref{bdry rect2 concl1}, 
\begin{equation*}
	\frac{\Mass(R \llcorner \{d_{\K} < s\})}{s} + C(m) \,s \,\Mass(T)^{\frac{m-1}{m}} 
		\leq \frac{\Mass(R \llcorner \{d_{\K} < t\})}{t} + C(m) \,t \,\Mass(T)^{\frac{m-1}{m}} 
\end{equation*}
for all $\varepsilon \,\Mass(R)^{\frac{1}{m+1}} \leq s < t \leq \theta \,\Mass(R)^{\frac{1}{m+1}}$.  Setting $t = \theta \,\Mass(R)^{\frac{1}{m+1}}$ and using \eqref{bdry rect2 eqn8}, we obtain \eqref{bdry rect2 concl2}. 
\end{proof}

\begin{corollary} \label{boundary rectifiability cor}
Let $\K$ be a bounded proper convex subset of $\R^{n+1}$.  Let $T \in \mathbf{I}_m(\R^{n+1} \setminus \K)$ and $R \in \mathbf{I}_{m+1}(\R^{n+1} \setminus \K)$ such that $\partial R = T$ in $\R^{n+1} \setminus \K$, $R$ is relatively area minimizing in $\R^{n+1} \setminus \K$, and $(T,R)$ is a relative isoperimetric minimizer in $\R^{n+1} \setminus \K$.  Then $T \in \mathbf{I}_m(\R^{n+1})$ and $R \in \mathbf{I}_{m+1}(\R^{n+1})$ with
\begin{align} 
	\label{bdry rect3 concl1} &\Mass(\partial T \llcorner \partial \K) \leq C(m) \,\Mass(T)^{\frac{m-1}{m}}, \\
	\label{bdry rect3 concl2} &\Mass(\partial R \llcorner \partial \K) \leq C(m) \,\Mass(T) .  
\end{align}
\end{corollary}

\begin{proof}
Since $(T,R)$ is a relative isoperimetric minimizer, $\varepsilon = 0$ and thus \eqref{bdry rect2 concl1} and \eqref{bdry rect2 concl2} hold true for all $0 < s \leq \theta \,\Mass(R)^{\frac{1}{m+1}}$.  By arguing as we did at the end of the proof of Lemma \ref{boundary rectifiability lemma1} using \eqref{bdry rect2 concl1} and \eqref{bdry rect2 concl2}, we conclude that $T \in \mathbf{I}_m(\R^{n+1})$, $R \in \mathbf{I}_{m+1}(\R^{n+1})$, and \eqref{bdry rect3 concl1} and \eqref{bdry rect3 concl2} hold true. 
\end{proof}

Lemma \ref{boundary rectifiability lemma2} extends to the case where $\K$ is the closure of a bounded open subset with $C^2$-boundary but $\K$ is not convex to give us the following result Corollary \ref{boundary rectifiability remark2}.  Note that Corollary \ref{boundary rectifiability remark2} only gives us useful non-concentration estimates for $T$ and $R$ along $\partial \K$ if $\Mass(R) \leq C(m) \,\kappa_0^{-m-1}$, which is too restrictive to extend the existence result Theorem B to sets $\K$ which are not convex.  

\begin{corollary} \label{boundary rectifiability remark2} 
There exists a constant $\theta = \theta(m) \in (0,\infty)$ such that the following holds true.  Let $\varepsilon \in (0,\theta)$.  Let $\K$ be the closure of a bounded open subset of $\R^{n+1}$ with $C^2$-boundary.  Let $T \in \mathbf{I}_m(\R^{n+1} \setminus \K)$ and $R \in \mathbf{I}_{m+1}(\R^{n+1} \setminus \K)$ such that $\partial R = T$ in $\R^{n+1} \setminus \K$, $R$ is relatively area minimizing in $\R^{n+1} \setminus \K$, and $T$ and $R$ satisfy \eqref{bdry rect2 hyp}.  Then 
\begin{align} 
	\label{bdry rect rmk2 concl1} &\frac{\Mass(T \llcorner \{d_{\K} < s\})}{s} \leq C(m) \,\big( \Mass(T)^{\frac{m-1}{m}} + \kappa_0 \,\Mass(T) \big), \\ 
	\label{bdry rect rmk2 concl2} &\frac{\Mass(R \llcorner \{d_{\K} < s\})}{s} \leq C(m) \,\big( \Mass(T) + \kappa_0 \,\Mass(R) \big)  
\end{align}
for all $\varepsilon \,\Mass(R)^{\frac{1}{m+1}} \leq s \leq \min\{\theta(m) \,\Mass(R)^{\frac{1}{m+1}}, 1/(2\kappa_0)\}$. 

In the special case that $(T,R)$ is a relative isoperimetric minimizer in $\R^{n+1} \setminus \K$ (i.e.~$\varepsilon = 0$), then $T \in \mathbf{I}_m(\R^{n+1})$ and $R \in \mathbf{I}_{m+1}(\R^{n+1})$ with 
\begin{align} 
	\label{bdry rect rmk2 concl3} &\Mass(\partial T \llcorner \partial \K) \leq C(m) \,\big( \Mass(T)^{\frac{m-1}{m}} + \kappa_0 \,\Mass(T) \big), \\
	&\Mass(\partial R \llcorner \partial \K) \leq C(m) \,\big( \Mass(T) + \kappa_0 \,\Mass(R) \big) .  \nonumber 
\end{align}
\end{corollary}

\begin{proof}
We modify the arguments from the proof of Lemma \ref{boundary rectifiability lemma2}.  We replace $\op{Lip} \xi_{\K} = 1$ with 
\begin{equation*} 
	\|\nabla \xi_{\K}(x)\| \leq 1 + 2 \kappa_0 d_{\K}(x)
\end{equation*}
for each $x \in U$, where $d_{\K}(x) = \op{dist}(x,\K)$ and $\kappa_0$ is as in \eqref{kappa0 defn} (see Subsection \ref{sec:prelims_sets}).  In particular, $\op{Lip} h(\cdot,t) \leq 1$ with $|\nabla h(\cdot,t)| \leq 1 + 2\kappa_0 s$ on $\{ d_{\K} \leq s \}$.  As a result, in place of \eqref{bdry rect2 eqn5} we get 
\begin{align} \label{bdry rect rmk2 eqn1}
	\Mass(T \llcorner \{d_{\K} < s\}) \leq{}& s \,(1 + 2 \kappa_0 s)^{m-1} \,\frac{d}{ds} \Mass(T \llcorner \{d_{\K} < s\})
		\\& + 2 H_0 s \,(1 + 2 \kappa_0 s)^m \,\Mass(T \llcorner \{d_{\K} < s\}) + \varepsilon \,\Mass(T) \nonumber  
\end{align}
for $\mathcal{L}^1$-a.e.~$s > 0$ such that $X_s$ in \eqref{bdry rect2 eqn1} satisfies $\Mass(X_s) \leq \tfrac{1}{2} \,\Mass(R)$ and $s < 1/(2\kappa_0)$.  Notice that 
\begin{equation*}
	\Mass(X_s) \leq s \,(1 + 2 \kappa_0 s)^m \,\Mass(T) \leq 2^m s \,\Mass(T) 
\end{equation*} 
so $\Mass(X_s) \leq \tfrac{1}{2} \,\Mass(R)$ is implied by $s < 1/(2^{m+2} H_0)$.  It follows from Taylor's theorem that $(1 + x)^{m-1} \leq 1 + (m-1) \,2^{m-2} x$ for all $x \in [0,1]$, which applied with $x = 2 \kappa_0 s$ in \eqref{bdry rect rmk2 eqn1} gives us 
\begin{align} \label{bdry rect rmk2 eqn2}
	\Mass(T \llcorner \{d_{\K} < s\}) \leq{}& s \,(1 + (m-1) \,2^{m-1} \kappa_0 s) \,\frac{d}{ds} \Mass(T \llcorner \{d_{\K} < s\}) 
		\\& + 2^{m+1} H_0 s \,\Mass(T \llcorner \{d_{\K} < s\}) + \varepsilon \,\Mass(T) \nonumber 
\end{align}
for $\mathcal{L}^1$-a.e.~$0 < s < \min\{1/(2^{m+2} H_0), \,1/(2\kappa_0)\}$.  By subtracting $2^{m+1} H_0 s \,\Mass(T \llcorner \{d_{\K} < s\})$ from both sides of \eqref{bdry rect rmk2 eqn2} and dividing both sides by $1 - 2^{m+1} H_0 s \geq 1/2$, 
\begin{equation*}
	\Mass(T \llcorner \{d_{\K} < s\}) \leq s \,(1 + C \kappa_0 s + C H_0 s) \,\frac{d}{ds} \Mass(T \llcorner \{d_{\K} < s\}) 
		+ 2 \varepsilon \,\Mass(T)
\end{equation*}
for $\mathcal{L}^1$-a.e.~$0 < s < \min\{1/(2^{m+2} H_0), \,1/(2\kappa_0)\}$ and some constant $C = C(m) \in (0,\infty)$, or equivalently 
\begin{equation*}
	0 \leq \frac{d}{ds} \left( (1 + C \kappa_0 s + C H_0 s) \,\frac{\Mass(T \llcorner \{d_{\K} < s\})}{s} - 2 \varepsilon \,\frac{\Mass(T)}{s} \right) 
\end{equation*}
or $\mathcal{L}^1$-a.e.~$0 < s < \min\{1/(2^{m+2} H_0), \,1/(2\kappa_0)\}$.  By integrating over $[s,t]$, 
\begin{align} \label{bdry rect rmk2 eqn3}
	&(1 + C  \kappa_0 s + C H_0 s) \,\frac{\Mass(T \llcorner \{d_{\K} < s\})}{s} - 2 \varepsilon s \,\Mass(T) 
		\\&\hspace{10mm} \leq (1 + C  \kappa_0 t + C H_0 t) \,\frac{\Mass(T \llcorner \{d_{\K} < t\})}{t} - 2 \varepsilon t \,\Mass(T) \nonumber 
\end{align}
for all $0 < s < t \leq \min\{1/(2^{m+2} H_0), \,1/(2\kappa_0)\}$. 

To show \eqref{bdry rect rmk2 concl1}, we argue as we did in the proof of Lemma \ref{boundary rectifiability lemma2} using \eqref{bdry rect rmk2 eqn3}.  In particular, if $\Mass(T \llcorner \{d_{\K} < s\}) \leq 4 \varepsilon \,\Mass(T)$, then it immediately follows that $\Mass(T \llcorner \{d_{\K} < s\}) \leq C(m) \,s \,\Mass(T)^{\frac{m-1}{m}}$.  If instead $\Mass(T \llcorner \{d_{\K} < s\}) > 4 \varepsilon \,\Mass(T)$, it follows from \eqref{bdry rect rmk2 eqn3} that  
\begin{equation*}
	\frac{\Mass(T \llcorner \{d_{\K} < s\})}{s} \leq C(m) \,\frac{\Mass(T \llcorner \{d_{\K} < t\})}{t} 
\end{equation*}
for $s < t \leq \min\{1/(2^{m+2} H_0), \,1/(2\kappa_0)\}$ and by substituting for $t$ the smaller of $1/(2^{m+2} H_0)$ and $1/(2\kappa_0)$ we get \eqref{bdry rect rmk2 concl1}.  \eqref{bdry rect rmk2 concl2} follows from \eqref{bdry rect rmk2 concl1} and Lemma \ref{boundary rectifiability lemma1}.  \eqref{bdry rect rmk2 concl3} follows from \eqref{bdry rect rmk2 concl1} and \eqref{bdry rect rmk2 concl2} like in Corollary \ref{boundary rectifiability cor}. 
\end{proof}

Finally, we obtain local mass estimates for $T$ and $R$ near $\partial \K$.  This is based on~\cite[Theorem 3.4]{Gru85}.  Note that the proof of Lemma \ref{boundary rectifiability lemma4} assumes $\K$ has a $C^2$-boundary, and moreover that as a local result Lemma \ref{boundary rectifiability lemma4} holds true whenever $\K$ has a $C^2$-boundary without the requirement that $\K$ is convex.  For each $y \in \partial \K$ and $r,s > 0$, let 
\begin{equation*}
	Q_{r,s}(y) = B_r(y) \cap \{ x : 0 < \op{dist}(x,\K) < s \} .
\end{equation*}

\begin{lemma} \label{boundary rectifiability lemma4} 
Suppose $\K \subseteq \R^{n+1}$ is the closure of a bounded open subset with $C^2$-boundary.  Let $T \in \mathbf{I}_m(\R^{n+1} \setminus \K)$ and $R \in \mathbf{I}_{m+1}(\R^{n+1} \setminus \K)$ such that $\partial R = T$ in $\R^{n+1} \setminus \K$, $R$ is relatively area minimizing in $\R^{n+1} \setminus \K$, and $(T,R)$ is a relative isoperimetric minimizer in $\R^{n+1} \setminus \K$.  Then 
\begin{equation} \label{bdry rect4 concl1} 
	\frac{\|T\|(Q_{r,s}(y))}{r^{m-1} s} \leq C(m) \,\frac{\|T\|(B_{2r}(y))}{r^m} 
\end{equation}
for all $y \in \partial \K$ and $0 < s < r \leq \min\{1/(2^{m+2} H_0), \,1/(4n\kappa_0)\}$ and 
\begin{equation} \label{bdry rect4 concl2} 
	\frac{\|R\|(Q_{r,s}(y))}{r^m s} \leq C(m) \left( \frac{\|T\|(B_{2r}(y))}{r^m} + \frac{\|R\|(B_{2r}(y))}{r^{m+1}} \right) 
\end{equation}
for all $y \in \partial \K$ and $0 < s < r \leq 1/(4n\kappa_0)$, where $\kappa_0$ is as in \eqref{kappa0 defn}.  In particular, 
\begin{gather} 
	\label{bdry rect4 concl3} \frac{\|\partial T\|(B_r(y))}{r^{m-1}} \leq C(m) \,\frac{\|T\|(B_{2r}(y))}{r^m} \\
	\label{bdry rect4 concl4} \frac{\|\partial R\|(B_r(y))}{r^m} \leq C(m) \left( \frac{\|T\|(B_{2r}(y))}{r^m} + \frac{\|R\|(B_{2r}(y))}{r^{m+1}} \right)  
\end{gather}
for every $y \in \partial \K$ and $0 < r < \min\{1/(2^{m+2} H_0), \,1/(4n\kappa_0)\}$. 
\end{lemma}

\begin{proof}
Fix $y \in \partial \K$.  For each $0 < s < r \leq 1/(4n\kappa_0)$, let $\widehat{g}_{r,s} : Q_{r,s}(y) \cup (\R^{n+1} \setminus B_{2r}(y)) \cup \partial \K \rightarrow \R^{n+1}$ be the function defined by  
\begin{align*}
	\widehat{g}_{r,s}(x) &= \xi_{\K}(x) - x \text{ on } Q_{r,s}(y), \\ 
	\widehat{g}_{r,s}(x) &= 0 \text{ on } (\R^{n+1} \setminus B_{2r}(y)) \cup \K. 
\end{align*}
Clearly $\sup_{Q_{r,s}(y)} |\widehat{g}_{r,s}| = s$.  We claim that $\op{Lip} \widehat{g}_{r,s} \leq s/r$.  By \eqref{grad d xi nu}, 
\begin{equation*}
	\nabla_{e_i} (\xi_{\K}(x) - x) = \frac{1}{1 + \kappa_i d_{\K}(x)} \,e_i - e_i 
		= \frac{\kappa_i d_{\K}(x)}{1 + \kappa_i d_{\K}(x)} \,e_i
\end{equation*}
on $Q_{2r,s}(y)$, where $\kappa_i$ is the principal curvature of $\partial \K$ at $\xi_{\K}(x)$ in the principal direction $e_i$, and so  
\begin{equation*}
	\op{Lip} \widehat{g}_{r,s} |_{Q_{r,s}(y) \cup (\K \cap B_{2r}(y))} \leq \op{Lip} (\xi_{\K}(x) - x) |_{Q_{2r,s}(y)} 
		\leq \frac{\sqrt{n}\,\kappa_0 s}{1 - \kappa_0 s} < 2n \kappa_0 s \leq \frac{s}{r}  
\end{equation*}
using $s < r \leq 1/(4n\kappa_0)$.  If $x \in Q_{r,s}(y)$ and $z \in \R^{n+1} \setminus B_{2r}(y)$ then 
\begin{equation*}
	|\widehat{g}_{r,s}(x) - \widehat{g}_{r,s}(z)| = |\widehat{g}_{r,s}(x)| = d_{\K}(x) \leq s \leq \frac{s}{r} \,|x-z|.
\end{equation*}
By Kirszbraun's theorem~\cite[2.10.43]{Fed69}, $\widehat{g}_{r,s}$ extends to $g_{r,s} : \R^{n+1} \rightarrow \R^{n+1}$ such that $\sup g_{r,s} \leq s$ and $\op{Lip} g_{r,s} \leq r/s$.  Define $f_{r,s} : \R^{n+1} \rightarrow \R^{n+1}$ and $h_{r,s} : [0,1] \times \R^{n+1} \rightarrow \R^{n+1}$ by $f_{r,s}(x) = x + g_{r,s}(x)$ and $h_{r,s}(t,x) = x + t \, g_{r,s}(x)$ for all $t \in [0,1]$ and $x \in \R^{n+1}$. 

We will apply Lemma \ref{almost minimizing lemma} with 
\begin{equation} \label{bdry rect4 eqn1}
	X = X_{r,s} = h_{r,s \#}(\llbracket 0,1 \rrbracket \times T) 
\end{equation}
in $\R^{n+1} \setminus \K$ for all $0 < s < r \leq 1/(4n\kappa_0)$.  Since $h_{r,s}(t,x) = x$ for all $t \in [0,1]$ and $x \in \R^{n+1} \setminus B_{2r}(y)$, 
\begin{equation} \label{bdry rect4 eqn2}
	X_{r,s} = h_{r,s \#}(\llbracket 0,1 \rrbracket \times (T \llcorner B_{2r}(y))) 
\end{equation}
in $\R^{n+1} \setminus \K$ for all $0 < s < r \leq 1/(4n\kappa_0)$ and in particular $X_{r,s}$ has compact support.  We want to compute $\Mass_{\R^{n+1} \setminus \K}(T + \partial X_{r,s})$ and $\Mass(X_{r,s})$.  By applying the homotopy formula in \eqref{bdry rect4 eqn1} and using $f_{r,s}(x) = x$ for all $x \in \R^{n+1} \setminus B_{2r}(y)$, 
\begin{equation} \label{bdry rect4 eqn3}
	\partial X_{r,s} = f_{r,s \#}(T) - T = f_{r,s \#}(T \llcorner B_{2r}(y)) - T \llcorner B_{2r}(y) 
\end{equation}
in $\R^{n+1} \setminus \K$ for all $0 < s < r \leq 1/(4n\kappa_0)$.  Since $f_{r,s}(Q_{r,s}(y)) \subseteq \K$ and $Q_{r,s}(y) \subseteq f_{r,s}^{-1}(f_{r,s}(Q_{r,s}(y)))$, 
\begin{align*}
	\Mass_{\R^{n+1} \setminus \K}\big( f_{r,s \#}\big( T \llcorner B_{2r}(y) \big) \big) 
	&\leq \Mass\big( f_{r,s \#}\big( T \llcorner B_{2r}(y) \big) \llcorner \big( \R^{n+1} \setminus f_{r,s}(Q_{r,s}(y)) \big) \big)
	\\&= \Mass\big( f_{r,s \#}\big( T \llcorner \big( B_{2r}(y) \setminus f_{r,s}^{-1}(f_{r,s}(Q_{r,s}(y))) \big) \big) \big)  
	\\&\leq \Mass\big( f_{r,s \#}\big( T \llcorner \big( B_{2r}(y) \setminus Q_{r,s}(y) \big) \big) \big)  
\end{align*}
for all $0 < s < r \leq 1/(4n\kappa_0)$.  Since $\op{Lip} g_{r,s} \leq s/r$ we have $\op{Lip} f_{r,s} \leq 1 + s/r$ and thus  
\begin{equation} \label{bdry rect4 eqn4}
	\Mass_{\R^{n+1} \setminus \K}(f_{r,s \#}(T \llcorner B_{2r}(y)))) \leq \left(1 + \frac{s}{r}\right)^m \Mass(T \llcorner B_{2r}(y) \setminus Q_{r,s}(y)). 
\end{equation}
for all $0 < s < r \leq 1/(4n\kappa_0)$.  By \eqref{bdry rect4 eqn3} and \eqref{bdry rect4 eqn4}, 
\begin{equation} \label{bdry rect4 eqn5}
	\Mass_{\R^{n+1} \setminus \K}(T + \partial X_{r,s}) \leq \Mass(T \llcorner \R^{n+1} \setminus B_{2r}(y)) 
		+ \left(1 + \frac{s}{r}\right)^m \Mass(T \llcorner B_{2r}(y) \setminus Q_{r,s}(y))
\end{equation}
for all $0 < s < r \leq 1/(4n\kappa_0)$.  Using \eqref{bdry rect4 eqn2} and $\op{Lip} g_{r,s} \leq s/r$, which implies $\op{Lip} h_{r,s}(\cdot,t) \leq 1 + s/r$ for all $t \in [0,1]$, 
\begin{equation} \label{bdry rect4 eqn6}
	\Mass(X_{r,s}) = \Mass(h_{r,s \#}(\llbracket 0,1 \rrbracket \times (T \llcorner B_{2r}(y)))) \leq s \left(1 + \frac{s}{r}\right)^m \Mass(T \llcorner B_{2r}(y)) 
\end{equation}
for all $0 < s < r \leq 1/(4n\kappa_0)$.  Note that when applying Lemma \ref{almost minimizing lemma}, we will want to assume that $\Mass(X_{r,s}) \leq 2^m \,s \,\Mass(T) \leq \tfrac{1}{2} \,\Mass(R)$, which as in Lemma \ref{boundary rectifiability lemma2} follows if $s \leq 1/(2^{m+2} H_0)$.  Hence Lemma \ref{almost minimizing lemma} together with \eqref{bdry rect4 eqn5} and \eqref{bdry rect4 eqn6} gives us 
\begin{equation*}
	\Mass(T \llcorner B_{2r}(y)) \leq \left(1 + \frac{s}{r}\right)^m \Mass(T \llcorner (B_{2r}(y) \setminus Q_{r,s}(y))) 
		+ 2 H_0 s \left(1 + \frac{s}{r}\right)^m \Mass(T \llcorner B_{2r}(y))
\end{equation*}
for all $0 < s < r \leq  \min\{1/(2^{m+2} H_0), \,1/(4n\kappa_0)\}$.  By Taylor's theorem, $(1+s/r)^m \leq 1 + m \,2^{m-1} \,s/r$, so 
\begin{equation*}
	\Mass(T \llcorner B_{2r}(y)) \leq \left( 1 + \frac{m\,2^{m-1}\,s}{r} \right) \Mass(T \llcorner (B_{2r}(y) \setminus Q_{r,s}(y))) 
		+ 2^{m+1} H_0 s \Mass(T \llcorner B_{2r}(y)) 
\end{equation*}
for all $0 < s < r \leq \min\{1/(2^{m+2} H_0), \,1/(4n\kappa_0)\}$.  By subtracting $\Mass(T \llcorner (B_{2r}(y) \setminus Q_{r,s}(y)))$ from both sides, 
\begin{equation*}
	\Mass(T \llcorner Q_{r,s}(y)) \leq C(m) \,\frac{s}{r} \,\Mass(T \llcorner B_{2r}(y))
\end{equation*}
for all $0 < s < r \leq \min\{1/(2^{m+2} H_0), \,1/(4n\kappa_0)\}$.  Dividing by $r^{m-1} s$ we obtain \eqref{bdry rect4 concl1}. 

Next we will use $R$ being relatively area minimizing, comparing $R$ to $R_{r,s} \in \mathbf{I}_{m+1}(\R^{n+1} \setminus \K)$ given by 
\begin{equation} \label{bdry rect4 eqn7}
	R_{r,s} = f_{r,s \#} R - h_{r,s \#}(\llbracket 0,1 \rrbracket \times T) 
\end{equation}
for all $0 < s < r \leq 1/(4n\kappa_0)$.  Since $f_{r,s}(x) = h_{r,s}(t,x) = x$ for all $t \in [0,1]$ and $x \in \R^{n+1} \setminus B_{2r}(y)$, 
\begin{equation} \label{bdry rect4 eqn8}
	R_{r,s} = R \llcorner \R^{n+1} \setminus B_{2r}(y) + f_{r,s \#} (R \llcorner B_{2r}(y)) - h_{r,s \#}(\llbracket 0,1 \rrbracket \times (T \llcorner B_{2r}(y))) 
\end{equation}
in $\R^{n+1} \setminus \K$ and in particular $\op{spt}(R_{r,s} - R)$ is compact.  By applying the homotopy formula in \eqref{bdry rect4 eqn7}, $\partial R_{r,s} = T$ in $\R^{n+1} \setminus \K$ for all $0 < s < r \leq 1/(4n\kappa_0)$.  Hence since $R$ is relatively area minimizing, $\Mass(R) \leq \Mass(R_{r,s})$ and so by \eqref{bdry rect4 eqn8} 
\begin{equation} \label{bdry rect4 eqn9}
	\Mass(R \llcorner B_{2r}(y)) \leq \Mass_{\R^{n+1} \setminus \K}(f_{r,s \#}(R \llcorner B_{2r}(y))) 
		+ \Mass(h_{r,s \#}( \llbracket 0,1 \rrbracket \times (T  \llcorner B_{2r}(y)) )) 
\end{equation}
for all $0 < s < r \leq 1/(4n\kappa_0)$.  Arguing like we did above to obtain \eqref{bdry rect4 eqn4} using $f_{r,s}(Q_{r,s}(y)) \subseteq \K$ and $\op{Lip} f_{r,s} \leq 1 + s/r$, 
\begin{align} \label{bdry rect4 eqn10}
	\Mass(f_{r,s \#}(R \llcorner B_{2r}(y))) &\leq \Mass(f_{r,s \#}( R \llcorner (B_{2r}(y) \setminus Q_{R,s}(y)) )) 
	\\&\leq \left(1+\frac{s}{r}\right)^{m+1} \Mass(R \llcorner (B_{2r}(y) \setminus Q_{R,s}(y))) \nonumber 
\end{align}
for all $0 < s < r \leq 1/(4n\kappa_0)$.  Since $\op{Lip} h_{r,s}(\cdot,t) \leq 1 + s/r$ for all $t \in [0,1]$, 
\begin{equation} \label{bdry rect4 eqn11}
	\Mass(h_{r,s \#}(\llbracket 0,1 \rrbracket \times (T \llcorner B_{2r}(y))) \leq s \left(1+\frac{s}{r}\right)^m \Mass(T \llcorner B_{2r}(y))
\end{equation}
for all $0 < s < r \leq 1/(4n\kappa_0)$.  Thus by \eqref{bdry rect4 eqn9}, \eqref{bdry rect4 eqn10}, and \eqref{bdry rect4 eqn11}, 
\begin{equation*}
	\Mass(R \llcorner B_{2r}(y)) \leq \left(1+\frac{s}{r}\right)^{m+1} \Mass(R \llcorner (B_{2r}(y) \setminus Q_{r,s}(y))) 
		+ s \left(1+\frac{s}{r}\right)^m \Mass(T \llcorner B_{2r}(y)) 
\end{equation*}
for all $0 < s < r \leq 1/(4n\kappa_0)$.  By Taylor's theorem, $(1+s/r)^{m+1} \leq 1 + (m+1) \,2^m \,s/r$, so 
\begin{equation*}
	\Mass(R \llcorner B_{2r}(y)) 
	\leq \left( 1 + \frac{(m+1) \,2^m \,s}{r} \right) \Mass(R \llcorner (B_{2r}(y) \setminus Q_{r,s}(y)))  + 2^m \,s \,\Mass(T \llcorner B_{2r}(y)) 
\end{equation*}
for all $0 < s < r \leq 1/(4n\kappa_0)$.  By subtracting $\Mass(R \llcorner (B_{2r}(y) \setminus Q_{r,s}(y)))$ from both sides, 
\begin{align*}
	\Mass(R \llcorner Q_{r,s}(y)) \leq C(m) \,\frac{s}{r} \,\Mass(R \llcorner B_{2r}(y)) + C(m) \,s \,\Mass(T \llcorner B_{2r}(y))
\end{align*}
for all $0 < s < r \leq 1/(4n\kappa_0)$.  Dividing by $r^m s$ we obtain \eqref{bdry rect4 concl2}. 

To see \eqref{bdry rect4 concl3} and \eqref{bdry rect4 concl4}, fix $y \in \partial \K$ and $r > 0$.  Let us consider $T$ as a current of $B_r(y)$.  For each integer $k \geq 0$ set $s_k = 2^{-k} r$.  By slicing theory using \eqref{meas_good_slices} with $\vartheta = 1/4$, \eqref{bdry rect4 concl1}, and  \eqref{bdry rect4 concl2}, for each integer $k \geq 1$ there exists $s^*_k \in (s_k,s_{k-1})$ such that $T \llcorner (B_r(y) \cap \{d_{\K} > s^*_k\}) \in \mathbf{I}_{m,{\rm loc}}(B_r(y))$, $\langle T, d_{\K}, s^*_k \rangle \llcorner B_r(y) \in \mathbf{I}_{m-1,{\rm loc}}(B_r(y))$, $R \llcorner (B_r(y) \cap \{d_{\K} > s^*_k\}) \in \mathbf{I}_{m+1,{\rm loc}}(B_r(y))$, and $\langle R, d_{\K}, s^*_k \rangle \llcorner B_r(y) \in \mathbf{I}_{m,{\rm loc}}(B_r(y))$ with 
\begin{gather}
	\label{bdry rect4 eqn12} \partial (T \llcorner \{d_{\K} > s^*_k\}) = -\langle T, d_{\K}, s^*_k \rangle \text{ in } B_r(y), \\
	\label{bdry rect4 eqn13} \partial (R \llcorner \{d_{\K} > s^*_k\}) = T \llcorner \{d_{\K} > s^*_k\} - \langle R, d_{\K}, s^*_k \rangle \text{ in } B_r(y), \\
	\label{bdry rect4 eqn14} \frac{\|\langle T, d_{\K}, s^*_k \rangle\|(B_r(y))}{r^{m-1}} \leq \frac{4 \,\|T\|(Q_{r,s_k}(y))}{r^{m-1} s_k} 
		\leq C(m) \,\frac{\|T\|(B_{2r}(y))}{r^m}, \\
	\label{bdry rect4 eqn15} \frac{\|\langle R, d_{\K}, s^*_k \rangle\|(B_r(y))}{r^m} \leq \frac{4 \,\|R\|(Q_{r,s_k}(y))}{r^m s_k} 
		\leq C(m) \left( \frac{\|T\|(B_{2r}(y))}{r^m} + \frac{\|R\|(B_{2r}(y))}{r^{m+1}} \right) 
\end{gather}
Clearly $T \llcorner \{d_{\K} > s^*_k\} \rightarrow T$ and $R \llcorner \{d_{\K} > s^*_k\} \rightarrow R$ in the mass norm topology on $B_r(y)$.  By the Federer-Fleming compactness theorem, \eqref{bdry rect4 eqn12}, \eqref{bdry rect4 eqn13}, \eqref{bdry rect4 eqn14}, and \eqref{bdry rect4 eqn15}, after passing to a subsequence $T \llcorner \{d_{\K} > s^*_k\}$ and $R \llcorner \{d_{\K} > s^*_k\}$ converge in the flat norm topology to integral currents in $B_r(y)$ as $k \rightarrow \infty$ and thus $T \llcorner B_r(y) \in \mathbf{I}_m(B_r(y))$ and $R \llcorner B_r(y) \in \mathbf{I}_{m+1}(B_r(y))$.  Moreover, after passing to a subsequence $-\langle T, d_{\K}, s^*_k \rangle \rightarrow \partial T$ and $-\langle R, d_{\K}, s^*_k \rangle \rightarrow \partial R - T$ weakly in $B_r(y)$.  By the semi-continuity of mass and \eqref{bdry rect4 eqn14} we have 
\begin{equation*}
	\frac{\|\partial T\|(B_r(y))}{r^{m-1}} \leq \liminf_{j \rightarrow \infty} \frac{\|\langle T, d_{\K}, s^*_k \rangle\|(B_r(y))}{r^{m-1}} 
		\leq C(m) \,\frac{\|T\|(B_{2r}(y))}{r^m}, 
\end{equation*}
proving \eqref{bdry rect4 concl3}, and similarly by \eqref{bdry rect4 eqn15} we have \eqref{bdry rect4 concl4}. 
\end{proof}

\section{Non-concentration of mass at infinity} \label{sec:nonconinf_sec} 

The main goal of this section is to show that if the relative isoperimetric ratio of $(T,R)$ is close to $\gamma_{m,n}(\K)$, then the mass of $T$ and $R$ cannot concentrate at infinity. 

Let us first consider the special case where $(T,R)$ is a relative isoperimetric minimizer and $R$ consists of two separate components.  In other words, suppose there exists currents $T_1,T_2 \in \mathbf{I}_m(\R^{n+1} \setminus \K)$ and $R_1,R_2 \in \mathbf{I}_{m+1}(\R^{n+1} \setminus \K)$ such that each $R_i$ is relatively area minimizing in $\R^{n+1} \setminus \K$ with 
\begin{equation*}
	T = T_1 + T_2, \quad R = R_1 + R_2, \quad \partial R_i = T_i \text{ for } i = 1,2
\end{equation*}
in $\R^{n+1} \setminus \K$ and 
\begin{equation*}
	\op{spt} R_1 \cap \op{spt} R_2 = \emptyset. 
\end{equation*}
Using $(T,R)$ being relative isoperimetric minimizing and $R_1$ and $R_2$ being both relatively area minimizing 
\begin{align*}
	\Mass(T)^{\frac{m+1}{m}} = \gamma_{m,n}(\K) \,\Mass(R) = \gamma_{m,n}(\K) \,\big( \Mass(R_1) + \Mass(R_2) \big)
		\leq \Mass(T_1)^{\frac{m+1}{m}} + \Mass(T_2)^{\frac{m+1}{m}} . 
\end{align*}
By dividing both sides by $\Mass(T)^{\frac{m+1}{m}}$, 
\begin{equation*}
	1 \leq r^{\frac{m+1}{m}} + (1-r)^{\frac{m+1}{m}} \quad \text{where} \quad r = \frac{\Mass(T_1)}{\Mass(T)}. 
\end{equation*}
However, this holds true if and only if $r = 0$ or $r = 1$.  Thus $R$ can have only one component. 

By modifying this computation, we make the following observation previously noted by Almgren in~\cite{Alm86} in the context of the isoperimetric inequality in $\R^{n+1}$.  Suppose the relative isoperimetric ratio of $(T,R)$ is close to $\gamma_{m,n}(\K)$.  Divide both $T$ and $R$ into two currents by slicing them with a sphere $\partial B_{\rho}(0)$ which is chosen so that the masses of $T$ and $R$ are negligible along $\partial B_{\rho}(0)$.  Arguing much like above, one can show that either $\Mass(T \llcorner B_{\rho}(0))$ and $\Mass(R \llcorner B_{\rho}(0))$ are small or $\Mass(T \llcorner \R^{n+1} \setminus B_{\rho}(0))$ and $\Mass(R \llcorner \R^{n+1} \setminus B_{\rho}(0))$ are small.  In the special case that $\op{diam}(\K) < \rho$, by translating before slicing we can assume that $\K \subset B_{\rho}(0)$ and then deduce from the isoperimetric inequality in $\R^{n+1}$ that $T$ and $R$ cannot concentrate at infinity.  If instead $\rho << \op{diam}(\K)$ we will have to show that $T$ concentrates in some ball, which is necessarily near $\partial \K$, in order to conclude that $T$ and $R$ cannot concentrate at infinity; we will do in Lemma \ref{local concentration lemma} via the deformation theorem. 

\begin{lemma} \label{nonconcentration lemma}
For every $0 < \sigma < 1 - 2^{-1/m}$ there exists $\varepsilon = \varepsilon(m,\sigma) > 0$ and $\beta_0 = \beta_0(m,\sigma) \in [1,\infty)$ such that the following holds true.  Let $\K$ be a closed subset $\R^{n+1}$ such that 
\begin{equation} \label{noncon hyp1} 
	0 < \gamma_{m,n}(\K) \leq 2^{-\frac{1}{m}} \frac{\mathcal{H}^m(\sphere)^{\frac{m+1}{m}}}{\mathcal{H}^{m+1}(\ball)}. 
\end{equation}
Let $T \in \mathbf{I}_m(\R^{n+1} \setminus \K)$ and $R \in \mathbf{I}_{m+1}(\R^{n+1} \setminus \K)$ such that $\partial R = T$ in $\R^{n+1} \setminus \K$, $R$ is relatively area minimizing in $\R^{n+1} \setminus \K$, and 
\begin{equation} \label{noncon hyp2} 
	\frac{\Mass(T)^{\frac{m+1}{m}}}{\Mass(R)} \leq (1+\varepsilon) \,\gamma_{m,n}(\K). 
\end{equation}
Then for every $\beta \in [\beta_0,\infty)$, either 
\begin{align} \label{noncon concl1}
	&\Mass(T \llcorner B_{\beta\,\Mass(R)^{\frac{1}{m+1}}}(0)) \leq \sigma\,\Mass(T), \\
	&\Mass(R \llcorner B_{\beta\,\Mass(R)^{\frac{1}{m+1}}}(0)) \leq \sigma\,\Mass(R) \nonumber 
\end{align}
or 
\begin{align} \label{noncon concl2}
	&\Mass(T \llcorner \R^{n+1} \setminus B_{2\beta\,\Mass(R)^{\frac{1}{m+1}}}(0)) \leq \sigma\,\Mass(T), \\
	&\Mass(R \llcorner \R^{n+1} \setminus B_{2\beta\,\Mass(R)^{\frac{1}{m+1}}}(0)) \leq \sigma\,\Mass(R). \nonumber 
\end{align}
If additionally $\K \subset B_{\beta \,\Mass(R)^{\frac{1}{m+1}}}(0)$, then \eqref{noncon concl2} must hold true.  If instead $\op{dist}(0,\K) > 2\beta \,\Mass(R)^{\frac{1}{m+1}}$, then \eqref{noncon concl1} must hold true.
\end{lemma}

\begin{proof}
Consider the continuous function $f : [0,1]^2 \setminus \{(0,0), (1,1)\} \rightarrow (0,\infty)$ defined by 
\begin{equation*}
	f(x,y) = \min\left\{ \frac{x^{\frac{m+1}{m}}}{y}, \frac{(1-x)^{\frac{m+1}{m}}}{1-y} \right\} . 
\end{equation*}
In particular, $f(0,y) = 0$ for all $y \in (0,1]$, $f(1,y) = 0$ for all $y \in [0,1)$, $f(x,0) = (1-x)^{\frac{m+1}{m}}$ for all $x \in (0,1]$, and $f(x,1) = x^{\frac{m+1}{m}}$ for all $x \in [0,1)$.  We claim that $f(x,y) < 1$ for all $(x,y) \in [0,1]^2 \setminus \{(0,0), (1,1)\}$.  Otherwise $f(x,y) \geq 1$ for some $(x,y) \in (0,1)^2$, which implies 
\begin{equation*}
	x^{\frac{m}{m+1}} \geq y, \quad (1-x)^{\frac{m}{m+1}} \geq 1-y. 
\end{equation*}
By adding both inequalities, 
\begin{equation*}
	x^{\frac{m}{m+1}} + (1-x)^{\frac{m}{m+1}} \geq 1,
\end{equation*}
which is impossible by the strict concavity of $x^{\frac{m}{m+1}}$.  Fix $\sigma > 0$.  There exists $\varepsilon = \varepsilon(m,\sigma) \in (0,1)$ such that 
\begin{equation*}
	(1+\varepsilon)^{-3} > \min\left\{ \frac{x^{\frac{m+1}{m}}}{y}, \frac{(1-x)^{\frac{m+1}{m}}}{1-y} \right\}
\end{equation*}
for all $x,y \in [0,1]^2 \setminus ([0,\sigma]^2 \cup [1-\sigma,1]^2)$.  Moreover, there exists $\delta = \delta(m,\sigma) > 0$ such that 
\begin{equation} \label{noncon eqn1}
	(1+\varepsilon)^{-2} > \min\left\{ \frac{(x+\delta)^{\frac{m+1}{m}}}{y}, \frac{(1-x+\delta)^{\frac{m+1}{m}}}{1-y} \right\}
\end{equation}
for all $x,y \in [0,1]^2 \setminus ([0,\sigma]^2 \cup [1-\sigma,1]^2)$.

Let $\rho_0 = \beta \,\Mass(R)^{\frac{1}{m+1}}$.  By slicing theory using \eqref{meas_good_slices} with $\vartheta = 1/2$ there exists $\rho \in (\rho_0,2\rho_0)$ such that $T \llcorner B_{\rho}(0), T \llcorner \R^{n+1} \setminus B_{\rho}(0) \in \mathbf{I}_m(\R^{n+1} \setminus \K)$, $R \llcorner B_{\rho}(0), R \llcorner \R^{n+1} \setminus B_{\rho}(0) \in \mathbf{I}_{m+1}(\R^{n+1} \setminus \K)$, and $\langle R, |\cdot|, \rho \rangle \in \mathbf{I}_m(\R^{n+1} \setminus \K)$ with 
\begin{gather} 
	\partial (R \llcorner B_{\rho}(0)) 
		= T \llcorner B_{\rho}(0) + \langle R, |\cdot|, \rho \rangle \text{ in } \R^{n+1} \setminus \K, \label{noncon eqn2} \\
	\partial (R \llcorner \R^{n+1} \setminus B_{\rho}(0)) 
		= T \llcorner \R^{n+1} \setminus B_{\rho}(0) - \langle R, |\cdot|, \rho \rangle \text{ in } \R^{n+1} \setminus \K, \label{noncon eqn3} \\
	\Mass(\langle R, |\cdot|, \rho \rangle) < \frac{2}{\rho_0} \,\Mass(R). \label{noncon eqn4} 
\end{gather}
By taking $\beta_0 = \beta_0(m,\sigma) > 2 \,\delta^{-1} \,\gamma_{m,n}(\K)^{-\frac{m}{m+1}}$, where $\delta$ is as in \eqref{noncon eqn1}, and using $\rho_0 = \beta \,\Mass(R)^{\frac{1}{m+1}}$ for $\beta \geq \beta_0$, \eqref{noncon eqn4} gives us 
\begin{align} \label{noncon eqn5}
	\Mass(\langle R, |\cdot|, \rho \rangle) 
		&< \frac{2}{\rho_0} \,\Mass(R) 
		= \frac{2}{\beta} \,\Mass(R)^{\frac{m}{m+1}} 
		\leq \frac{2}{\beta_0} \,\Mass(R)^{\frac{m}{m+1}} 
		\\&\leq \frac{2}{\beta_0\,\gamma_{m,n}(\K)^{\frac{m}{m+1}}} \,\Mass(T) 
		< \delta \,\Mass(T). \nonumber 
\end{align}

Let 
\begin{equation*}
	x = \frac{\Mass(T \llcorner B_{\rho}(0))}{\Mass(T)}, \quad y = \frac{\Mass(R \llcorner B_{\rho}(0))}{\Mass(R)}. 
\end{equation*}
Suppose $(x,y) \not\in [0,\sigma]^2 \cup [1-\sigma,1]^2$.  Since $R \llcorner B_{\rho}(0)$ and $R \llcorner \R^{n+1} \setminus B_{\rho}(0)$ are relatively area minimizing in $\R^{n+1} \setminus \K$ with boundaries given by \eqref{noncon eqn2} and \eqref{noncon eqn3}
\begin{equation*}
	\gamma_{m,n}(\K) 
	\leq \min \left\{ \frac{\Mass(T \llcorner B_{\rho}(0) + \langle R, |\cdot|, s \rangle)^{\frac{m+1}{m}}}{\Mass(R \llcorner B_{\rho}(0))}, 
		\frac{\Mass(T \llcorner \R^{n+1} \setminus B_{\rho}(0) - \langle R, |\cdot|, s \rangle)^{\frac{m+1}{m}}}{
			\Mass(R \llcorner \R^{n+1} \setminus B_{\rho}(0))} \right\} .
\end{equation*}
Using \eqref{noncon eqn5}, \eqref{noncon hyp2}, and \eqref{noncon eqn1} 
\begin{align*}
	\gamma_{m,n}(\K) 
	&< \min \left\{ \frac{(x+\delta)^{\frac{m}{m+1}}}{y}, \frac{(1-x+\delta)^{\frac{m}{m+1}}}{1-y} \right\} \frac{\Mass(T)^{\frac{m+1}{m}}}{\Mass(R)}
	\\&< (1+\varepsilon)^{-2} \cdot (1+\varepsilon) \,\gamma_{m,n}(\K) = (1+\varepsilon)^{-1} \,\gamma_{m,n}(\K)
\end{align*}
which is impossible.  Therefore, either $(x,y) \in [0,\sigma]^2$, which implies \eqref{noncon concl1}, or $(x,y) \in [1-\sigma,1]^2$, which implies \eqref{noncon concl2}. 

Suppose that $\K \subset B_{\rho_0}(0)$.  Further suppose that $(x,y) \in [0,\sigma]^2$.  By the isoperimetric inequality applied to the area minimizing current $R \llcorner \R^{n+1} \setminus B_{\rho}(0)$ with boundary given by \eqref{noncon eqn3}
\begin{equation*}
	\frac{\mathcal{H}^m(\sphere)^{\frac{m+1}{m}}}{\mathcal{H}^{m+1}(\ball)} 
	\leq \frac{\Mass(T \llcorner \R^{n+1} \setminus B_{\rho}(0) - \langle R, |\cdot|, s \rangle)^{\frac{m+1}{m}}}{
			\Mass(R \llcorner \R^{n+1} \setminus B_{\rho}(0))} .
\end{equation*}
Using \eqref{noncon eqn5} and \eqref{noncon hyp2},  
\begin{equation*}
	\frac{\mathcal{H}^m(\sphere)^{\frac{m+1}{m}}}{\mathcal{H}^{m+1}(\ball)} 
	\leq \frac{(1+\delta)^{\frac{m+1}{m}}}{1-\sigma} \,\frac{\Mass(T)^{\frac{m+1}{m}}}{\Mass(R)} 
	\leq \frac{(1+\delta)^{\frac{m+1}{m}}}{1-\sigma} \,(1+\varepsilon) \,\gamma_{m,n}(\K). 
\end{equation*}
which provided $\sigma < 1-2^{-1/m}$ and $\varepsilon$ and $\delta$ are sufficiently small contradicts \eqref{noncon hyp1}.  Therefore, $(x,y) \in [1-\sigma,1]^2$, which implies \eqref{noncon concl2}.  Arguing along the same lines using the isoperimetric inequality applied to $R \llcorner B_{\rho}(0)$, if $\op{dist}(0,\mathcal{\K}) > 2\beta \,\Mass(R)^{\frac{1}{m+1}}$, then \eqref{noncon concl1} must hold true. 
\end{proof}

\begin{lemma} \label{local concentration lemma}
For each bounded proper convex subset $\K \subset \R^{n+1}$ there exists $c = c(m,n,\K) > 0$, $\alpha = \alpha(m,n,\K) > 0$, and $\beta_1 = \beta_1(m,n,\K) > 0$ such that the following holds true.  Let $T \in \mathbf{I}_m(\R^{n+1} \setminus \K)$ and $R \in \mathbf{I}_{m+1}(\R^{n+1} \setminus \K)$ such that $\partial R = T$ in $\R^{n+1} \setminus \K$ and $R$ is relatively area minimizing in $\R^{n+1} \setminus \K$.  Suppose 
\begin{gather} 
	\Mass(R) \leq \alpha \label{loccon1 hyp1} \\
	\frac{\Mass(T)^{\frac{m+1}{m}}}{\Mass(R)} \leq 2 \,\gamma_{m,n}(\K). \label{loccon1 hyp2} 
\end{gather}
Then there exists $x_0 \in \R^{n+1} \setminus \K$ such that 
\begin{equation} \label{loccon1 concl}
	\Mass(T \llcorner B_{\beta_1 \,\Mass(R)^{\frac{1}{m+1}}}(x_0)) \geq c\,\Mass(R)^{\frac{m}{m+1}}. 
\end{equation}
\end{lemma}

\begin{proof} 
First we claim that there exists a bi-Lipschitz diffeomorphism $\varphi : \R^{n+1} \rightarrow \R^{n+1}$ such that $\varphi(\K) = \K_0 \equiv [0,1]^{n+1}$.  To see this, translate so that $0 \in \op{int} \K$.  Since $\K$ is bounded and convex, $\K$ is star-shaped and has a Lipschitz boundary.  In particular, there exists a Lipschitz function $f : \mathbb{S}^n \rightarrow (0,\infty)$ such that 
\begin{equation*}
	\K = \{ r \omega : 0 < r < f(\omega), \, \omega \in \mathbb{S}^n \} . 
\end{equation*}
We define the bi-Lipschitz diffeomorphism $\phi_{\K} : \R^{n+1} \rightarrow \R^{n+1}$ such that $\phi_{\K}(\K) = \overline{B_1(0)}$ by  
\begin{equation*}
	\phi_{\K}(r\omega) = \frac{r\omega}{f(\omega)} 
\end{equation*}
for all $r > 0$ and $\omega \in \mathbb{S}^n$.  By the same argument there is a bi-Lipschitz diffeomorphism $\phi_{\K_0} : \R^{n+1} \rightarrow \R^{n+1}$ such that $\phi_{\K_0}(\K_0) = \overline{B_1(0)}$.  Thus we may let $\varphi = \phi_{\K_0}^{-1} \circ \phi_{\K}$.

Let $N \geq 0$ be an integer to be determined.  Following~\cite{Alm83}, let $\mathbf{K}(N) = \{ \eta_{2^{-N} z,2^{-N}}([0,1]^{n+1}) : z \in \mathbb{Z}^{n+1} \}$ denote the set of all standard coordinate cubes of side length $2^{-N}$.  For $k \in \{0,1,2,\ldots,n+1\}$, let $\mathbf{K}_k(N)$ denote the set of all $k$-dimensional closed faces of coordinate cubes $L \in \mathbf{K}(N)$.  By applying the deformation theorem~\cite[Theorem 1.15]{AlmDef} to $\varphi_{\#} T$, we can show that there exists a constant $C_0 = C_0(m,n,\K) \in (0,\infty)$ and currents $P \in \mathbf{I}_m(\R^{n+1} \setminus \K)$ and $Q \in \mathcal{I}_{m+1,{\rm loc}}(\R^{n+1} \setminus \K)$ such that $P$ is given by  
\begin{equation} \label{loccon1 eqn1}
	P = \sum_{L \in \mathbf{K}_m(N), \,\op{int} L \subset \R^{n+1} \setminus \mathcal{K}_0} \theta(L) \,\llbracket L \rrbracket ,
\end{equation}
where for each $L \in \mathbf{K}_m(N)$ we fix an orientation and let $\theta(L) \in \mathbb{Z}$, 
\begin{equation} \label{loccon1 eqn2}
	T - \varphi^{-1}_{\#} P = \partial Q 
\end{equation}
in $\R^{n+1} \setminus \K$, 
\begin{equation} \label{loccon1 eqn3}
	\Mass(P \llcorner L) \leq C_0 \,\Mass(T \llcorner \varphi^{-1}(\op{Nbs} L)) 
\end{equation}
for each $L \in \mathbf{K}(N)$ with $L \subset \R^{n+1} \setminus \op{int} \K_0$, where we let $\op{Nbs} L$ denote the union of all closed cubes in $\mathbf{K}(N)$ which intersect $L$, and 
\begin{equation} \label{loccon1 eqn4}
	\Mass(Q) \leq C_0\, 2^{-N} \,\Mass(T).
\end{equation}
In order to apply~\cite[Theorem 1.15]{AlmDef} to integer-multiplicity rectifiable currents of $\R^{n+1} \setminus \K_0$, we modify the argument of~\cite[Theorem 1.15]{AlmDef} as follows.  We inductively construct currents $T_k \in \mathbf{I}_m(\R^{n+1} \setminus \K_0)$ and $Q_k \in \mathcal{I}_{m,{\rm loc}}(\R^{n+1} \setminus \K_0)$ for $k = n+1,n,\ldots,m,m-1$.  When $k = n+1$, $T_{n+1} = \varphi_{\#} T$ and $Q_{n+1} = 0$.  For each $k = n+1,n,\ldots,m$ and $k$-dimensional cubes $L \in \mathbf{K}_k(N)$, we pick a ``good'' point $p_L \in \op{int} L$.  Then we project $T_k \llcorner L$ onto $\partial L$ via radial projection at $p_L$ and restrict the image to $\R^{n+1} \setminus \K_0$, sweeping out a current $Q_{k-1} \llcorner L$ in the process.  $T_{k-1}$ is the final image after all the radial projections, restricted to $\R^{n+1} \setminus \K_0$, and $Q_{k-1}$ is the sum of all the swept out currents.  Note that the radial projections of Almgren in~\cite{AlmDef} are different from that of Federer and Fleming in~\cite{FF60}, see~\cite{AlmDef} for details.  Notice that each $L \in \mathbf{K}_k(N)$ either satisfes $\op{int} L \subset \R^{n+1} \setminus \K_0$ or $L \subset \K_0$ and similarly for its boundary faces.  It follows that at each step we only radially project $T_k \llcorner L$ onto $\partial L$ for $L \in \mathbf{K}_k(N)$ with $\op{int} L \subset \R^{n+1} \setminus \K_0$.  As a result $T_{k-1}$ and $Q_{k-1}$ are well-defined as $T_{k-1} \in \mathbf{I}_m(\R^{n+1} \setminus \K_0)$ and $Q_{k-1} \in \mathbf{I}_{m+1}(\R^{n+1} \setminus \K_0)$.  By the argument of~\cite{AlmDef}, in the end we obtain $P = T_{m-1}$ of the form \eqref{loccon1 eqn1} and $\varphi_{\#} Q = \sum_{k=m-1}^n Q_k$ satisfying \eqref{loccon1 eqn2}, \eqref{loccon1 eqn3}, and \eqref{loccon1 eqn4}.  

Define $\alpha$, $\beta_1$, $c$, and $N$ by 
\begin{gather} \label{loccon1 eqn5}
	\alpha = 2^{2m+1} \,C_0^{m+1} \,\gamma_{m,n}(\K)^m, \quad
	\beta_1 = 2\,\sqrt{n+1} \,\op{Lip}(\varphi^{-1}) \,\alpha^{\frac{-1}{m+1}}, \quad 
	c = \frac{1}{2^m \,C_0 \,\alpha^\frac{m}{m+1}}, \\
	\frac{1}{2} \left(\frac{\Mass(R)}{\alpha}\right)^{\frac{1}{m+1}} < 2^{-N} \leq \left(\frac{\Mass(R)}{\alpha}\right)^{\frac{1}{m+1}}. \nonumber 
\end{gather}
Notice that this choice of $N$ together with \eqref{loccon1 hyp1} guarantees that $N \geq 0$.

To show \eqref{loccon1 concl}, suppose to the contrary that 
\begin{equation} \label{loccon1 eqn6}
	\Mass(T \llcorner B_{\beta_1 \,\Mass(R)^{\frac{1}{m+1}}}(x)) < c\,\Mass(R)^{\frac{m}{m+1}} 
\end{equation}
for every $x \in \R^{n+1} \setminus \K$.  For an arbitrary cube $L \in \mathbf{K}(N)$ with $\op{int} L \subset \R^{n+1} \setminus \K$, using  \eqref{loccon1 eqn5}, 
\begin{equation} \label{loccon1 eqn7}
	\varphi^{-1}(\op{Nbs} L) \subset \varphi^{-1}(B_{\sqrt{n+1}\,2^{1-N}}(y)) \subset B_{\sqrt{n+1} \,\op{Lip}(\varphi^{-1}) \,2^{1-N}}(x) 
		\subset B_{\beta_1 \,\Mass(R)^{\frac{1}{m+1}}}(x), 
\end{equation} 
where $y$ is the center of the cube $L$ and $x = \varphi^{-1}(y)$.  Thus by \eqref{loccon1 eqn3}, \eqref{loccon1 eqn7}, \eqref{loccon1 eqn6}, and  \eqref{loccon1 eqn5},  
\begin{align*}
	\Mass(P \llcorner L) 
	&\leq C_0 \,\Mass(T \llcorner \varphi^{-1}(\op{Nbs} L)) 
	\leq C_0 \,\Mass(T \llcorner B_{\beta_1 \,\Mass(R)^{\frac{1}{m+1}}}(x)) 
	< C_0 \,c\,\Mass(R)^{\frac{m}{m+1}} 
	\\&< C_0 \cdot \frac{1}{2^m \,C_0 \,\alpha^{\frac{m}{m+1}}} \cdot 2^{m(1-N)} \,\alpha^{\frac{m}{m+1}} 
	= 2^{-mN}. 
\end{align*}
Hence $\theta(L) = 0$ as otherwise $\Mass(P \llcorner L) = |\theta(L)| \,\mathcal{H}^m(L) \geq 2^{-mN}$ by \eqref{loccon1 eqn1}.  Therefore, $P = 0$ in $\R^{n+1} \setminus \K$.  Hence \eqref{loccon1 eqn2} gives us $T = \partial Q$ in $\R^{n+1} \setminus \K$.  By  \eqref{loccon1 eqn4}, \eqref{loccon1 hyp2}, and \eqref{loccon1 eqn5}, 
\begin{equation*}
	\Mass(Q) 
	\leq C_0\, 2^{-N} \,\Mass(T) 
	\leq C_0 \cdot \left(\frac{\Mass(R)}{\alpha}\right)^{\frac{1}{m+1}} \cdot (2 \,\gamma_{m,n}(K) \,\Mass(R))^{\frac{m}{m+1}} 
	= \frac{1}{2} \,\Mass(R), 
\end{equation*}
contradicting $R$ being relatively area minimizing.  Therefore, \eqref{loccon1 concl} must hold true. 
\end{proof}

\section{Convergence of almost relative area minimizers} \label{sec:convergence_sec}  

Let $1 \leq m \leq n+1$.  Suppose $\K_j, \K \subset \R^{n+1}$ are closed convex subsets and $\K \subset \R^{n+1}$ such that $\K_j \rightarrow \K$ locally in Hausdorff distance.  Suppose $R_j \in \mathbf{I}_m(\R^{n+1} \setminus \K_j)$ is a sequence of relative area minimizing currents converging weakly to a current $R \in \mathbf{I}_m(\R^{n+1} \setminus \K)$.  We want to show that $R$ is relatively area minimizing in $\R^{n+1} \setminus \K$ and $\Mass_W(R_j) \rightarrow \Mass_W(R)$ for any bounded open set $W \subset \R^{n+1}$ with $\|R\|(\partial W) = 0$.  It is known how to show this for area minimizing currents in the interior of $\R^{n+1} \setminus \K$, see for instance~\cite[Theorem 34.5]{Sim83}.  However, we want to this up to the boundary of $\K$.  We have to be careful since $R_j$ might have mass concentrating near the boundary of $\K_j$; for instance, for $j = 1,2,3,\ldots$ consider $R_j = j \,\llbracket B_{1+1/j}(0) \setminus B_1(0) \rrbracket$ in $\R^{n+1} \setminus B_1(0)$.  In the past two sections we have obtained estimates that rule out concentration of mass at the boundary of domains and at infinity.  In light of this, we prove Lemma \ref{convergence lemma} below.  Since we will want to also apply Lemma \ref{convergence lemma} to relative isoperimetric minimizers, we will prove Lemma \ref{convergence lemma} for integral currents which are almost relative area minimizing like in Lemma \ref{almost minimizing lemma}. 

\begin{lemma} \label{convergence lemma}
Let $1 \leq m \leq n$.  Let $\K_j, \K$ be closed convex subsets of $\R^{n+1}$ such that $\K_j \rightarrow \K$ locally in Hausdorff distance (see Subsection \ref{sec:prelims_notation}).  Let $R_j \in \mathbf{I}_m(\R^{n+1} \setminus \K_j)$ and $R \in \mathbf{I}_m(\R^{n+1} \setminus \K)$ such that 
\begin{equation} \label{conv bddmass hyp}
	\limsup_{j \rightarrow \infty} (\Mass_W(R_j) + \Mass_{W \setminus \K_j}(\partial R_j)) < \infty 
\end{equation} 
for every bounded open subset $W \subset \R^{n+1}$ and $R_j \rightarrow R$ in the flat norm topology on compact subsets of the interior of $\R^{n+1} \setminus \K$.  Assume that each $R_j$ is almost relative area minimizing in the sense that there exists $\lambda_j \in [0,\infty)$ and $\mu_j \in (0,\infty]$ with $\lambda_j \rightarrow 0$ and $\mu_j \rightarrow \infty$ such that
\begin{equation} \label{conv aam hyp}
	\Mass_W(R_j) \leq \Mass_{W \setminus \K_j}(R_j + \partial X) + \lambda_j \,\Mass(X) 
\end{equation}
for every bounded open subset $W \subset \R^{n+1}$ and current $X \in \mathbf{I}_{m+1}(\R^{n+1} \setminus \K_j)$ with compact support such that $\op{spt} X \subset W$ and $\Mass(X) \leq \mu_j$.  Further assume that for every $\varepsilon > 0$ and every bounded open set $W \subset \R^{n+1}$ there exists $\delta > 0$ and $J \geq 1$ such that for all $j \geq J$ 
\begin{equation} \label{conv nonconbdry hyp}
	\Mass_W(R_j \llcorner \{d_{\K_j} < \delta\}) \leq \varepsilon .
\end{equation}
Then $R$ is relatively area minimizing in $\R^{n+1} \setminus \K$ and $\|R_j\| \rightarrow \|R\|$ in the sense of Radon measures on $\R^{n+1}$. 
\end{lemma} 

\begin{remark}
When $m = n+1$ this theorem is trivial.  In particular, \eqref{conv aam hyp} only makes sense if $m \leq n$.  We could instead assume each $R_j$ is relatively area minimizing in $\R^{n+1} \setminus \K_j$.  However, given any closed convex set $\K \subset \R^{n+1}$ and any current $R \in \mathbf{I}_{n+1}(\R^{n+1} \setminus \K)$, by the constancy theorem~\cite[Theorem 26.27]{Sim83}, $Q = R$ is the only current $Q \in \mathbf{I}_{n+1}(\R^{n+1} \setminus \K)$ such that $\partial Q = \partial R$ in $\R^{n+1} \setminus \K$ and $\op{spt}(Q-R)$ is compact.  Thus the condition that $R$ is relatively area minimizing in $\R^{n+1} \setminus \K$ trivial.  Moreover, whenever $\K_j, \K \subset \R^{n+1}$ are as in Lemma \ref{convergence lemma} and $R_j \in \mathbf{I}_{n+1}(\R^{n+1} \setminus \K_j)$ such that \eqref{conv bddmass hyp} and \eqref{conv nonconbdry hyp} hold true, we can associate each $R_j$ to a BV-function as in~\cite[Remark 26.28]{Sim83} and thus as a consequence of the compactness of BV-functions after passing to a subsequence $R_j \rightarrow R$ in the mass norm topology on compact subsets of $\R^{n+1}$ for some $R \in \mathbf{I}_{n+1}(\R^{n+1} \setminus \K)$.  Note that \eqref{conv nonconbdry hyp} automatically holds true in the codimension one, multiplicity one setting.
\end{remark}

\begin{proof}[Proof of Lemma \ref{convergence lemma}]  
Our approach involves modifying a well-known argument in the case of area minimizers, see for instance~\cite[Theorem 34.5]{Sim83}.  Let $K \subset \R^{n+1}$ be a compact set and $W \subset \R^{n+1}$ be a bounded open set such that $K \subset W$.  Assume $\K \cap \overline{W} \neq \emptyset$ as otherwise the proof is same with obvious modifications and is in fact a simplification of the argument below.  Choose a smooth function $\phi : \R^{n+1} \rightarrow [0,1]$ such that $\phi = 1$ on $K$ and $\phi = 0$ in $\R^{n+1} \setminus W$.  For $0 < \beta < 1$, set $W_{\beta} = \{ x \in \R^{n+1} : \phi(x) > \beta \}$ and note that $K \subset W_{\beta} \subset W$ for all $0 < \beta < 1$.  Let $X \in \mathbf{I}_{m+1}(\R^{n+1} \setminus \K)$ be any integral current with $\op{spt} X \subseteq K$.  Fix $\varepsilon > 0$ and choose $\delta \in (0,\infty)$ such that \eqref{conv nonconbdry hyp} holds true and 
\begin{equation} \label{conv eqn1}
	\Mass(X \llcorner \{d_{\K} < \delta\}) + \Mass((\partial X) \llcorner \{0 < d_{\K} < \delta\}) \leq \varepsilon . 
\end{equation}
By the semi-continuity of mass and \eqref{conv nonconbdry hyp}, 
\begin{equation} \label{conv nonconbdry eqn}
	\Mass_W(R \llcorner \{d_{\K} < \delta/2\}) \leq \liminf_{j \rightarrow \infty} \Mass(R_j \llcorner \{d_{\K_j} < \delta\}) \leq \varepsilon. 
\end{equation}

Since $R_j \rightarrow R$ in the flat norm topology in compact subsets of $\R^{n+1} \setminus \K$, for each sufficiently large $j$ there exists $A_j \in \mathbf{I}_{m+1}(W \cap \{d_{\K} > \delta/8\})$ and $B_j \in \mathbf{I}_m(W \cap \{d_{\K}> \delta/8\})$ such that 
\begin{gather} 
	\label{conv eqn2} R - R_j = \partial A_j + B_j \text{ in } W \cap \{d_{\K} > \delta/8\}, \\ 
	\label{conv eqn3} \lim_{j \rightarrow \infty} \big( \Mass_W(A_j) + \Mass_W(B_j) \big) = 0.  
\end{gather}
Notice that by \eqref{conv eqn2}, \eqref{conv nonconbdry hyp}, \eqref{conv nonconbdry eqn}, and \eqref{conv eqn3}, 
\begin{align} \label{conv eqn4} 
	\Mass_W((\partial A_j) \llcorner \{\delta/8 < d_{\K} < \delta/2\}) 
		&\leq \Mass_W(R_j \llcorner \{d_{\K_j} < \delta\}) + \Mass_W(R \llcorner \{d_{\K} < \delta/2\}) + \Mass_W(B_j) 
		\\&\leq 3 \varepsilon \nonumber 
\end{align}
for all sufficiently large $j$.  

We want to replace $A_j$ and $B_j$ with integral currents $A'_j$ and $B'_j$ supported away from $\partial W$.  We can accomplish this by restricting $A_j$ and $B_j$ to a set $W_{\alpha}$ for some $\alpha$.  In particular, by \eqref{conv eqn3} after passing to a subsequence 
\begin{equation*}
	\Mass_W(A_j) \leq 8^{-j}
\end{equation*} 
and thus by slicing theory using \eqref{meas_good_slices} with $\vartheta = 4^{-j}$ there exists $\alpha \in (0,1)$ such that $A_j \llcorner W_{\alpha} \in \mathbf{I}_{m+1}(\R^{n+1} \cap \{d_{\K} > \delta/8\})$, $(\partial A_j) \llcorner W_{\alpha}, \langle A_j, \phi, \alpha \rangle \in \mathbf{I}_m(\R^{n+1} \cap \{d_{\K} > \delta/8\})$, and 
\begin{gather} 
	\label{conv eqn5} \partial (A_j \llcorner W_{\alpha}) = (\partial A_j) \llcorner W_{\alpha} - \langle A_j, \phi, \alpha \rangle 
		\text{ in } \R^{n+1} \cap \{d_{\K} > \delta/8\}, \\
	\label{conv eqn6} \Mass(\langle A_j, \phi, \alpha \rangle) \leq 4^j \,\Mass_W(A_j) \leq 2^{-j} \rightarrow 0, \\
	\label{conv eqn7} \Mass(R_j \llcorner \partial W_{\alpha}) = \Mass(R \llcorner \partial W_{\alpha}) = 0.  
\end{gather}
For each $j$, let $A'_j \in \mathbf{I}_{m+1}(\R^{n+1} \cap \{d_{\K} > \delta/8\})$ and $B'_j \in \mathcal{I}_{m,{\rm loc}}(\R^{n+1} \cap \{d_{\K} > \delta/8\})$ be given by 
\begin{equation*}
	A'_j = A_j \llcorner W_{\alpha}, \quad B'_j = B_j \llcorner W_{\alpha} + \langle A_j, \phi, \alpha \rangle 
\end{equation*}
in $\R^{n+1} \cap \{d_{\K} > \delta/8\}$.  Then by \eqref{conv eqn2}, \eqref{conv eqn5}, \eqref{conv eqn3}, and \eqref{conv eqn6},
\begin{gather} 
	\label{conv eqn8} R \llcorner W_{\alpha} - R_j \llcorner W_{\alpha} = \partial A'_j + B'_j \text{ in } \R^{n+1} \cap \{d_{\K} > \delta/8\}, \\ 
	\label{conv eqn9} \lim_{j \rightarrow \infty} \big( \Mass(A'_j) + \Mass(B'_j) \big) = 0.  
\end{gather}
Moreover, by \eqref{conv eqn5}, \eqref{conv eqn4}, and \eqref{conv eqn6}, 
\begin{align} \label{conv eqn10} 
	\Mass((\partial A'_j) \llcorner \{\delta/8 < d_{\K} < \delta/2\}) 
	&\leq \Mass_W((\partial A_j) \llcorner \{\delta/8 < d_{\K} < \delta/2\}) + \Mass(\langle A_j, \phi, \alpha \rangle) 
	\leq 4 \,\varepsilon.   
\end{align}

Notice that $X$ lies in $\R^{n+1} \setminus \K$ with a boundary in $\K$ which might lie outside $\K_j$ and might not be rectifiable, and similarly for $A'_j$.  By slicing and stretching $A'_j$ and $X$ up to $\K_j$, we want to approximate the integral currents $A'_j$ and $X$ by integral currents $A''_j$ and $X_j$ of $\R^{n+1} \setminus \K_j$.  In particular, by slicing theory again using \eqref{meas_good_slices} with $\vartheta = 1/8$, \eqref{conv eqn1}, \eqref{conv eqn9}, and \eqref{conv eqn10}, there exists $\delta_j \in (\delta/8,\delta/4)$ such that for all sufficiently large $j$ we have $A'_j \llcorner \{ d_{\K} > \delta_j \} \in \mathbf{I}_{m+1}(\R^{n+1})$, $\langle A'_j, d_{\K}, \delta_j \rangle, \,(\partial A'_j) \llcorner \{ d_{\K} > \delta_j \} \in \mathbf{I}_m(\R^{n+1})$, $\langle \partial A'_j, d_{\K}, \delta_j \rangle \in \mathbf{I}_{m-1}(\R^{n+1})$, $X \llcorner \{ d_{\K} > \delta_j \} \in \mathbf{I}_{m+1}(\R^{n+1})$, $\langle X, d_{\K}, \delta_j \rangle, \,(\partial X) \llcorner \{ d_{\K} > \delta_j \} \in \mathbf{I}_m(\R^{n+1})$, $\langle \partial X, d_{\K}, \delta_j \rangle \in \mathbf{I}_{m-1}(\R^{n+1})$, 
\begin{gather} \label{conv eqn11} 
	\partial (A'_j \llcorner \{d_{\K} > \delta_j\}) = (\partial A'_j)  \llcorner \{d_{\K} > \delta_j\} - \langle A'_j, d_{\K}, \delta_j \rangle , \\
	\partial (X_j \llcorner \{d_{\K} > \delta_j\}) = (\partial X_j)  \llcorner \{d_{\K} > \delta_j\} - \langle X_j, d_{\K}, \delta_j \rangle , \nonumber \\
	\partial \langle A'_j, d_{\K}, \delta_j \rangle = -\langle \partial A'_j, d_{\K}, \delta_j \rangle, \quad 
	\partial \langle X_j, d_{\K}, \delta_j \rangle = -\langle \partial X_j, d_{\K}, \delta_j \rangle \nonumber
\end{gather}
in $\R^{n+1}$, and  
\begin{equation} \label{conv eqn12} 
	\max\{ \Mass(\langle A'_j, d_{\K}, \delta_j \rangle),\, \Mass(\langle \partial A'_j, d_{\K}, \delta_j \rangle),\, 
		\Mass(\langle X, d_{\K}, \delta_j \rangle),\, \Mass(\langle \partial X, d_{\K}, \delta_j \rangle) \}
		\leq \frac{256 \,\varepsilon}{\delta}.  
\end{equation}
For each $j$, let $h_j : [0,1] \times \R^{n+1} \setminus \K_j \rightarrow \R^{n+1} \setminus \K_j$ be the homotopy given by $h_j(t,x) = (1-t) \,x + t \,\xi_{\K_j}(x)$.  For each sufficiently large $j$ let $A''_j \in \mathbf{I}_{m+1}(\R^{n+1})$ and $X_j \in \mathbf{I}_{m+1}(\R^{n+1})$ be given by 
\begin{align} \label{conv eqn13} 
	A''_j &= A'_j \llcorner \{d_{\K} > \delta_j\} + h_{j \#}(\llbracket 0,1 \rrbracket \times \langle A'_j, d_{\K}, \delta_j \rangle) , \\
	X_j &= X \llcorner \{d_{\K} > \delta_j\} + h_{j \#}(\llbracket 0,1 \rrbracket \times \langle X, d_{\K}, \delta_j \rangle) . \nonumber
\end{align}
By \eqref{conv eqn9} and \eqref{conv eqn12},  
\begin{equation} \label{conv eqn14} 
	\Mass(A''_j) 
	\leq \Mass(A'_j) + \delta \,\Mass(\langle A'_j, d_{\K}, \delta_j \rangle) 
	\leq 257 \,\varepsilon 
\end{equation}
and similarly by \eqref{conv eqn12}, 
\begin{equation} \label{conv eqn15} 
	\Mass(X_j) 
	\leq \Mass(X) + \delta \,\Mass(\langle X_j, d_{\K}, \delta_j \rangle) \\
	\leq \Mass(X) + 256 \,\varepsilon .  
\end{equation}
By the homotopy formula and \eqref{conv eqn11},   
\begin{equation} \label{conv eqn16} 
	\partial h_{j \#}(\llbracket 0,1 \rrbracket \times \langle A'_j, d_{\K}, \delta_j \rangle) 
	= \langle A'_j, d_{\K}, \delta_j \rangle + h_{j \#}(\llbracket 0,1 \rrbracket \times \langle \partial A'_j, d_{\K}, \delta_j \rangle) 
\end{equation}
in $\R^{n+1} \setminus \K_j$.  By taking the boundary of $A''_j$ as defined in \eqref{conv eqn13} using \eqref{conv eqn11} and \eqref{conv eqn16}, 
\begin{equation*} 
	\partial A''_j = (\partial A'_j)  \llcorner \{d_{\K} > \delta_j\} + h_{j \#}(\llbracket 0,1 \rrbracket \times \langle \partial A'_j, d_{\K}, \delta_j \rangle) 
\end{equation*}
in $\R^{n+1} \setminus \K_j$.  Thus by \eqref{conv eqn10} and \eqref{conv eqn12}, 
\begin{align} \label{conv eqn17} 
	\Mass_{\R^{n+1} \setminus \K_j}((\partial A''_j) \llcorner \{d_{\K} < \delta/2\}) 
	&\leq \Mass((\partial A'_j) \llcorner \{\delta/4 < d_{\K} < \delta/2\}) + \delta \,\Mass(\langle \partial A'_j, d_{\K}, \delta_j \rangle) 
	\\&\leq 260 \,\varepsilon . \nonumber 
\end{align}
Similarly, using  \eqref{conv eqn11}, \eqref{conv eqn1}, and \eqref{conv eqn12},
\begin{equation} \label{conv eqn18} 
	\Mass_{\R^{n+1} \setminus \K_j}((\partial X_j) \llcorner \{d_{\K} < \delta/2\}) 
	\leq \Mass((\partial X) \llcorner \{0 < d_{\K} < \delta/2\}) + \delta \,\Mass(\langle \partial X, d_{\K}, \delta_j \rangle) 
	\leq 257 \,\varepsilon .  
\end{equation}

Let $0 < \beta < \alpha$.  Since $R_j$ satisfies the almost relatively area minimizing property \eqref{conv aam hyp}, we can compare $R_j$ to $R_j + \partial A''_j + \partial X_j$ to obtain 
\begin{equation*}
	\Mass_{W_{\beta}}(R_j) \leq \Mass_{W_{\beta} \setminus \K_j}(R_j + \partial A''_j + \partial X_j) + \lambda_j \,\Mass(A''_j + X_j)
\end{equation*}
for all sufficiently large $j$.  By \eqref{conv eqn8} and $A''_j = A'_j$ and $X_j = X$ on $\R^{n+1} \cap \{d_{\K} > \delta/4 \}$ for large $j$, 
\begin{align*}
	\Mass_{W_{\beta}}(R_j) 
	&\leq \Mass_{W_{\beta} \cap \{d_{\K} > \delta/4\}}(R_j + \partial A'_j + B'_j + \partial X) 
		+ \Mass_{W \cap \{d_{\K} < \delta/2\}}(R_j) 
		\\&\hspace{5mm} + \Mass_{(\R^{n+1} \setminus \K_j) \cap \{d_{\K} < \delta/2\}}(\partial A''_j) 
		+ \Mass(B'_j) + \Mass_{(\R^{n+1} \setminus \K_j) \cap \{d_{\K} < \delta/2\}}(\partial X_j) 
		\\&\hspace{5mm} + \lambda_j \,\Mass(A''_j) + \lambda_j \,\Mass(X_j)
	\\&\leq \Mass_{W_{\beta} \setminus \K}(R + \partial X) + \Mass_{W \cap \{d_{\K} < \delta/2\}}(R_j) 
		+ \Mass_{(\R^{n+1} \setminus \K_j) \cap \{d_{\K} < \delta/2\}}(\partial A''_j) 
		\\&\hspace{5mm} + \Mass(B'_j) + \Mass_{(\R^{n+1} \setminus \K_j) \cap \{d_{\K} < \delta/2\}}(\partial X_j) 
		+ \lambda_j \,\Mass(A''_j) + \lambda_j \,\Mass(X_j). 
\end{align*}
By the mass bounds \eqref{conv nonconbdry hyp}, \eqref{conv eqn17}, \eqref{conv eqn9}, \eqref{conv eqn18}, \eqref{conv eqn14}, and \eqref{conv eqn15}, 
\begin{equation*}
	\Mass_{W_{\beta}}(R_j) \leq \Mass_{W_{\beta} \setminus \K}(R + \partial X) + \lambda_j \,\Mass(X) + 1032 \,\varepsilon 
\end{equation*}
for $j$ sufficiently large.  By letting $\beta \downarrow \alpha$ noting \eqref{conv eqn7},  
\begin{equation*}
	\Mass_{W_{\alpha}}(R_j) \leq \Mass_{W_{\alpha} \setminus \K}(R + \partial X) + \lambda_j \,\Mass(X) + 1032 \,\varepsilon 
\end{equation*}
for $j$ sufficiently large.  By letting $j \rightarrow \infty$, noting that $\lambda_j \rightarrow 0$ and $\varepsilon$ is arbitrary, we obtain 
\begin{equation} \label{conv eqn19}
	\limsup_{j \rightarrow \infty} \Mass_{W_{\alpha}}(R_j) \leq \Mass_{W_{\alpha} \setminus \K}(R + \partial X) 
\end{equation}
for all $X \in \mathbf{I}_{m+1}(\R^{n+1} \setminus \K)$ with $\op{spt} X \subseteq K$. 

By \eqref{conv eqn19} and the semi-continuity of mass, 
\begin{equation*}
	\Mass_{W_{\alpha}}(R) \leq \Mass_{W_{\alpha} \setminus \K}(R + \partial X) 
\end{equation*}
for all $X \in \mathbf{I}_{m+1}(\R^{n+1} \setminus \K)$ with $\op{spt} X \subseteq K$.  Since $X$ is arbitrary, $R$ is relatively area minimizing in $\R^{n+1} \setminus \K$.

By setting $X = 0$ in \eqref{conv eqn19} and using $K \subset W_{\alpha} \subset W$, 
\begin{equation*}
	\limsup_{j \rightarrow \infty} \|R_j\|(K) \leq \|R\|(W) 
\end{equation*}
for every compact subset $K \subset \R^{n+1}$ and bounded open subset $W \subset \R^{n+1}$ such that $K \subset W$.  Letting $W$ decrease to $K$, 
\begin{equation} \label{conv eqn20}
	\limsup_{j \rightarrow \infty} \|R_j\|(K) \leq \|R\|(K) 
\end{equation}
for every compact subset $K \subset \R^{n+1}$.  By the semi-continuity of mass, 
\begin{equation} \label{conv eqn21}
	\|R\|(W) \leq \liminf_{j \rightarrow \infty} \|R_j\|(W) 
\end{equation}
for every bounded open set $W \subset \R^{n+1}$.  By \eqref{conv eqn20} and \eqref{conv eqn21} and a standard approximation argument, $\|R_j\| \rightarrow \|R\|$ in the sense of Radon measures.  Note that we only showed $\|R_j\| \rightarrow \|R\|$ for a subsequence of $\{R_j\}$, but by repeating the argument starting with any subsequence of $\{R_j\}$ we obtain $\|R_j\| \rightarrow \|R\|$ for the original sequence $\{R_j\}$.  
\end{proof}

\section{Existence of relative isoperimetric minimizers} \label{sec:existence_sec}

This section is focused on the proof of the existence result Theorem B from the introduction.  Our approach will involve a nontrivial application of the direct method as discussed in Subsection \ref{sec:intro method} of the introduction. 

\begin{proof}[Proof of Theorem B]  Let $\K$ be a proper convex subset of $\R^{n+1}$ with $C^2$-boundary.  Let $T_j \in \mathbf{I}_m(\R^{n+1} \setminus \K)$ and $R_j \in \mathbf{I}_m(\R^{n+1} \setminus \K)$ such that  $\partial R_j = T_j$ in $\R^{n+1} \setminus \K$, $R_j$ is relatively area minimizing in $\R^{n+1} \setminus \K$, and 
\begin{equation} \label{existence_eqn1}
	\gamma_{m,n}(\K) = \lim_{j \rightarrow \infty} \frac{\Mass(T_j)^{\frac{m+1}{m}}}{\Mass(R_j)}. 
\end{equation}
Our aim is to let $T_j$ and $R_j$ converge to currents $T_0$ and $R_0$ such that $(T_0,R_0)$ is a relative isoperimetric minimizer.  However, in order to obtain nonzero limits of $T_j$ and $R_j$, we need to control the masses of $T_j$ and $R_j$.  Thus after passing to a subsequence, we may assume that one of the following three cases holds true: 
\begin{enumerate}
	\item[(a)]  $\lim_{j \rightarrow \infty} \Mass(R_j) = \infty$, 
	\item[(b)]  $\lim_{j \rightarrow \infty} \Mass(R_j) = 0$, and 
	\item[(c)]  $\lim_{j \rightarrow \infty} \Mass(R_j)$ exists as a positive real number in $(0,\infty)$. 
\end{enumerate}
The basic idea is to show in the case (a) that after translating and scaling $T_j, R_j$ converge to integral currents of $\R^{n+1}$, contradicting \eqref{existence hyp}.   In the case (b), after translating and scaling $T_j, R_j$ converge to integral currents of a half-space, again contradicting \eqref{existence hyp}.   In case (c), we let $T_j \rightarrow T_0$ and $R_j \rightarrow R_0$ where $(T_0,R_0)$ is the desired relative isoperimetric minimizer.  The details are as follows. \\

\noindent \textbf{Case (a).}  Suppose $\lim_{j \rightarrow \infty} \Mass(R_j) = \infty$.  After translating assume $0 \in \K$.  Rescale letting 
\begin{equation*}
	r_j = \left(\frac{\Mass(R_j)}{\mathcal{H}^{m+1}(\ball)}\right)^{\frac{1}{m+1}} \rightarrow \infty, \quad 
	\widetilde{K}_j = \eta_{0,r_j}(K), \quad 
	\widetilde{T}_j = \eta_{0,r_j \#} T_j, \quad 
	\widetilde{R}_j = \eta_{0,r_j \#} R_j
\end{equation*}
(where $\eta_{y,r}(x) = (x-y)/r$ for all $x,y \in \R^{n+1}$ and $r > 0$ as in Subsection \ref{sec:prelims_notation}) so that $\widetilde{T}_j \in \mathbf{I}_m(\R^{n+1} \setminus \widetilde{K}_j)$ and $\widetilde{R}_j \in \mathbf{I}_{m+1}(\R^{n+1} \setminus \widetilde{K}_j)$ such that $\widetilde{R}_j$ is relatively area minimizing with $\partial \widetilde{R}_j = \widetilde{T}_j$ in $\R^{n+1} \setminus \widetilde{K}_j$ and by \eqref{existence hyp} and \eqref{existence_eqn1} 
\begin{gather} 
	\label{existence_eqnA1} 0 \in \widetilde{K}_j, \quad \lim_{j \rightarrow \infty} \op{diam}(\widetilde{K}_j) = 0,  \\
	\label{existence_eqnA2} \lim_{j \rightarrow \infty} \Mass(\widetilde{T}_j) < \mathcal{H}^m(\sphere), \quad 
	\Mass(\widetilde{R}_j) = \mathcal{H}^{m+1}(\ball), \\ 
	\label{existence_eqnA3} \gamma_{m,n}(\K) = \lim_{j \rightarrow \infty} \frac{\Mass(\widetilde{T}_j)^{\frac{m+1}{m}}}{\Mass(\widetilde{R}_j)}. 
\end{gather}
(Note that by \eqref{existence_eqnA3} and $\Mass(\widetilde{R}_j) = \mathcal{H}^{m+1}(\ball)$ for all $j$, $\lim_{j \rightarrow \infty} \Mass(\widetilde{T}_j)$ exists.)  By the Federer-Fleming compactness theorem and \eqref{existence_eqnA2}, after passing to a subsequence $\widetilde{T}_j \rightarrow \widetilde{T}_0$ and $\widetilde{R}_j \rightarrow \widetilde{R}_0$ weakly in $\R^{n+1} \setminus \{0\}$ for some $\widetilde{T}_0 \in \mathbf{I}_m(\R^{n+1} \setminus \{0\})$ and $\widetilde{R}_j \in \mathbf{I}_{m+1}(\R^{n+1} \setminus \{0\})$.  Clearly $\partial \widetilde{R}_0 = \widetilde{T}_0$ in $\R^{n+1} \setminus \{0\}$.  By the semi-continuity of mass and \eqref{existence_eqnA2}, 
\begin{equation} \label{existence_eqnA4} 
	\Mass(\widetilde{T}_0) \leq \lim_{j \rightarrow \infty} \Mass(\widetilde{T}_j) < \mathcal{H}^m(\sphere). \\ 
\end{equation}

By Lemma \ref{boundary rectifiability lemma2} and \eqref{existence_eqnA2}, for every $\varepsilon > 0$ and $j$ sufficiently large 
\begin{align} 
	&\Mass(\widetilde{T}_j \llcorner \{d_{\widetilde{K}_j} < s\}) \leq C(m) \,s \,\Mass(\widetilde{T}_j)^{\frac{m-1}{m}} \leq C(m) \,s, \nonumber \\
	\label{existence_eqnA5} &\Mass(\widetilde{R}_j \llcorner \{d_{\widetilde{K}_j} < s\}) \leq C(m) \,s \,\Mass(\widetilde{T}_j) \leq C(m) \,s  
\end{align}
for all $\varepsilon \leq s \leq s_0$, where $s_0 = s_0(m) \in (0,\infty)$ is a constant.  By \eqref{existence_eqnA1} and the semi-continuity of mass, 
\begin{align} 
	\label{existence_eqnA6} &\Mass(\widetilde{T}_0 \llcorner B_s(0)) \leq C(m) \,s, \\ 
	\label{existence_eqnA7} &\Mass(\widetilde{R}_0 \llcorner B_s(0)) \leq C(m) \,s 
\end{align}
for all $0 < s \leq s_0$.  By Lemma \ref{nonconcentration lemma}, \eqref{existence_eqnA1}, and \eqref{existence_eqnA2}, for every $\varepsilon > 0$ there exists $\rho = \rho(m,\varepsilon) > 0$ such that for $j$ sufficiently large 
\begin{equation} \label{existence_eqnA8} 
	\Mass(\widetilde{R}_j \llcorner \R^{n+1} \setminus B_{\rho}(0)) \leq \varepsilon. 
\end{equation}
By Lemma \ref{convergence lemma}, \eqref{existence_eqnA2}, \eqref{existence_eqnA5}, and \eqref{existence_eqnA8}, $\widetilde{R}_0$ is relatively area-minimizing in $\R^{n+1} \setminus \{0\}$ and 
\begin{equation} \label{existence_eqnA9} 
	\Mass(\widetilde{R}_0) = \lim_{j \rightarrow \infty} \Mass(\widetilde{R}_j) = \mathcal{H}^{m+1}(\ball). 
\end{equation}

We want to show that $\widetilde{T}_0 \in \mathbf{I}_m(\R^{n+1})$ and $\widetilde{R}_0 \in \mathbf{I}_{m+1}(\R^{n+1})$ with $\partial \widetilde{R}_0 = \widetilde{T}_0$ in $\R^{n+1}$.  It is possible that the boundaries of $\widetilde{T}_0$ or $\widetilde{R}_0$, as currents of $\R^{n+1}$, have infinite mass at the origin.  However, we can rule this out using the non-concentration estimates \eqref{existence_eqnA6} and \eqref{existence_eqnA7}.  In particular, by slicing theory with $\vartheta = 1/4$, \eqref{existence_eqnA6}, and \eqref{existence_eqnA7} for each sufficiently large integer $j$ there exists $s_j \in (2^{-j-1},2^{-j})$ such that $\widetilde{T}_0 \llcorner \R^{n+1} \setminus B_{s_j}(0) \in \mathbf{I}_m(\R^{n+1})$, $\langle \widetilde{T}_0, |\cdot|, s_j \rangle \in \mathbf{I}_{m-1}(\R^{n+1})$, $\widetilde{R}_0 \llcorner \R^{n+1} \setminus B_{s_j}(0) \in \mathbf{I}_{m+1}(\R^{n+1})$, and $\langle \widetilde{R}_0, |\cdot|, s_j \rangle \in \mathbf{I}_m(\R^{n+1})$ with 
\begin{gather} 
	\label{existence_eqnA10} \partial (\widetilde{T}_0 \llcorner \R^{n+1} \setminus B_{s_j}(0)) = -\langle \widetilde{T}_0, |\cdot|, s_j \rangle 
		\text{ in } \R^{n+1}, \\ 
	\partial (\widetilde{R}_0 \llcorner \R^{n+1} \setminus B_{s_j}(0)) = \widetilde{T}_0 \llcorner \R^{n+1} \setminus B_{s_j}(0) 
		- \langle \widetilde{R}_0, |\cdot|, s_j \rangle \text{ in } \R^{n+1}, \nonumber \\ 
	\Mass(\langle \widetilde{T}_0, |\cdot|, s_j \rangle) \leq C(m), \quad 
	\Mass(\langle \widetilde{R}_0, |\cdot|, s_j \rangle) \leq C(m). \nonumber 
\end{gather}
By \eqref{existence_eqnA4} and \eqref{existence_eqnA9}, $\widetilde{T}_0 \llcorner \R^{n+1} \setminus B_{s_j}(0) \rightarrow \widetilde{T}_0$ and $\widetilde{R}_0 \llcorner \R^{n+1} \setminus B_{s_j}(0) \rightarrow \widetilde{R}_0$ in the mass norm topology in $\R^{n+1}$.  Moreover, by the Federer-Fleming compactness theorem and \eqref{existence_eqnA4}, \eqref{existence_eqnA9}, and \eqref{existence_eqnA10}, after passing to a subsequence $\widetilde{T}_0 \llcorner \R^{n+1} \setminus B_{s_j}(0)$ and $\widetilde{R}_0 \llcorner \R^{n+1} \setminus B_{s_j}(0)$ converge weakly to integral currents of $\R^{n+1}$ as $j \rightarrow \infty$.  Therefore, $\widetilde{T}_0 \in \mathbf{I}_m(\R^{n+1})$ and $\widetilde{R}_0 \in \mathbf{I}_{m+1}(\R^{n+1})$.  It follows that since $\mathcal{H}^m(\{0\}) = 0$, $\partial \widetilde{R}_0 = \widetilde{T}_0$ in $\R^{n+1}$.  Since $\widetilde{R}_0 \in \mathbf{I}_{m+1}(\R^{n+1})$ and $\widetilde{R}_0$ is relatively area minimizing in $\R^{n+1} \setminus \{0\}$, $\widetilde{R}_0$ is area minimizing in $\R^{n+1}$. 

By \eqref{existence_eqnA3}, \eqref{existence_eqnA4}, \eqref{existence_eqnA9}, and \eqref{existence hyp} 
\begin{equation*} 
	\frac{\Mass(\widetilde{T}_0)^{\frac{m+1}{m}}}{\Mass(\widetilde{R}_0)} 
	\leq \lim_{j \rightarrow \infty} \frac{\Mass(\widetilde{T}_j)^{\frac{m+1}{m}}}{\Mass(\widetilde{R}_j)} = \gamma_{m,n}(\K) 
	< \frac{\mathcal{H}^m(\sphere)^{\frac{m+1}{m}}}{\mathcal{H}^{m+1}(\ball)} , 
\end{equation*}
contradicting the isoperimetric inequality in $\R^{n+1}$. \\

\noindent \textbf{Case (b).}  Suppose $\lim_{j \rightarrow \infty} \Mass(R_j) = 0$.  Let $c = c(m,n,\K) \in (0,\infty)$ and $\beta_1 = \beta_1(m,n,\K) \in (0,\infty)$ be as in Lemma \ref{local concentration lemma} and $\beta_0 = \beta_0(m,c/2) \in [1,\infty)$ be as in Lemma \ref{nonconcentration lemma} with $\sigma = c/2$.  Let $\beta \geq \max\{\beta_0,\beta_1\}$.  By Lemma \ref{local concentration lemma} and \eqref{existence_eqn1}, for each sufficiently large $j$ there exists $y_j \in \R^{n+1} \setminus \K$ such that 
\begin{equation} \label{existence_eqnB1}
	\Mass(T_j \llcorner B_{\beta \,\Mass(R_j)^{\frac{1}{m+1}}}(y_j)) \geq c\,\Mass(T_j), 
\end{equation}
where we used \eqref{existence_eqn1} to get $\Mass(T_j) \leq C(m) \,\Mass(R_j)^{\frac{m}{m+1}}$ for large $j$.  By Lemma \ref{nonconcentration lemma} we must have $\op{dist}(y_j,\partial \K) \leq 2 \beta \,\Mass(R_j)^{\frac{1}{m+1}}$ as otherwise \eqref{noncon concl1} holds true with $\sigma = c/2$, $T = T_j$ and $R = R_j$, contradicting \eqref{existence_eqnB1}.  Hence there exists $x_j \in \partial \K$ such that $|x_j - y_j| \leq 2 \beta \,\Mass(R_j)^{\frac{m}{m+1}}$ and thus 
\begin{equation} \label{existence_eqnB2}
	\Mass(T_j \llcorner B_{3\beta \,\Mass(R_j)^{\frac{1}{m+1}}}(x_j)) \geq c\,\Mass(T_j) .
\end{equation}

Translate and rescale letting 
\begin{equation*}
	r_j = \left(\frac{2\,\Mass(R_j)}{\mathcal{H}^{m+1}(\ball)}\right)^{\frac{1}{m+1}} \rightarrow 0, \quad 
	\widetilde{K}_j = \eta_{x_j,r_j}(K), \quad 
	\widetilde{T}_j = \eta_{x_j,r_j \#} T_j, \quad 
	\widetilde{R}_j = \eta_{x_j,r_j \#} R_j 
\end{equation*}
so that $\widetilde{T}_j \in \mathbf{I}_m(\R^{n+1} \setminus \widetilde{K}_j)$ and $\widetilde{R}_j \in \mathbf{I}_{m+1}(\R^{n+1} \setminus \widetilde{K}_j)$ such that $\widetilde{R}_j$ is relatively area minimizing with $\partial \widetilde{R}_j = \widetilde{T}_j$ in $\R^{n+1} \setminus \widetilde{K}_j$ and by \eqref{existence_eqn1} and \eqref{existence hyp} 
\begin{gather}
	\label{existence_eqnB3} \lim_{j \rightarrow \infty} \Mass(\widetilde{T}_j) < \frac{1}{2} \,\mathcal{H}^m(\sphere), \quad 
	\Mass(\widetilde{R}_j) = \frac{1}{2} \,\mathcal{H}^{m+1}(\ball), \\ 
	\label{existence_eqnB4} \gamma_{m,n}(\K) = \lim_{j \rightarrow \infty} \frac{\Mass(\widetilde{T}_j)^{\frac{m+1}{m}}}{\Mass(\widetilde{R}_j)}. 
\end{gather}
Also, by \eqref{existence_eqnB2}, 
\begin{equation} \label{existence_eqnB5}
	\Mass(\widetilde{T}_j \llcorner B_{\widetilde{\beta}}(0)) \geq \widetilde{c} > 0
\end{equation}
where $\widetilde{\beta} = \widetilde{\beta}(m,n,\K) \in (0,\infty)$ and $\widetilde{c} = \widetilde{c}(m,n,\K) \in (0,\infty)$ are constants.  Since $\K$ has a $C^2$-boundary, there exists a half-space $\widetilde{\K}_0$ with $0 \in \partial \widetilde{\K}_0$ such that $\widetilde{K}_j \rightarrow \widetilde{\K}_0$ locally in Hausdorff distance and $\partial \widetilde{K}_j \rightarrow \partial \widetilde{\K}_0$ locally in the $C^2$ topology.  After an orthogonal change of coordinates, assume $\widetilde{\K}_0 = \R^{n+1}_+$.  By the Federer-Fleming compactness theorem and \eqref{existence_eqnB3}, after passing to a subsequence, $\widetilde{T}_j \rightarrow \widetilde{T}_0$ and $\widetilde{R}_j \rightarrow \widetilde{R}_0$ weakly in $\R^{n+1}_+$ for some $\widetilde{T}_0 \in \mathbf{I}_m(\R^{n+1}_+)$ and $\widetilde{R}_0 \in \mathbf{I}_{m+1}(\R^{n+1}_+)$ such that $\partial \widetilde{R}_0 = \widetilde{T}_0$ in $\R^{n+1}_+$ and by the semi-continuity of mass 
\begin{equation} \label{existence_eqnB6} 
	\Mass(\widetilde{T}) \leq \lim_{j \rightarrow \infty} \Mass(\widetilde{T}_j) < \frac{1}{2} \,\mathcal{H}^m(\sphere). 
\end{equation}

By Lemma \ref{boundary rectifiability lemma2} and \eqref{existence_eqnB3}, for every $\varepsilon > 0$ and $j$ sufficiently large 
\begin{equation} \label{existence_eqnB7}
	\Mass(\widetilde{R}_j \llcorner \{d_{\widetilde{K}_j} < s\}) \leq C(m) \,s \,\Mass(\widetilde{T}_j) \leq C(m) \,s  
\end{equation}
for all $\varepsilon \leq s \leq s_0$, where $s_0 = s_0(m) \in (0,\infty)$ is a constant.  By Lemma \ref{nonconcentration lemma}, \eqref{existence_eqnB3}, and \eqref{existence_eqnB5}, for every $\varepsilon > 0$ there exists $\rho = \rho(m,n,\K,\varepsilon) > 0$ such that 
\begin{equation} \label{existence_eqnB8} 
	\limsup_{j \rightarrow \infty} \Mass(\widetilde{R}_j \llcorner \R^{n+1} \setminus B_{\rho}(0)) \leq \varepsilon 
\end{equation}
for all large $j$.  By Lemma \ref{convergence lemma}, \eqref{existence_eqnB3}, \eqref{existence_eqnB7}, and \eqref{existence_eqnB8}, $\widetilde{R}_0$ is relatively area minimizing in $\R^{n+1}_+$ and 
\begin{equation} \label{existence_eqnB9} 
	\Mass(\widetilde{R}_0) = \lim_{j \rightarrow \infty} \Mass(\widetilde{R}_j) = \frac{1}{2} \,\mathcal{H}^{m+1}(\ball). 
\end{equation}

By \eqref{existence_eqnB4}, \eqref{existence_eqnB6}, \eqref{existence_eqnB9}, and \eqref{existence hyp}, 
\begin{equation*} 
	\frac{\Mass(\widetilde{T}_0)^{\frac{m+1}{m}}}{\Mass(\widetilde{R}_0)} 
	\leq \lim_{j \rightarrow \infty} \frac{\Mass(\widetilde{T}_j)^{\frac{m+1}{m}}}{\Mass(\widetilde{R}_j)} = \gamma_{m,n}(\K) 
	< 2^{-\frac{1}{m}} \,\frac{\mathcal{H}^m(\sphere)^{\frac{m+1}{m}}}{\mathcal{H}^{m+1}(\ball)} , 
\end{equation*}
contradicting the sharp isoperimetric inequality in $\R^{n+1}_+$. \\

\noindent \textbf{Case (c).}  Suppose $\lim_{j \rightarrow \infty} \Mass(R_j)$ exists and 
\begin{equation} \label{existence_eqnC1} 
	0 < L = \lim_{j \rightarrow \infty} \Mass(R_j) < \infty. 
\end{equation}
We do not need to scale as by assumption $R_j$ satisfies \eqref{existence_eqnC1} and by \eqref{existence_eqnC1}, \eqref{existence_eqn1}, and \eqref{existence hyp} 
\begin{equation} \label{existence_eqnC2} 
	\lim_{j \rightarrow \infty} \Mass(T_j) \leq C(m) \,L^{\frac{m}{m+1}} < \infty. 
\end{equation}
By the Federer-Fleming compactness theorem, \eqref{existence_eqnC1}, and \eqref{existence_eqnC2}, after passing to a subsequence $T_j \rightarrow T_0$ and $R_j \rightarrow R_0$ weakly in $\R^{n+1} \setminus \K$ for some $T_0 \in \mathbf{I}_m(\R^{n+1} \setminus \K)$ and $R_0 \in \mathbf{I}_{m+1}(\R^{n+1} \setminus \K)$ such that $\partial R_0 = T_0$ in $\R^{n+1} \setminus \K$ and by the semi-continuity of mass 
\begin{equation} \label{existence_eqnC3}
	\Mass(T_0) \leq \lim_{j \rightarrow \infty} \Mass(T_j). 
\end{equation}

By Lemma \ref{boundary rectifiability lemma2} and \eqref{existence_eqnC1}, for every $\varepsilon > 0$ and $j$ sufficiently large 
\begin{equation} \label{existence_eqnC4}
	\Mass(R_j \llcorner \{d_{\K} < s\}) \leq C(m) \,s \,\Mass(R_j) \leq C(m) \,L \,s 
\end{equation}
for all $\varepsilon \leq s \leq s_0$, where $s_0 = s_0(m,L) \in (0,\infty)$ is a constant.  By Lemma \ref{nonconcentration lemma} and \eqref{existence_eqnC1}, noting that $\K$ is a fixed bounded set, for every $\varepsilon > 0$ there exists $\rho = \rho(m,\K,L,\varepsilon) > 0$ such that 
\begin{equation} \label{existence_eqnC5} 
	\limsup_{j \rightarrow \infty} \Mass(R_j \llcorner \R^{n+1} \setminus B_{\rho}(0)) \leq \varepsilon 
\end{equation}
for all large $j$.  By Lemma \ref{convergence lemma}, \eqref{existence_eqnC1}, \eqref{existence_eqnC2}, \eqref{existence_eqnC4}, and \eqref{existence_eqnC5}, $R_0$ is relatively area minimizing in $\R^{n+1} \setminus \K$ and 
\begin{equation} \label{existence_eqnC6} 
	\Mass(R_0) = \lim_{j \rightarrow \infty} \Mass(R_j) . 
\end{equation}

By \eqref{existence_eqn1}, \eqref{existence_eqnC3}, and \eqref{existence_eqnC6}, 
\begin{equation*} 
	\frac{\Mass(T_0)^{\frac{m+1}{m}}}{\Mass(R_0)} 
	\leq \lim_{j \rightarrow \infty} \frac{\Mass(T_j)^{\frac{m+1}{m}}}{\Mass(R_j)} = \gamma_{m,n}(\K), 
\end{equation*}
which by the minimality of $\gamma_{m,n}(\K)$ implies 
\begin{equation*} 
	\frac{\Mass(T_0)^{\frac{m+1}{m}}}{\Mass(R_0)} = \gamma_{m,n}(\K). \qedhere
\end{equation*}
\end{proof}

\section{First variational of the isoperimetric ratio} \label{sec:variation sec}

In this section we compute the first variation of relative area minimizers and relative isoperimetric minimizers.  This follows from the standard first variational formula for area and a first variational formula for the isoperimetric ratio from~\cite{Alm83}.  We only consider variational vectors fields which are tangent to $\partial \K$ so that when we vary $T$ and $R$ they remain currents of $\R^{n+1} \setminus \K$.  Note that these first variational formulas hold true whenever $\K$ is the closure of an open set with $C^2$-boundary and do not require convexity. 

\begin{lemma} \label{first variation R lemma}
Let $\K$ be the closure of an open subset of $\R^{n+1}$ with $C^2$-boundary.  Let $T \in \mathbf{I}_m(\R^{n+1} \setminus \K)$ and $R \in \mathbf{I}_{m+1}(\R^{n+1} \setminus \K)$ such that $R$ is relatively area minimizing with $\partial R = T$ in $\R^{n+1} \setminus \K$.  Then $R$ has zero distributional mean curvature in the sense that 
\begin{equation} \label{first variation R concl1}
	\int \op{div}_R \zeta(x) \,d\|R\|(x) = 0
\end{equation}
for all $\zeta \in C^1_c(\R^{n+1} \setminus \op{spt} T;\R^{n+1})$ such that $\zeta(x)$ is tangent to $\partial \K$ at each $x \in \partial \K$, where $\op{div}_R \zeta(x)$ denotes the divergence of $\zeta$ computed with respect to the approximate tangent plane of $R$ at $x$ for $\|R\|$-a.e.~$x$. 
\end{lemma}

\begin{proof}
We want to let $\varepsilon > 0$, $\Phi : (-\varepsilon,\varepsilon) \times \R^{n+1} \rightarrow \R^{n+1}$, and $\Phi_t = \Phi(t,\cdot)$ such that $\Phi_t$ is a one-parameter family of $C^1$-diffeomorphisms with $\Phi(0,x) = x$ and $\tfrac{\partial \Phi}{\partial t}(0,x) = \zeta(x)$ for all $x \in \R^{n+1}$ and $\Phi_t(\R^{n+1} \setminus \K) = \R^{n+1} \setminus \K$ for all $t \in (-\varepsilon,\varepsilon)$.  Notice that if $\op{spt} \zeta \cap \partial \K = \emptyset$, we can take $\Phi_t(x) = x + t \,\zeta(x)$.  On the other hand, for each $x_0 \in \partial \K$ there exists $\delta(x_0) > 0$ and a $C^2$-diffeomorphism $\varphi : B_{\delta(x_0)}(x_0) \rightarrow \R^{n+1}$ such that 
\begin{gather*}
	\varphi(B_{\delta(x_0)}(x_0) \setminus \K) = \varphi(B_{\delta(x_0)}(x_0)) \cap \R^{n+1}_+, \\ 
	\varphi(B_{\delta(x_0)}(x_0) \cap \partial \K) = \varphi(B_{\delta(x_0)}(x_0)) \cap \{ x : x_{n+1} = 0 \}.
\end{gather*}
If $\op{spt} \zeta \subseteq B_{\delta(x_0)}(x_0)$, then for some $\varepsilon > 0$ sufficiently small and each $t \in (-\varepsilon,\varepsilon)$ we have the $C^1$-diffeomorphism $\Psi_t(x) = x + t\,\varphi_{\#} \zeta(x)$ for all $x \in \varphi(B_{\delta(x_0)}(x_0))$.  Thus we can take $\Phi_t(x) = (\varphi^{-1} \circ \Psi_t \circ \varphi)(x)$ for all $x \in B_{\delta(x_0)}(x_0)$ and $\Phi_t(x) = x$ for all $x \in \R^{n+1} \setminus B_{\delta(x_0)}(x_0)$.  Notice that $\varphi_{\#} \zeta$ is tangent to $\{ x : x_{n+1} = 0 \}$ and thus $\Phi_t(\R^{n+1} \setminus \K) = \R^{n+1} \setminus \K$ for all $t \in (-\varepsilon,\varepsilon)$.  To construct $\Phi$ for a general vector field $\zeta$, we can find a partition of unity consisting of smooth functions $\chi_j : \R^{n+1} \rightarrow [0,1]$ for $j = 0,1,\ldots,N$ such that $\op{spt} \chi_0 \cap \partial \K = \emptyset$, for each $j = 1,2,\ldots,N$ we have $\op{spt} \chi_j \subseteq B_{\delta(x_j)}(x_j)$ for some $x_j \in \partial \K$ with corresponding radius $\delta(x_j) > 0$ as above, and $\sum_{j=0}^N \chi_j = 1$ on $\R^{n+1}$.  There exist $\varepsilon > 0$ and, for each $j = 0,1,\ldots,N$, $\Phi_j : (-\varepsilon,\varepsilon) \times \R^{n+1} \rightarrow \R^{n+1}$ and $\Phi_{j,t} = \Phi_j(t,\cdot)$ such that $\Phi_{j,t}$ is a one-parameter family of $C^1$-diffeomorphisms with $\Phi_j(0,x) = x$ and $\tfrac{\partial \Phi_j}{\partial t}(0,x) = \chi_j(x) \,\zeta(x)$ for all $x \in \R^{n+1}$ and $\Phi_{j,t}(\R^{n+1} \setminus \K) = \R^{n+1} \setminus \K$ for all $t \in (-\varepsilon,\varepsilon)$.  We take $\Phi_t = \Phi(t,\cdot) = \Phi_{0,t} \circ \Phi_{1,t} \circ \cdots \circ \Phi_{N,t}$ on $\R^{n+1}$ for each $t \in (-\varepsilon,\varepsilon)$ to get a family of $C^1$-diffeomorphisms such that $\Phi(0,x) = x$ and 
\begin{equation*}
	\frac{\partial \Phi}{\partial t}(0,x) 
	= \sum_{j=0}^N \frac{\partial \Phi_j}{\partial t}(0,t)
	= \sum_{j=0}^N \chi_j(x) \,\zeta(x) = \zeta(x)
\end{equation*}
for all $x \in \R^{n+1}$. 

By the construction above, $\Phi_{t\#}(R) \in \mathbf{I}_m(\R^{n+1} \setminus \K)$ with $\partial \Phi_{t\#}(R) = T$ in $\R^{n+1} \setminus \K$ for all $t \in (-\varepsilon,\varepsilon)$.  Since $R$ is relatively area minimizing, by the first variation formula for area~\cite[Sections 9 and 16]{Sim83} 
\begin{equation*}
	\left. \frac{d}{dt} \Mass(\Phi_{t\#}(R)) \right|_{t=0} = \int \op{div}_R \zeta \,d\|R\| = 0. \qedhere
\end{equation*}
\end{proof}

\begin{remark} \label{first variation R rmk} 
In the special case where $\op{spt} R \subseteq M$ for an orientable embedded $C^2$-submanifold-with-boundary in $\R^{n+1} \setminus \op{spt} T$, we have the following.  Fix an orientation of $M$ and let $\theta : M \rightarrow \mathbb{Z}$ be the $\mathcal{H}^m$-integrable signed multiplicity function of $M$ so that 
\begin{equation*}
	R(\omega) = \int_M \omega \,\theta
\end{equation*}
for all $\omega \in \mathcal{D}^m(\R^{n+1} \setminus \K)$.  By the constancy theorem $\theta$ is constant on $M$; in other words, by \eqref{first variation R concl1} \begin{equation*}
	\int_M (\op{div}_M \zeta) \,\theta \,d\mathcal{H}^m = 0
\end{equation*}
for every $\zeta \in C^1_c(\R^{n+1} \setminus (\op{spt} T \cup \K);\R^{n+1})$ which is tangent to $M$ and so $\theta$ has distributional gradient zero on the interior of $M$, implying $\theta$ is constant on $M$.  In particular, if $M$ is connected, $\op{spt} R = M \setminus \op{int} \K$.  Note that $R$ cannot touch $\partial K$ tangentially by Lemma \ref{boundary rectifiability lemma4}.  By the divergence theorem, \eqref{first variation R concl1} gives us 
\begin{equation*}
	\int_M \op{div}_M \zeta \,\theta \,d\mathcal{H}^m = \int_M \mathbf{H}_M \cdot \zeta \,\theta \,d\mathcal{H}^m 
		+ \int_{\partial M} \eta_M \cdot \zeta \,\theta \,d\mathcal{H}^m = 0
\end{equation*}
for all $\zeta \in C^1_c(\R^{n+1} \setminus \op{spt} T;\R^{n+1})$ which is tangent to $\partial \K$, where $\mathbf{H}_M$ is the mean curvature of $M$ and $\eta_M$ is the outward unit conormal to $M$ along $\partial M$.  Therefore, $R$ has zero mean curvature and $R$ meets $\partial \K$ orthogonally. 
\end{remark}

\begin{proof}[Proof of Theorem C]
Arguing as in the proof of Lemma \ref{first variation R lemma}, there exists $\varepsilon > 0$, $\Phi : (-\varepsilon,\varepsilon) \times \R^{n+1} \rightarrow \R^{n+1}$, and $\Phi_t = \Phi(t,\cdot)$ such that $\Phi_t$ is a one-parameter family of $C^1$-diffeomorphisms with $\Phi(0,x) = x$ and $\tfrac{\partial \Phi}{\partial t}(0,x) = \zeta(x)$ for all $x \in \R^{n+1}$ and $\Phi_t(\R^{n+1} \setminus \K) = \R^{n+1} \setminus \K)$ for all $t \in (-\varepsilon,\varepsilon)$.  

Notice that $\Phi_{t\#}(T) \in \mathbf{I}_m(\R^{n+1} \setminus \K)$ with $\partial \Phi_{t\#}(T) = 0$ in $\R^{n+1} \setminus \K$ for all $t \in (-\varepsilon,\varepsilon)$.  By the first variation formula for area~\cite[Sections 9 and 16]{Sim83},  
\begin{equation} \label{varT eqn1}
	\left. \frac{d}{dt} \Mass(\Phi_{t\#}(T)) \right|_{t=0} = \int \op{div}_T \zeta \,d\|T\|. 
\end{equation}
For $\|T\|$-a.e.~$x$, let $\xi(x) = \xi_1 \wedge \xi_2 \wedge \cdots \wedge \xi_m$ be the simple unit $m$-vector orienting $T$ at $x$, where $\xi_1,\xi_2,\ldots,\xi_m$ is any orthonormal basis for the approximate tangent plane of $T$ at $x$.  Then 
\begin{equation*}
	\Phi_{(t,x) \#}(1 \wedge \xi(x)) = \zeta \wedge \bigwedge_{i=1}^m (\xi_i + t\,\nabla_{\xi_i} \zeta(x)) = \zeta(x) \wedge \xi(x) + O(t)
\end{equation*}
uniformly for $\|T\|$-a.e.~$x$ as $t \rightarrow 0^+$, where $\Phi_{(t,x) \#}$ denotes the pushforward of $\Phi$ at $(t,x) \in [0,1] \times (\R^{n+1} \setminus \K)$.  Hence 
\begin{equation*}
	\Mass(\Phi_{\#}(\llbracket 0, t \rrbracket \times T)) = \int_0^t \int  |\Phi_{(s,x) \#}(1 \wedge \xi(x))| \,d\|T\|(x) \,ds 
		= t \int  |\zeta^{\perp}(x)| \,d\|T\|(x) + O(t^2) 
\end{equation*}
as $t \rightarrow 0^+$, where $\llbracket 0, t \rrbracket$ is the one-dimensional integral current associated with the open interval $(0,t)$ and $\zeta(x)^{\perp}$ denotes the orthogonal projection of $\zeta(x)$ onto the orthogonal complement of the approximate tangent plane of $T$ at $x$ for $\|T\|$-a.e.~$x$.  In particular, the right-side derivative of $\Mass(\Phi_{\#}(\llbracket 0, t \rrbracket \times T))$ with respect to $t$ is given by 
\begin{equation} \label{varT eqn2}
	\left. \frac{d}{dt} \Mass(\Phi_{\#}(\llbracket 0, t \rrbracket \times T)) \right|_{t=0^+} = \int  |\zeta^{\perp}| \,d\|T\| . 
\end{equation}
Since $R$ is relative area minimizing and $\partial (R + \Phi_{\#}(\llbracket 0, t \rrbracket \times T)) = \Phi_{t\#}(T)$ in $\R^{n+1} \setminus \K$, the relatively area minimizing current $R_t \in \mathbf{I}_{m+1}(\R^{n+1} \setminus \K)$ with $\partial R_t = \Phi_{t\#}(T)$ in $\R^{n+1} \setminus \K$ satisfies 
\begin{equation} \label{varT eqn3}
	\Mass(R_t) \geq \Mass(R) - \Mass(\Phi_{\#}(\llbracket 0, t \rrbracket \times T)) 
\end{equation}
for all $t \in (-\varepsilon,\varepsilon)$.  Hence since $(T,R)$ is a relative isoperimetric minimizer 
\begin{equation*}
	\frac{\Mass(T)}{\Mass(R_0)^{\frac{m}{m+1}}} 
	\leq \frac{\Mass(\Phi_{t\#}(T))}{(\Mass(R) - \Mass(\Phi_{\#}(\llbracket 0, t \rrbracket \times T)))^{\frac{m}{m+1}}} 
\end{equation*}
for all $t \in (-\varepsilon,\varepsilon)$ with equality at $t = 0$.  Thus differentiating at time $t = 0$ using \eqref{varT eqn1} and \eqref{varT eqn2}, 
\begin{align*}
	0 &\leq \left. \frac{d}{dt} \frac{\Mass(\Phi_{t\#}(T))}{(\Mass(R) - \Mass(\Phi_{\#}(\llbracket 0, t \rrbracket \times T)))^{\frac{m}{m+1}}} \right|_{t=0^+}
	\\&= \frac{1}{\Mass(R)^{\frac{m}{m+1}}} \left( \int \op{div}_T \zeta \,d\|T\| 
		+ \frac{m}{m+1} \,\frac{\Mass(T)}{\Mass(R)} \int |\zeta^{\perp}| \,d\|T\| \right) .
\end{align*}
By replacing $\zeta$ and $-\zeta$ and rearranging terms, 
\begin{equation} \label{varT eqn4}
	\int \op{div}_T \zeta \,d\|T\| \leq \frac{m}{m+1} \,\frac{\Mass(T)}{\Mass(R)} \int |\zeta^{\perp}| \,d\|T\| 
		\leq \frac{m}{m+1} \,\frac{\Mass(T)}{\Mass(R)} \int |\zeta| \,d\|T\| .
\end{equation}

It immediately follows from \eqref{varT eqn4} that $T$ has distributional mean curvature $\mathbf{H}_T \in L^{\infty}(\|T\|;\R^{n+1})$ such that \eqref{first variation T concl1} holds true for all $\zeta \in C^1_c(\R^{n+1} \setminus \K;\R^{n+1})$ and \eqref{first variation T concl2} holds true.  For each $\delta > 0$ let $\eta_{\delta} \in C^1(\R^{n+1})$ such that $0 \leq \eta_{\delta} \leq 1$, $\eta_{\delta}(x) = 1$ if $d_{\K}(x) \leq \delta/2$, and $\eta_{\delta}(x) = 0$ if $d_{\K}(x) \geq \delta$.  By \eqref{varT eqn4}, 
\begin{equation*}
	\int \op{div}_T (\eta_{\delta} \zeta) \,d\|T\| \leq H_0 \,\Mass(T \llcorner \{d_{\K} \leq \delta\}) \rightarrow 0
\end{equation*}
as $\delta \downarrow 0$ and thus using the dominated convergence theorem 
\begin{align*}
	\int \op{div}_T \zeta \,d\|T\| 
	&= \lim_{\delta \downarrow 0} \left( \int \op{div}_T ((1-\eta_{\delta}) \zeta) \,d\|T\| + \op{div}_T (\eta_{\delta} \zeta) \,d\|T\| \right)
	\\&= \lim_{\delta \downarrow 0} \left( \int \mathbf{H}_T \cdot (1-\eta_{\delta}) \zeta \,d\|T\| + \op{div}_T (\eta_{\delta} \zeta) \,d\|T\| \right)
	\\&= \int \mathbf{H}_T \cdot \zeta \,d\|T\|
\end{align*}
for all $\zeta \in C^1_c(\R^{n+1};\R^{n+1})$ such that $\zeta(x)$ is tangent to $\partial \K$ at each $x \in \partial \K$, proving \eqref{first variation T concl1}.  Moreover, \eqref{varT eqn4} and \eqref{first variation T concl1} imply that 
\begin{equation*}
	\left| \int \mathbf{H}_T \cdot \zeta \,d\|T\| \right| \leq H_0 \int |\zeta^{\perp}| \,d\|T\|
\end{equation*}
and it follows that if $\mathbf{H}_T$ is approximately continuous at $x \in \op{spt} T$ then $\mathbf{H}_T(x)$ is orthogonal to the approximate tangent plane of $T$ at $x$.  
\end{proof}

\begin{remark} \label{first variation T codim one} 
In the special case $m = n$, we can instead note that $R_t = \Phi_{t\#} R$ and thus we can replace \eqref{varT eqn3} with $\Mass(R_t) = \Mass(\Phi_{t\#} R)$.  Modifying the above argument appropriately, we conclude that $T$ has constant scalar mean curvature $H_0$ as in \eqref{H0 defn}. 
\end{remark}

\begin{remark} \label{first variation T rmk} 
In the special case where $\op{spt} T \subseteq M$ for an orientable embedded $C^2$-submanifold-with-boundary in $\R^{n+1}$, by arguing much like we did in Remark \ref{first variation R rmk}, \eqref{first variation T concl1} implies that $T$ has constant multiplicity and $T$ meets $\partial \K$ orthogonally.  Moreover, if $M$ is connected then $\op{spt} T = M \setminus \op{int} \K$. 
\end{remark}

\section{Monotonicity formula and tangent cones} \label{sec:monotonicity sec}

Suppose $(T,R)$ is a relative isoperimetric minimizer in $\R^{n+1} \setminus \K$, where $\K$ is the closure of an open subset with $C^2$-boundary.  We want to understand the local structure of $T$ by using the first variation formula of Theorem C to obtain monotonicity formulas and tangent cones for $T$.  In particular the tangent cones to $T$ will be area minimizing in $\R^{n+1}$ at interior points of $T$ and relatively area minimizing in a half-space at boundary points of $T$.  Note that since the results of this section are local and follow from Theorem C, they hold true whenever $\K$ has $C^2$-boundary without requiring convexity.  With obvious changes, the same results apply to $R$ away from $\op{spt} T$ by the first variation formula of Lemma \ref{first variation R lemma}.

The monotonicity formula for integral varifolds with bounded mean curvature in the sense of Theorem C was previously considered by Gr\"{u}ter and Jost in~\cite{GJ86}.  Let $\kappa_0$ be as in \eqref{kappa0 defn}.  For each $x \in \R^{n+1} \setminus \op{int} \K$, let $\widetilde{x} = 2x - \xi_{\K}(x) \in \K$ be the reflection of $x$ across $\partial \K$.  For each $y \in \R^{n+1} \setminus \K$ and $0 < \rho < 1/(2\kappa_0)$, define 
\begin{equation*}
	\widetilde{B}_{\rho}(y) = \{ x \in \R^{n+1} : |\widetilde{x} - y| < \rho \} . 
\end{equation*}
Observe that if $d_{\K}(y) > \rho$ then $\widetilde{B}_{\rho}(y) \cap (\R^{n+1} \setminus \K) = \emptyset$.  The monotonicity formulas of Gr\"{u}ter and Jost are as follows.

\begin{lemma} \label{monotonicity lemma}
Let $\K$ be the closure of a bounded open subset with $C^2$-boundary and $T \in \mathbf{I}_m(\R^{n+1} \setminus \K)$ and $R \in \mathbf{I}_{m+1}(\R^{n+1} \setminus \K)$ such that $\partial R = T$ in $\R^{n+1} \setminus \K$, $R$ is relatively area minimizing in $\R^{n+1} \setminus \K$, and $(T,R)$ is a relative isoperimetric minimizer in $\R^{n+1} \setminus \K$.  Then 
\begin{align*} 
	&\left( \frac{\|T\|(B_{\sigma}(y)) + \|T\|(\widetilde{B}_{\sigma}(y))}{\omega_m \sigma^m} \right)^{1/p} 
		+ \frac{C H_0 \,\Mass(T)^{1/p}}{m-p} \,\sigma^{1-m/p} 
	\\&\hspace{10mm} \leq \left( \frac{\|T\|(B_{\rho}(y)) + \|T\|(\widetilde{B}_{\rho}(y))}{\omega_m \rho^m} \right)^{1/p} 
		+ \frac{C H_0 \,\Mass(T)^{1/p}}{m-p} \,\rho^{1-m/p} \nonumber 
\end{align*}
for all $0 < \sigma < \rho \leq \min\{ \op{dist}(y,\partial \K), \,1/(4\kappa_0) \}$ and 
\begin{align*} 
	&(1 + C \kappa_0 \sigma) \left( \frac{\|T\|(B_{\sigma}(y)) + \|T\|(\widetilde{B}_{\sigma}(y))}{\omega_m \sigma^m} \right)^{1/p} 
		+ \frac{C H_0 \,\Mass(T)^{1/p}}{m-p} \,\sigma^{1-m/p} 
	\\&\hspace{10mm} \leq (1 + C \kappa_0 \rho) \left( \frac{\|T\|(B_{\rho}(y)) + \|T\|(\widetilde{B}_{\rho}(y))}{\omega_m \rho^m} \right)^{1/p} 
		+ \frac{C H_0 \,\Mass(T)^{1/p}}{m-p} \,\rho^{1-m/p} \nonumber 
\end{align*}
for all $\op{dist}(y,\partial \K) \leq \sigma < \rho \leq 1/(4\kappa_0)$, where $C = C(m,p) \in (0,\infty)$ is a constant. 
\end{lemma}
\begin{proof}
Follows from the first variation formula for $T$ in Theorem C.  See~\cite[Theorem 3.1]{GJ86}.  Note that \cite{GJ86} uses $L^p$ bounds on mean curvature, for which we use $\|\mathbf{H}_T\|_{L^p(\|T\|)} \leq H_0 \,\Mass(T)^{1/p}$. 
\end{proof}

An immediate consequence of the monotonicity formulas of Lemma \ref{monotonicity lemma} is that 
\begin{equation*}
	\Theta^m(\|T\|,y) = \lim_{\rho \downarrow 0} \frac{\|T\|(B_{\rho}(y))}{\omega_m \rho^m} 
\end{equation*}
exists for all $y \in \R^{n+1} \setminus \op{int} \K$. 

\begin{corollary} \label{semicontinuity cor}
Let $\K$ and $T$ be as in Lemma \ref{monotonicity lemma}.  Then the function $\widetilde{\Theta}^m(\|T\|,\cdot) : \R^{n+1} \setminus \op{int} \K \rightarrow [0,\infty)$ given by 
\begin{equation*} 
	\widetilde{\Theta}^m(\|T\|,y) = \begin{cases} 
		\Theta^m(\|T\|,y) &\text{if } y \in \R^{n+1} \setminus \K \\
		2 \,\Theta^m(\|T\|,y) &\text{if } y \in \partial \K. 
	\end{cases}
\end{equation*}
is an upper semi-continuous function.  In other words, whenever $y_j,y \in \R^{n+1} \setminus \op{int} \K$ such that $y_j \rightarrow y$, 
\begin{equation} \label{semicont density1}
	\widetilde{\Theta}^m(\|T\|,y) \geq \lim_{j \rightarrow \infty} \widetilde{\Theta}^m(\|T\|,y_j). 
\end{equation}
In particular, $\Theta^m(\|T\|,y) \geq 1$ for all $y \in \op{spt} T \setminus \partial \K$ and $\Theta^m(\|T\|,y) \geq 1/2$ for all $y \in \op{spt} T \cap \partial \K$. 
\end{corollary}
\begin{proof}
The semi-continuity of density \eqref{semicont density1} is just~\cite[Corollary 3.2]{GJ86}.  Since $T$ is an integer-multiplicity rectifiable current, $\Theta^m(\|T\|,y) \geq 1$ for $\mathcal{H}^m$-a.e.~$y \in \op{spt} T \setminus \K$, which together with \eqref{semicont density1} implies $\widetilde{\Theta}^m(\|T\|,y) \geq 1$ for all $y \in \op{spt} T$.
\end{proof}

\begin{corollary} \label{T R cpt spt cor}
Let $\K$ be the closure of a bounded open subset with $C^2$-boundary and $T \in \mathbf{I}_m(\R^{n+1} \setminus \K)$ and $R \in \mathbf{I}_{m+1}(\R^{n+1} \setminus \K)$ such that $\partial R = T$ in $\R^{n+1} \setminus \K$, $R$ is relatively area minimizing in $\R^{n+1} \setminus \K$, and $(T,R)$ is a relative isoperimetric minimizer in $\R^{n+1} \setminus \K$.  Then $T$ and $R$ both have compact support as currents of $\R^{n+1}$. 
\end{corollary}
\begin{proof}
By the lower bound on density in Corollary \ref{semicontinuity cor} and the monotonicity formula for integral varifolds with bounded mean curvature~\cite[Theorem 17.6]{Sim83}, 
\begin{equation} \label{cpt spt eqn1}
	\|T\|(B_{\rho}(y)) \geq e^{-H_0 \rho} \omega_m \rho^m
\end{equation}
for all $y \in \op{spt} T$ with $\op{dist}(y,\K) \geq 1$ and $\rho \in (0,1)$.  By a standard covering argument using \eqref{cpt spt eqn1} and $\Mass(T) < \infty$, $T$ has compact support.  By similarly using the monotonicity formula and a covering argument to cover $\op{spt} R \cap \{ x : \op{dist}(x,\op{spt} T \cup \K) \geq 1 \}$, $R$ also has compact support. 
\end{proof}

Using the monotonicity formulas of Lemma \ref{monotonicity lemma} we can establish the existence of tangent cones to $T$.  First let us consider the tangent behavior of $\K$.  For each $y \in \R^{n+1} \setminus \op{int} \K$ and $\rho > 0$, let $\K_{y,\rho} = \eta_{y,\rho}(\K)$, where $\eta_{y,\rho}(x) = (x-y)/\rho$.  Since $\K$ has a $C^2$-boundary, for every $y \in \partial \K$ as $\rho \downarrow 0$ we have $\K_{y,\rho} \rightarrow \K_{y,0}$ locally in Hausdorff distance and $\partial \K_{y,\rho} \rightarrow \partial \K_{y,0}$ locally in the $C^2$ topology in $\R^{n+1}$, where $\K_{y,0} \subset \R^{n+1}$ is a half-space such that $0 \in \partial \K_{y,0}$ and $\partial \K_{y,0}$ is the tangent hyperplane to $\partial \K$ at $y$.  If instead $y \in \R^{n+1} \setminus \K$, we set $\K_{y,0} = \emptyset$ so that $\R^{n+1} \setminus \K_{y,0} = \R^{n+1}$. 

For each $y \in \op{spt} T$ and $\rho > 0$ let $T_{y,\rho} = \eta_{y,\rho \#} T$.  Tangent cones to $T$ are defined as follows. 

\begin{definition}
Let $\K$ and $T$ be as in Lemma \ref{monotonicity lemma}.  We say that $C \in \mathcal{D}_m(\R^{n+1} \setminus \K_{y,0})$ is a tangent cone to $T$ at $y \in \op{spt} T$ if there exists a sequence $\rho_j \rightarrow 0^+$ such that $T_{y,\rho_j} \rightarrow C$ weakly $\R^{n+1} \setminus \K_{y,0}$. 
\end{definition}

In Lemma \ref{exist tan S lemma} below we will show that at least one tangent cone exists at each point $y \in \op{spt} S$.  It is not generally known whether the tangent cone $C$ is unique as $C$ might depend on the sequence of radii $\rho_j$. 

\begin{lemma} \label{exist tan S lemma}
Let $\K$ be the closure of a bounded open subset with $C^2$-boundary and $T \in \mathbf{I}_m(\R^{n+1} \setminus \K)$ and $R \in \mathbf{I}_{m+1}(\R^{n+1} \setminus \K)$ such that $\partial R = T$ in $\R^{n+1} \setminus \K$, $R$ is relatively area minimizing in $\R^{n+1} \setminus \K$, and $(T,R)$ is a relative isoperimetric minimizer in $\R^{n+1} \setminus \K$.  
\begin{enumerate}
	\item[(i)]  There exists at least one tangent cone $C$ to $T$ at each $y \in \op{spt} T$. 
	\item[(ii)]  If $y \in \op{spt} T$ and $C$ is a tangent cone to $T$ at $y$, then $C$ is a nonzero current of $\R^{n+1}$, $C \in \mathbf{I}_{m,{\rm loc}}(\R^n)$, and $C$ is a cone in the sense that $\eta_{0,\lambda \#} C = C$ for all $\lambda > 0$. 
	\item[(iii)]  If $y \in \op{spt} T \setminus \partial \K$ and $C$ is a tangent cone to $T$ at $y$, then $C$ is a (locally) area minimizing integral current with $\partial C = 0$ in $\R^{n+1}$.  If instead $y \in \op{spt} T \cap \partial \K$ and $C$ is a tangent cone to $T$ at $y$, then $C$ is a (locally) relatively area minimizing integral current with $\partial C = 0$ in the half-space $\R^{n+1} \setminus \K_{y,0}$. 
\end{enumerate}
\end{lemma}
\begin{proof}
Let $y \in \op{spt} T \cap U$.  Noting that $\Theta^m(\|T\|,y)$ exists as a consequence of Lemma \ref{monotonicity lemma} and that $\Theta^m(\|T\|,y) \geq 1/2 > 0$ by Corollary \ref{semicontinuity cor}, 
\begin{equation} \label{tangent cone eqn1}
	\lim_{\rho \downarrow 0} \frac{\|T_{y,\rho}\|(B_r(0))}{\omega_m r^m} 
		= \lim_{\rho \downarrow 0} \frac{\|T\|(B_{\rho r}(y))}{\omega_m (\rho r)^m} = \Theta^m(\|T\|,y) > 0
\end{equation}
for all $0 < r < \infty$.  By Lemma \ref{almost minimizing lemma},  
\begin{equation} \label{tangent cone eqn2}
	\Mass(T_{y,\rho}) \leq \Mass_{\R^{n+1} \setminus \K_{y,\rho}}(T_{y,\rho} + \partial X) + 2H_0 \rho \,\Mass(X) 
\end{equation}
for all $X \in \mathbf{I}_{m+1}(\R^{n+1} \setminus \K_{y,\rho})$ with $\Mass(X) \leq \tfrac{1}{2} \,\rho^{-m-1} \,\Mass(R)$.  By Lemma \ref{boundary rectifiability lemma4} and \eqref{tangent cone eqn1}, if $y \in \partial \K$ then 
\begin{align} \label{tangent cone eqn3}
	\limsup_{\rho \downarrow 0} \frac{\|T_{y,\rho}\|(B_r(0) \cap \{d_{\K_{y,\rho}} < s \})}{\omega_m r^{m-1} s} 
		&= \limsup_{\rho \downarrow 0} \frac{\|T\|(B_{\rho r}(y) \cap \{d_{\K} < \rho s\})}{\omega_m \,(\rho r)^{m-1} (\rho s)} 
		\\&\leq \lim_{\rho \downarrow 0} C(m) \,\frac{\|T\|(B_{\rho r}(y))}{\omega_m (\rho r)^m} = C(m) \,\Theta^m(\|T\|,y) \nonumber 
\end{align}
for all $0 < s < r < \infty$ and 
\begin{align} \label{tangent cone eqn4}
	\limsup_{\rho \downarrow 0} \frac{\|\partial T_{y,\rho}\|(B_r(0))}{\omega_{m-1} r^{m-1}} 
		&= \limsup_{\rho \downarrow 0} \frac{\|\partial T\|(B_{\rho r}(y))}{\omega_{m-1} (\rho r)^{m-1}} 
		\\&\leq \lim_{\rho \downarrow 0} C(m) \,\frac{\|T\|(B_{\rho r}(y))}{\omega_m (\rho r)^m} = C(m) \,\Theta^m(\|T\|,y) \nonumber 
\end{align}
for all $0 < r < \infty$. 

Take any $\rho_j \rightarrow 0^+$.   By the Federer-Fleming compactness theorem, \eqref{tangent cone eqn1}, and \eqref{tangent cone eqn4}, after passing to a subsequence $T_{y,\rho_j} \rightarrow C$ in the flat norm topology on compact subsets of $\R^{n+1}$ for some $C \in \mathbf{I}_{m,{\rm loc}}(\R^{n+1})$, proving (i).  By Lemma \ref{convergence lemma}, \eqref{tangent cone eqn1}, \eqref{tangent cone eqn2}, and \eqref{tangent cone eqn3}, $C$ is relatively area minimizing in $\R^{n+1} \setminus \K_{y,0}$, proving (iii), and also $\|T_{y,\rho_j}\| \rightarrow \|C\|$ in the sense of Radon measures on $\R^{n+1}$.  By \eqref{tangent cone eqn1} and the convergence $\|T_{y,\rho_j}\| \rightarrow \|C\|$ in the sense of Radon measures on $\R^{n+1}$, 
\begin{equation} \label{tangent cone eqn5}
	\frac{\|C\|(B_r(0))}{\omega_m r^m} = \lim_{j \rightarrow \infty} \frac{\|T_{y,\rho_j}\|(B_r(0))}{\omega_m r^m} = \Theta^m(\|T\|,y) > 0
\end{equation}
for all $0 < r < \infty$.  Hence $C$ is a nonzero current of $\R^{n+1}$.  Moreover, in the case that $y \in \op{spt} T \setminus \partial \K$, by \eqref{tangent cone eqn5} $C$ is a cone, see~\cite[Theorem 19.3, Theorem 35.1]{Sim83}. 

It remains to show that if $y \in \op{spt} T \cap \partial \K$ then $C$ is a cone.  After an orthogonal change of coordinates let $\K_{y,0} = \R^{n+1}_+$.  Let $\iota : \R^{n+1} \rightarrow \R^{n+1}$ be the reflection map about $\R^n \times \{0\}$ given by $\iota(x',x_{n+1}) = (x',-x_{n+1})$ for all $x = (x',x_{n+1}) \in \R^{n+1}$.  Let $\widetilde{C} = C - \iota_{\#} C \in \mathbf{I}_{m,{\rm loc}}(\R^{n+1})$.  We claim that $\widetilde{C}$ is area minimizing in $\R^{n+1}$.  To see this, let $X \in \mathbf{I}_{m+1}(\R^{n+1})$ with compact support and $\partial X = 0$ in $\R^{n+1}$.  Since $C$ is relatively area minimizing in $\R^{n+1}_+$, $\Mass(C) \leq \Mass(C + X \llcorner \R^{n+1}_+)$, and by symmetry $\Mass(\iota_{\#} C) \leq \Mass(\iota_{\#} C + X \llcorner \R^{n+1}_-)$.  Therefore $\Mass(\widetilde{C}) \leq \Mass(\widetilde{C} + X)$.  By \eqref{tangent cone eqn5}, 
\begin{equation*}
	\frac{\|\widetilde{C}\|(B_r(0))}{\omega_m r^m} = 2\,\Theta^m(\|T\|,y) 
\end{equation*}
for all $0 < r < \infty$ and so $\widetilde{C}$ is a cone.  Thus $C$ is also a cone. 
\end{proof}

\begin{corollary} \label{nontrivial bdry cones cor}
Let $\K$ and $T$ be as in Lemma \ref{exist tan S lemma}.  For each $y \in \op{spt} T \cap \partial \K$ and each tangent cone $C$ to $T$ at $y$, 
\begin{equation*}
	\|C\|(\K_{y,0}) = 0, \quad 0 < \omega_{m-1}^{-1} \,\|\partial C\|(B_1(0)) \leq \Theta^{*m-1}(\|\partial T\|,y). 
\end{equation*}
\end{corollary}
\begin{proof}
Let $\rho_j \downarrow 0$ such that $T_{y,\rho_j} \rightarrow C$ in the flat norm topology in $\R^{n+1}$.  Since $T_{y,\rho_j} \rightarrow C$ weakly in $\R^{n+1}$ and each $T_{y,\rho_j}$ satisfies \eqref{tangent cone eqn3}, $\|C\|(\K_{y,0}) = 0$.  Since $\partial T_{y,\rho_j} \rightarrow \partial C$ weakly in $\R^{n+1}$, by the semi-continuity of mass $\omega_{m-1}^{-1} \,\|\partial C\|(B_1(0)) \leq \Theta^{*m-1}(\|\partial T\|,y)$.  

Suppose $\partial C = 0$ in $\R^{n+1}$.  Since $C$ is relatively area minimizing in $\R^{n+1} \setminus \K_{y,0}$, $C$ is area minimizing in $\R^{n+1}$.  By~\cite[Theorem 36.5]{Sim83}, $C$ is an area minimizing cone with zero boundary in $\R^{n+1}$ such that $\op{spt} C$ lies in the half-space $\R^{n+1} \setminus \op{int} \K_{y,0}$ only if $\op{spt} C \subseteq \partial \K_{y,0}$.  But this contradicts $C$ being a nonzero locally integral current with $\|C\|(\K_{y,0}) = 0$.  Therefore, $\partial C \neq 0$ in $\R^{n+1}$ and in particular $\|\partial C\|(B_1(0)) > 0$. 
\end{proof}

\section{Restricted support hypersurfaces and their unit normals} \label{sec:restricted normal sec}

Let $\K$ be a proper convex subset of $\R^{n+1}$ with $C^2$-boundary and $(T,R)$ be an relative isoperimetric minimizer in $\R^{n+1} \setminus \K$.  We want to prove Theorem D, which involves computing areas of sets of unit normal vectors to support hyperplanes of the convex hull of $\op{spt} T$.  For the relative isoperimetric inequality, in contrast with the isoperimetric inequality, the support hyperplanes of $\op{spt} T$ can touch at either interior points $y \in \op{spt} T \setminus \K$ or boundary points $y \in \op{spt} T \cap \partial \K$.  The behavior at interior points and boundary points can be quite different, as $T$ has bounded mean curvature and an $(n-m)$-dimensional unit normal cone at interior points whereas one expects $T$ to have an $(n-m+1)$-dimensional unit normal cone at boundary points, for example consider when $T$ is an $n$-dimensional hemisphere in $\R^{n+1}_+$.  To treat the boundary points, we will use the concept of restricted support hyperplanes from~\cite{CGR06} and in particular~\cite[Theorem 5.3]{CGR06}. 

Let $X$ be any nonempty compact subset of $\R^{n+1}$.  For each $y \in \R^{n+1}$ we define the \textit{(unit) normal cone} $N_y X$ of $X$ at $y$ by 
\begin{equation*}
	N_y X = \{ \nu \in \mathbb{S}^n : \nu \cdot (x-y) \leq 0 \text{ for all } x \in X \} .
\end{equation*}
Whenever $y \in X$  and $\nu \in \mathbb{S}^n$ such that $X \subseteq \{ x \in \R^{n+1} : \nu \cdot (x-y) \leq 0 \}$, we call the hyperplane $\{ x \in \R^{n+1} : \nu \cdot (x-y) = 0 \}$ the \textit{support hyperplane} of $X$ at $y$ and $\nu$ the \textit{outward unit normal}.  Since $X$ is compact, for each $\nu \in \mathbb{S}^n$ the function $x \in X \mapsto \nu \cdot x$ attains its maximum value at some point $y \in X$ and thus the support hyperplane to $X$ with outward unit normal $\nu$ touches $X$ at $y$; that is, $\nu \in N_y X$.  Hence for each compact set $X \subset \R^{n+1}$ we have 
\begin{equation} \label{total normal cone is sphere}
	\mathbb{S}^n = \bigcup_{y \in X} N_y X. 
\end{equation}

Let $\sigma : X \rightarrow \mathbb{S}^n$ be a continuous map such that $\sigma(y) \in N_y X$ for all $y \in X$.  For each $\alpha \in (0,\pi/2)$, we define the \textit{$\alpha$-restricted normal cone} of $X$ at $y$ with respect to $\sigma$ by 
\begin{equation*}
	N^{\alpha}_y X/\sigma = \{ \nu \in N_y X : \sigma(y) \cdot \nu \geq \sin(\alpha) \text{ for all } x \in X \} .
\end{equation*}
We define the \textit{restricted normal cone} of $X$ at $y$ with respect to $\sigma$ by 
\begin{equation*}
	N_y X/\sigma = \bigcup_{\alpha \in (0,\pi/2)} N^{\alpha}_y X/\sigma = \{ \nu \in N_y X : \sigma(y) \cdot \nu > 0 \text{ for all } x \in X \} .
\end{equation*}
We set 
\begin{equation*}
	N^{\alpha} X/\sigma = \bigcup_{y \in X} N^{\alpha}_y X/\sigma, \quad 
	NX/\sigma = \bigcup_{y \in X} N_y X/\sigma 
\end{equation*}
as a subsets of $\mathbb{S}^n$.  For each $\alpha \in (0,\pi/2)$, $N^{\alpha} X/\sigma$ is closed.  To see this, suppose $y_i \in X$, $\nu_i \in N^{\alpha}_{y_i} X/\sigma$, and $\nu \in \mathbb{S}^n$ such that $\nu_i \rightarrow \nu$.  Since $X$ is compact, after passing to a subsequence we may let $y_i \rightarrow y$ for some $y \in X$.  Since $\nu_i  \in N^{\alpha}_{y_i} X/\sigma$, $\nu_i \cdot (x - y_i) \leq 0$ for all $x \in X$ and $\sigma(y_i) \cdot \nu_i \geq \sin(\alpha)$, which letting $i \rightarrow \infty$ using the continuity of $\sigma$ gives us $\nu \cdot (x - y) \leq 0$ for all $x \in X$ and $\sigma(y) \cdot \nu \geq \sin(\alpha)$.  Hence $\nu \in N^{\alpha}_y X/\sigma$.  Since $N^{\alpha} X/\sigma \subseteq N^{\beta} X/\sigma$ for all $0 < \beta < \alpha$ and 
\begin{equation*}
	NX/\sigma = \bigcup_{\alpha \in (0,\pi/2)} N^{\alpha} X/\sigma = \bigcup_{i=1}^{\infty} N^{1/i} X/\sigma 
\end{equation*}
$NX/\sigma$ is a Borel set.  Note that our definition of the restricted normal cones $N^{\alpha} X/\sigma$ and $N X/\sigma$ differ from~\cite{CGR06}, which used the restricted normal cone $\{ \nu \in N_y X : \sigma(y) \cdot \nu \geq 0 \}$ for each $y \in X$.  
Our main estimate regarding the area of $NX/\sigma$ is the following variant of~\cite[Theorem 5.3]{CGR06}. 

\begin{theorem} \label{rest norm abs thm}
Let $X \subset \R^{n+1}$ be a nonempty compact set and let $\sigma : X \rightarrow \mathbb{S}^n$ be a continuous map such that $\sigma(y) \in N_y X$ for all $y \in X$.  Then for every $\alpha \in (0,\pi/2)$,  
\begin{equation} \label{rest norm abs concl1}
	\mathcal{H}^n(N^{\alpha} X/\sigma) \geq \frac{1 - C \alpha}{2} \,\mathcal{H}^n(\mathbb{S}^n), 
\end{equation}
where $C = C(n) = 2 \,\mathcal{H}^{n-1}(\mathbb{S}^{n-1})/\mathcal{H}^n(\mathbb{S}^n)$.  Moreover, 
\begin{equation} \label{rest norm abs concl2}
	\mathcal{H}^n(NX/\sigma) \geq \frac{1}{2} \,\mathcal{H}^n(\mathbb{S}^n). 
\end{equation}
\end{theorem}

\begin{lemma} \label{rest norm abs lemma1}
Let $X$ be a closed convex subset of $\mathbb{S}^n$ and $u \in X$.  For every $\alpha \in (0,\pi/2)$,  
\begin{equation*} 
	\mathcal{H}^n(X \cap \{ \nu : u \cdot \nu \geq \sin(\alpha) \}) \geq \frac{1 - C \alpha}{2} \,\mathcal{H}^n(X), 
\end{equation*}
where $C = C(n) = 2 \,\mathcal{H}^{n-1}(\mathbb{S}^{n-1})/\mathcal{H}^n(\mathbb{S}^n)$.
\end{lemma}

\begin{proof}
Without loss of generality assume $X \neq \emptyset$ and let $u = e_{n+1}$ be the $(n+1)$-th unit coordinate vector.  Parameterize $\mathbb{S}^n$ by $(\cos(\theta) \,\omega, \sin(\theta))$ where $\theta \in [-\pi/2,\pi/2]$ and $\omega \in \mathbb{S}^{n-1}$.  Let 
\begin{align*}
	E &= X \cap \{ (\cos(\alpha) \,\omega, \sin(\alpha)) : \omega \in \mathbb{S}^{n-1} \} , \\
	A &= \{ (\cos(\theta) \,\omega, \sin(\theta)) : \theta \in [\alpha,\pi/2], \, \omega \in E \}, \\
	B &= \{ (\cos(\theta) \,\omega, \sin(\theta)) : \theta \in [\alpha,\pi/2], \, \omega \in \mathbb{S}^{n-1} \setminus E \} , \\
	A' &= \{ (\cos(\theta) \,\omega, \sin(\theta)) : \theta \in [-\pi/2,\alpha], \, \omega \in E \}, \\
	B' &= \{ (\cos(\theta) \,\omega, \sin(\theta)) : \theta \in [-\pi/2,\alpha], \, \omega \in \mathbb{S}^{n-1} \setminus E \} .
\end{align*}
In other words, $E$ is the intersection of $X$ and the latitude $\{ \theta = \alpha \}$.  $A$ is the union of all geodesic arcs from $u$ to points of $E$.  $B$ is the complement of $A$ in the region $\{ \theta \geq \alpha \}$.  $A'$ and $B'$ are reflections of $A$ and $B$ across the latitude $\{ \theta = \alpha \}$ into the region $\{ \theta \leq \alpha \}$.  

We compute that 
\begin{equation} \label{rest norm abs1 eqn1}
	\mathcal{H}^n(A) = \int_E \int_{\alpha}^{\pi/2} \cos^{n-1}(\theta) \,d\theta \,d\mathcal{H}^{n-1}(\omega) 
		= \mathcal{H}^{n-1}(E) \int_{\alpha}^{\pi/2} \cos^{n-1}(\theta) \,d\theta 
\end{equation}
and similarly 
\begin{equation} \label{rest norm abs1 eqn2}
	\mathcal{H}^n(A') = \mathcal{H}^{n-1}(E) \int_{-\pi/2}^{\alpha} \cos^{n-1}(\theta) \,d\theta .
\end{equation}
Since $\cos(\theta) \leq 1$, $\int_{0}^{\alpha} \cos^{n-1}(\theta) \,d\theta \leq \alpha$ and thus \eqref{rest norm abs1 eqn1} and \eqref{rest norm abs1 eqn2} give us 
\begin{equation} \label{rest norm abs1 eqn3}
	\frac{\mathcal{H}^n(A)}{\mathcal{H}^n(A')} 
	= \frac{\int_{0}^{\pi/2} \cos^{n-1}(\theta) \,d\theta - \int_{0}^{\alpha} \cos^{n-1}(\theta) \,d\theta}{
		\int_{0}^{\pi/2} \cos^{n-1}(\theta) \,d\theta + \int_{0}^{\alpha} \cos^{n-1}(\theta) \,d\theta}
	\geq \frac{1 - C \alpha}{1 + C \alpha}
\end{equation}
where $C = C(n) = \left( \int_{0}^{\pi/2} \cos^{n-1}(\theta) \,d\theta \right)^{-1}$.  By \eqref{rest norm abs1 eqn1} in the special case $X = \mathbb{S}^n$ and $\alpha = 0$ (so that $A = \mathbb{S}^n \cap \mathbb{R}^{n+1}_+$ and $E = \mathbb{S}^{n-1}$), $C = 2 \,\mathcal{H}^{n-1}(\mathbb{S}^{n-1})/\mathcal{H}^n(\mathbb{S}^n)$.

By the definition of $A$ and $E$ and the convexity of $X$, $A \subseteq X$.  Moreover, $X \cap B' = \emptyset$ since if there were $x \in X \cap B'$ then the geodesic arc from $u$ to $x$ lies in $X$ by convexity and intersects $E$, contradicting the definition of $B'$.  Hence by $A \subseteq X$ and \eqref{rest norm abs1 eqn3}, 
\begin{equation} \label{rest norm abs1 eqn4}
	\mathcal{H}^n(X \cap A) = \mathcal{H}^n(A) \geq \frac{1 - C \alpha}{1 + C \alpha} \,\mathcal{H}^n(A') 
		= \frac{1 - C \alpha}{1 + C \alpha} \,\mathcal{H}^n(X \cap A') .
\end{equation}
By \eqref{rest norm abs1 eqn4} and $X \cap B' = \emptyset$, 
\begin{align*}
	\mathcal{H}^n(X \cap \{ \theta \geq \alpha \}) &= \mathcal{H}^n(X \cap A) + \mathcal{H}^n(X \cap B)
	\\&\geq \frac{1 - C \alpha}{1 + C \alpha} \left( \mathcal{H}^n(X \cap A') + \mathcal{H}^n(X \cap B') \right) 
	\\&= \frac{1 - C \alpha}{1 + C \alpha} \,\mathcal{H}^n(X \cap \{ \theta \leq \alpha \}). 
\end{align*}
By adding $\frac{1 - C \alpha}{1 + C \alpha} \,\mathcal{H}^n(X \cap \{ \theta \geq \alpha \})$ to both sides we get 
\begin{equation*}
	\mathcal{H}^n(X \cap \{ \theta \geq \alpha \}) \geq \frac{1 - C \alpha}{2} \,\mathcal{H}^n(X). \qedhere
\end{equation*}
\end{proof}

\begin{lemma} \label{rest norm abs lemma2}
Let $X \subset \R^{n+1}$ be a nonempty finite set of points and let $\sigma : X \rightarrow \mathbb{S}^n$ be a map such that $\sigma(y) \in N_y X$ for all $y \in X$.  Then for every $\alpha \in (0,\pi/2)$, 
\begin{equation*} 
	\mathcal{H}^n(N^{\alpha} X/\sigma) \geq \frac{1 - C \,\alpha}{2} \,\mathcal{H}^n(\mathbb{S}^n), 
\end{equation*}
where $C = C(n) = 2 \,\mathcal{H}^{n-1}(\mathbb{S}^{n-1})/\mathcal{H}^n(\mathbb{S}^n)$.
\end{lemma}

\begin{proof}
Recall that \eqref{total normal cone is sphere} holds true.  Furthermore, $\nu \in \mathbb{S}^n$ belongs to $\op{int}_{\mathbb{S}^n} N_{x_i} X$, the interior of some $N_{x_i} X$ relative to $\mathbb{S}^n$, if and only if there exists a support hyperplane of $X$ at $x_i$ with outward unit normal $\nu$ which intersects $X$ only at $x_i$.  Thus 
\begin{equation} \label{rest norm abs2 eqn1}
	\op{int}_{\mathbb{S}^n} N_{x_i} X \cap \op{int}_{\mathbb{S}^n} N_{x_j} X = \emptyset 
\end{equation}
for all $i \neq j$.  By~\cite[Lemma 4.1]{CGR06}, for each $i$ either $N_{x_i} X = \emptyset$, $N_{x_i} X$ is a pair of antipodal points and thus $\mathcal{H}^n(N_{x_i} X) = 0$, or $N_{x_i} X$ is a convex spherical set.  Hence by \eqref{rest norm abs2 eqn1}, Lemma \ref{rest norm abs lemma1}, and \eqref{total normal cone is sphere} for every $\alpha \in (0,\alpha_0]$ 
\begin{equation*}
	\mathcal{H}^n(N^{\alpha} X/\sigma)
	= \sum_{i=1}^k \mathcal{H}^n(N^{\alpha}_{x_i} X/\sigma) 
	\geq \sum_{i=1}^k \frac{1 - C \alpha}{2} \,\mathcal{H}^n(N_{x_i} X) 
	= \frac{1 - C \alpha}{2} \,\mathcal{H}^n(\mathbb{S}^n). \qedhere
\end{equation*}
\end{proof}

\begin{proof}[Proof of Theorem \ref{rest norm abs thm}]
Since $X$ is compact, we may cover $X$ with finitely many open balls of radius $1/i$ centered at points of $X$.  Let $X_i$ denote the set of the centers of these balls.  Notice that $X_i \subseteq X$ and $\op{dist}_{\mathcal{H}}(X_i,X) < 1/i$.  We claim that for every $\delta > 0$ there exists $i_0 > 0$ such that for every $i \geq i_0$ 
\begin{equation} \label{rest norm abs eqn1}
	\sup_{\nu \in N^{\alpha} X_i/\sigma} \op{dist}_{\mathbb{S}^n}(\nu, N^{\alpha} X/\sigma) < \delta, 
\end{equation}
where $\op{dist}_{\mathbb{S}^n}$ denotes geodesic distance on $\mathbb{S}^{n+1}$.  Suppose to the contrary that for some $\delta > 0$ and infinitely many $i$ there exists $y_i \in X_i$ and $\nu_i \in N^{\alpha} _{y_i} X_i/\sigma$ such that 
\begin{equation} \label{rest norm abs eqn2}
	\op{dist}(\nu_i,N^{\alpha} X/\sigma) \geq \delta. 
\end{equation}
After passing to a subsequence let $y_i \rightarrow y$ for some $y \in X$ and $\nu_i \rightarrow \nu$ for some $\nu \in \mathbb{S}^n$.  Given any $x \in X$, there exists $x_i \in X_i$ such that $x_i \rightarrow x$.  Since $\nu_i \in N^{\alpha}_{y_i} X_i/\sigma$, $\nu_i \cdot (x_i - y_i) \leq 0$ for each $i$.  Letting $i \rightarrow \infty$ we obtain $\nu \cdot (x - y) \leq 0$.  Hence $\nu \in N_y X$.  Moreover, since $\nu_i \in N^{\alpha}_{y_i} X_i/\sigma$, $\sigma(y_i) \cdot \nu_i \geq \sin(\alpha)$ for each $i$.  Letting $i \rightarrow \infty$ using the continuity of $\sigma$ we obtain $\sigma(y) \cdot \nu \geq \sin(\alpha)$.  Hence $\nu \in N^{\alpha}_y X/\sigma$.  But now $\nu_i \rightarrow \nu$ and $\nu \in N^{\alpha}_y X/\sigma$, contradicting \eqref{rest norm abs eqn2}.  

Let $C = C(n) = 2 \,\mathcal{H}^{n-1}(\mathbb{S}^{n-1})/\mathcal{H}^n(\mathbb{S}^n)$.  To show \eqref{rest norm abs concl1}, suppose to the contrary that for some $\alpha \in (0,1/C)$ 
\begin{equation*}
	\mathcal{H}^n(N^{\alpha} X/\sigma) < \frac{1 - C \alpha}{2} \,\mathcal{H}^n(\mathbb{S}^n). 
\end{equation*}
Then, recalling from above that $N^{\alpha} X/\sigma$ is closed, $\mathbb{S}^n \setminus (N^{\alpha} X/\sigma)$ is an open subset of $\mathbb{S}^n$ with $\mathcal{H}^n(\mathbb{S}^n \setminus (N^{\alpha} X/\sigma)) > \frac{1 + C \alpha}{2} \,\mathcal{H}^n(\mathbb{S}^n)$.  Thus we can find a compact set $K \subset \mathbb{S}^n \setminus (N^{\alpha} X/\sigma)$ such that 
\begin{equation} \label{rest norm abs eqn3}
	\mathcal{H}^n(K) > \frac{1 + C \alpha}{2} \,\mathcal{H}^n(\mathbb{S}^n). 
\end{equation}
By \eqref{rest norm abs eqn1}, for each sufficiently large $i$ we have $K \cap (N^{\alpha} X_i/\sigma) = \emptyset$, and by Lemma \ref{rest norm abs lemma2}
\begin{equation*}
	\mathcal{H}^n(N^{\alpha} X_i/\sigma) \geq \frac{1 - C \alpha}{2} \,\mathcal{H}^n(\mathbb{S}^n), 
\end{equation*}
contradicting \eqref{rest norm abs eqn3}.  Therefore, \eqref{rest norm abs concl1} holds true. 

Finally, we obtain \eqref{rest norm abs concl2} by letting $\alpha \downarrow 0$ in \eqref{rest norm abs concl1}.
\end{proof}

Now suppose $\K$ is a bounded proper convex subset of $\R^{n+1}$ with $C^2$-boundary and $(T,R)$ is an relative isoperimetric minimizer in $\R^{n+1} \setminus \K$.  Note that by Corollary \ref{T R cpt spt cor}, $T$ has compact support.  Applying the discussion above to the support of $T$, for each $y \in \op{spt} T$ we define the normal cone $N_y T$ of $T$ at $y$ by 
\begin{equation*}
	N_y T \equiv N_y (\op{spt} T) = \{ \nu \in \mathbb{S}^n : \nu \cdot (x-y) \leq 0 \text{ for all } x \in \op{spt} T \} .
\end{equation*}
For each $\nu \in \mathbb{S}^n$ there exists a $y \in \op{spt} T$ such that support hyperplane with unit normal $\nu$ touches $\op{spt} T$ at $y$.  Thus 
\begin{equation*}
	\mathbb{S}^n = \bigcup_{y \in \op{spt} T} N_y T. 
\end{equation*}
We define the \textit{set of restricted normals} $RNT \subseteq \mathbb{S}^n$ by 
\begin{equation*}
	RNT = \bigcup_{y \in \op{spt} T \setminus \K} N_y T. 
\end{equation*}
If $\Pi = \{ x \in \R^{n+1} : \nu \cdot (x-y) = 0 \}$ is the support hyperplane of $\op{spt} T$ at some point $y \in \op{spt} T \setminus \K$ with outward unit normal $\nu$, we call $\Pi$ a \textit{restricted support hyperplane} of $\op{spt} T$ at $y$.  Note that this definition of restricted support normals and hyperplanes differs from~\cite{CGR06}, which also allowed a restricted support hyperplane $\Pi$ to touch the support of $T$ at some boundary point $y \in \op{spt} T \cap \K$ provided the outward unit normal $\nu$ to $\Pi$ satisfied $\nu_{\K}(y) \cdot \nu = 0$.  

Suppose $\op{spt} T \cap \partial \K \neq \emptyset$.  Regarding $T$ as an integral current of $\R^{n+1}$, $\partial T \in \mathbf{I}_{m-1}(\R^{n+1})$ and, by Corollary \ref{nontrivial bdry cones cor}, $\op{spt} \partial T = \op{spt} T \cap \partial \K$.  Thus for each $y \in \op{spt} T \cap \partial \K$ we define the normal cone $N_y \partial T$ of $\partial T$ at $y$ by 
\begin{equation*}
	N_y \partial T \equiv N_y (\op{spt} T \cap \partial \K) 
		= \{ \nu \in \mathbb{S}^n : \nu \cdot (x-y) \leq 0 \text{ for all } x \in \op{spt} T \cap \partial \K \} .
\end{equation*}
For each $y \in \op{spt} T \cap \partial \K$, the restricted normal cone $N_y \partial T/\nu_{\K}$ of $\partial T$ at $y$ with respect to the outward unit normal $\nu_{\K}$ to $\K$ is given by 
\begin{equation*}
	N_y \partial T/\nu_{\K} = \{ \nu \in N_y \partial T : \nu_{\K}(y) \cdot \nu > 0 \} 
\end{equation*}
and 
\begin{equation*}
	N\partial T/\nu_{\K} = \bigcup_{y \in \op{spt} T \cap \partial \K} N\partial T/\nu_{\K}
\end{equation*}
as a subset of $\mathbb{S}^n$.

Recall that for each $\nu \in N\partial T/\nu_{\K}$, there exists $y \in \op{spt} T$ such that $\nu \in N_y T$.  We claim that in fact there exists $y \in \op{spt} T \setminus \K$ such that $\nu \in N_y T$. 

\begin{lemma} \label{restricted normal lemma1}
Let $\K$ be a bounded proper convex subset of $\R^{n+1}$ with $C^2$-boundary.  Let $T \in \mathbf{I}_m(\R^{n+1} \setminus \K)$ and $R \in \mathbf{I}_{m+1}(\R^{n+1} \setminus \K)$ such that $\partial R = T$ in $\R^{n+1} \setminus \K$, $R$ is relatively area minimizing in $\R^{n+1} \setminus \K$, $(T,R)$ is a relative isoperimetric minimizer in $\R^{n+1} \setminus \K$, and $\op{spt} T \cap \partial \K \neq \emptyset$.  Then 
\begin{equation*}
	N\partial T/\nu_{\K} \subseteq RNT. 
\end{equation*}
\end{lemma}
\begin{proof}
Suppose $y \in \op{spt} T \cap \partial \K$ and $\nu \in N_y \partial T/\nu_{\K}$.  Let $\Pi_1$ denote the support hyperplane to $\op{spt} T$ with outward unit normal $\nu$.  Let $\Pi_2$ denote the support hyperplane to $\op{spt} T \cap \partial \K$ at $y$ with outward unit normal $\nu$.  Observe that $\op{spt} T \cap \partial \K \subseteq \op{spt} T$ and thus $\op{spt} T \cap \partial \K$ lies to one side of $\Pi_1$.  Hence if $\Pi_1 \neq \Pi_2$ then $\Pi_1$ must touch $\op{spt} T$ at some interior point in $\op{spt} T \setminus \mathcal{K}$, implying $\nu \in RNT$.  Thus we may suppose that $\Pi_1 = \Pi_2$.  In that case, $\Pi_1$ touches $\op{spt} T$ at $y$.  Moreover, since $\nu \in N_y \partial T/\nu_{\K}$, $\nu_{\K}(y) \cdot \nu > 0$.  Therefore $\nu \in N_y T$ with $\nu_{\K}(y) \cdot \nu > 0$.

After a change of coordinates, we may assume that $y = 0$, $\nu_{\K}(y) = (0,\ldots,0,0,1)$, and $\nu = (0,\ldots,0,\cos(\phi),\sin(\phi))$ for some $\phi \in (0,\pi/2]$.  We shall use the cylindrical coordinates $x = (x_1,\ldots,x_{n-1},r\cos(\theta),r\sin(\theta))$ on $\R^{n+1}$, where $x_1,\ldots,x_{n-1} \in \R$, $r \geq 0$, and $\theta \in \R$.  Let $C \in \mathbf{I}_m(\R^{n+1}_+)$ be any tangent cone to $T$ at $y$.  Since 
\begin{equation*}
	\op{spt} T \subseteq (\R^{n+1} \setminus \K) \cap \{ x : \nu \cdot (x-y) \leq 0 \} 
\end{equation*}
we have 
\begin{equation*}
	\op{spt} C \subseteq \{ x \in \R^{n+1} : \nu_{\K}(y) \cdot x \geq 0, \, \nu \cdot x \leq 0 \} 
		= \{ x \in \R^{n+1} : \pi/2+\phi \leq \theta \leq \pi \} .
\end{equation*}
In other words, the support of $C$ is contained in the wedge $\{ \pi/2+\phi \leq \theta \leq \pi \}$ with angle $\pi/2-\phi < \pi/2$.  Let $\widetilde{C} = C - \iota_{\#} C$ where $\iota : \R^{n+1} \rightarrow \R^{n+1}$ is the reflection about $\{ x : x_{n+1} = 0 \}$.  Then $\widetilde{C}$ is an area minimizing integral current of $\mathbb{R}^{n+1}$ contained in the wedge $\{ \pi/2+\phi \leq \theta \leq 3\pi/2-\phi \}$ with angle $\pi - 2\phi < \pi$.  However, by Corollary \ref{nontrivial bdry cones cor} and by~\cite[Theorem 36.5]{Sim83} -- which states that the support of an area minimizing cone is contained in a closed half-space $H \subseteq \R^{n+1}$ if and only if the support of the cone is contained in $\partial H$ -- it is impossible for $\widetilde{C}$ to lie in a wedge with angle $< \pi$.
\end{proof}

\begin{lemma} \label{restricted normal lemma2}
Let $\K$ be a bounded proper convex subset of $\R^{n+1}$ with $C^2$-boundary.  Let $T \in \mathbf{I}_m(\R^{n+1} \setminus \K)$ and $R \in \mathbf{I}_{m+1}(\R^{n+1} \setminus \K)$ such that $\partial R = T$ in $\R^{n+1} \setminus \K$, $R$ is relatively area minimizing in $\R^{n+1} \setminus \K$, and $(T,R)$ is a relative isoperimetric minimizer in $\R^{n+1} \setminus \K$.  Then 
\begin{equation} \label{restricted normal2 concl}
	\mathcal{H}^n(RNT) \geq \frac{1}{2} \,\mathcal{H}^n(\mathbb{S}^n). 
\end{equation}
\end{lemma}
\begin{proof}
Assume that $\op{spt} T \cap \partial \K \neq \emptyset$, as otherwise $RNT = \mathbb{S}^n$.  Notice that $\nu_{\K}$ is a continuous function on $\partial \K$ and $\nu_{\partial \K}(y) \in N_y \partial T$ for each $y \in \op{spt} T \cap \partial \K$ since by the convexity of $\K$ we have $\op{spt} T \cap \partial \K \subseteq \K \subseteq \{ x : \nu_{\K}(y) \cdot (x - y) \leq 0 \}$.  Thus as a direct consequence of Theorem \ref{rest norm abs thm} with $X = \op{spt} T \cap \partial \K$ and $\sigma = \nu_{\K}$ 
\begin{equation*}
	\mathcal{H}^n(N\partial T/\nu_{\K}) \geq \frac{1}{2} \,\mathcal{H}^n(\mathbb{S}^n). 
\end{equation*}
Since $N\partial T/\nu_{\K} \subseteq RNT$ by Lemma \ref{restricted normal lemma1}, it follows that \eqref{restricted normal2 concl} holds true.
\end{proof}

\section{Area-mean curvature characterization of hemispheres} \label{sec:AH sec}

The aim of this section is to prove Theorem D following the approach of Almgren in~\cite{Alm86} and using Lemma \ref{restricted normal lemma2} from the previous section.  Throughout this section, we let $\mathcal{K}$ be a bounded proper convex subset of $\mathbb{R}^{n+1}$ with $C^2$-boundary and we let $T \in \mathbf{I}_m(\mathbb{R}^{n+1} \setminus \mathcal{K})$ and $R \in \mathbf{I}_{m+1}(\mathbb{R}^{n+1} \setminus \mathcal{K})$ such that $\partial R = T$ in $\mathbb{R}^{n+1} \setminus \mathcal{K}$, $R$ is relatively area minimizing in $\mathbb{R}^{n+1} \setminus \mathcal{K}$, and $(T,R)$ is relative isoperimetric minimizing in $\mathbb{R}^{n+1} \setminus \mathcal{K}$.  Note that by Corollary \ref{T R cpt spt cor}, $T$ has compact support.  Assume that $T$ has distributional mean curvature $\mathbf{H}_T \in L^{\infty}(\|T\|,\R^{n+1})$ in the sense that \eqref{first variation T concl1} holds true and assume that $|\mathbf{H}_T(y)| \leq m$ for $\|T\|$-a.e.~$y$. 

Let $\A$ denote the convex hull of $\op{spt} T$.  Recall from Subsection \ref{sec:prelims_sets} that $d_{\A} : \R^{n+1} \rightarrow [0,\infty)$ is the distance function given by $d_{\A}(x) = \op{dist}(x,\A)$ for all $x \in \R^{n+1}$.  $\xi_{\A} : \R^{n+1} \rightarrow \A$ is the closest point projection map such that $\xi_{\A}(x)$ is the closest point of $\A$ to $x$ for each $x \in \R^{n+1} \setminus \A$ and $\xi_{\A}(x) = x$ for each $x \in \A$.  $\nu_{\A} : \R^{n+1} \setminus \A \rightarrow \mathbb{S}^n$ is the Gauss map given by $\nu_{\A}(x) = (x-\xi_{\A}(x))/d_{\A}(x)$ for each $x \in \R^{n+1} \setminus \A$.  $d_{\A} \in C^{0,1}(\R^{n+1}) \cap C^{1,1}_{\rm loc}(\R^{n+1} \setminus \A)$ with $\nabla d_{\A} = \nu_{\A}$ in $\R^{n+1} \setminus \A$, $\xi_{\A} \in C^{0,1}(\R^{n+1};\R^{n+1})$, and $\nu_{\A} \in C^{0,1}_{\rm loc}(\R^{n+1} \setminus \A;\R^{n+1})$.  Let $\B_0 = \A \setminus \op{int} \A$ and for each $s > 0$ let 
\begin{equation*}
	\B_s = \{ x \in \R^{n+1} : d_{\A}(x) = s \} .
\end{equation*}
Observe that since $d_{\A} \in C^{1,1}_{\rm loc}(\R^{n+1} \setminus \A)$ with $\nabla d_{\A} = \nu_{\A}$, for each $s > 0$, $\B_s$ is a $C^{1,1}$-submanifold and in particular by Rademacher's theorem the unit normal $\nu_{\A} |_{\B_s}$ is differentiable at $\mathcal{H}^n$-a.e.~$x \in \B_s$.  Thus $\B_s$ has principal curvatures $a_1 \leq a_2 \leq \cdots \leq a_n$ at $\mathcal{H}^n$-a.e.~$x \in \B_s$.  Since $\A$ is convex and $\B_s$ is the level set consisting of points distance $s$ from $\A$, $0 \leq a_i \leq 1/s$ for all $i = 1,2,\ldots,n$. 

Let $x_0 \in \B_s$ such that $\nu_{\A} |_{\B_s}$ is differentiable at $x_0$ and let $y_0 = \xi_{\A}(x_0)$.  We want to compute the Jacobians of $\xi_{\A}$ and $\nu_{\A}$ at $x_0$.  After a change of variables, we may assume $x_0 = (0,s)$, $y_0 = (0,0)$, and $\nu_{\A}(x_0) = (0,1)$.  There exists $r > 0$ and a $C^{1,1}$-function $f : B^n_r(0) \rightarrow \R$ such that $f(0) = s$, $Df(0) = 0$, $f$ is twice differentiable at the origin, and 
\begin{equation} \label{AH halfspheres Bs graph}
	\B_s \cap B^n_r(0,s) \cap [s-r,s+r] = \op{graph} f.
\end{equation}
By Taylor's theorem, after an orthogonal change of variables 
\begin{equation} \label{AH halfspheres taylor}
	f(x') = s - \frac{1}{2} \sum_{i=1}^n a_i x_i^2 + E(x')
\end{equation}
for all $x' = (x_1,x_2,\ldots,x_n) \in B^n_r(0)$ where $a_i$ are the principal curvatures of $\B_s$ at $x_0$ and $E : B^n_r(0) \rightarrow \R$ is a $C^{1,1}$-function which is twice-differentiable at the origin with $E(0) = 0$, $DE(0) = 0$, and $D^2 E(0) = 0$.  Hence 
\begin{equation*}
	\nu_{\A}(x',f(x')) = \frac{(-Df(x'),1)}{\sqrt{1+|Df(x')|^2}}, \quad 
	\xi_{\A}(x',f(x')) = (x',f(x')) - s \,\nu_{\A}(x',f(x')). 
\end{equation*}
By differentiation, 
\begin{equation*}
	D\nu_{\A}(0,s) = \left(\begin{matrix} 
		a_1 & 0 & \cdots & 0 & 0 \\
		0 & a_2 & \cdots & 0 & 0 \\
		\vdots & \vdots & \ddots & \vdots & \vdots \\
		0 & 0 & \cdots & a_n & 0 \\
		0 & 0 & \cdots & 0 & 0
	\end{matrix}\right) , \quad 
	D\xi_{\A}(0,s) = \left(\begin{matrix} 
		1-s a_1 & 0 & \cdots & 0 & 0 \\
		0 & 1-s a_2 & \cdots & 0 & 0 \\
		\vdots & \vdots & \ddots & \vdots & \vdots \\
		0 & 0 & \cdots & 1-s a_n & 0 \\
		0 & 0 & \cdots & 0 & 0
	\end{matrix}\right) .
\end{equation*}
For each $k \in \{1,2,\ldots,n\}$ and each Lipschitz function $F : \B_s \rightarrow \mathbb{R}^{n+1}$ we define the $k$-th Jacobian $J_k F(x_0)$ of $F$ at $\mathcal{H}^n$-a.e.~$x_0 \in \B_s$ by 
\begin{equation*}
	J_k F(x_0) = \left\| {\bigwedge}_k DF(x_0) \right\| ,
\end{equation*} 
see~\cite[Corollary 3.2.20, Theorem 3.2.22]{Fed69}.  The Jacobian of $\nu_{\A} |_{\B_s} : \B_s \rightarrow \mathbb{R}^{n+1}$ is given by 
\begin{equation} \label{AH halfspheres nu jac}
	J_n \nu_{\A}(x_0) = a_1 a_2 \cdots a_n. 
\end{equation}
The $m$-th Jacobian of $\xi_{\A} |_{\B_s} : \B_s \rightarrow \mathbb{R}^{n+1}$ is given by $J_m \xi_{\A}(x_0) = 0$ if $a_m = 1/s$ and 
\begin{equation} \label{AH halfspheres xi jac}
	J_m \xi_{\A}(x_0) = \prod_{i=1}^m (1-s a_i) \neq 0.
\end{equation}
otherwise.  

By~\cite[Proposition 5(4)]{Alm86}, for all $s > 0$ and $\mathcal{H}^n$-a.e.~$x_0 \in \B_s \setminus \xi_{\A}^{-1}(\op{spt} T)$, at least one of the principal curvatures of $\B_s$ at $x_0$ is zero, i.e.~$a_1 = 0$, and so 
\begin{equation*}
	J_n \nu_{\A}(x_0) = 0. 
\end{equation*}

By~\cite[Proposition 6(1)]{Alm86}, for all $0 < s < 1/(3m+3)$ and $\mathcal{H}^n$-a.e.~$x_0 \in \B_s \cap \xi_{\A}^{-1}(\op{spt} T \setminus \K)$, the principal curvatures $a_1 \leq a_2 \leq \cdots \leq a_n$ of $\B_s$ at $x_0$ satisfy 
\begin{equation} \label{AH halfspheres H eqn1}
	a_1 + a_2 + \cdots + a_m \leq m. 
\end{equation}
\cite[Proposition 6]{Alm86} involves constructing a variational vector field $\zeta$ and then applying the first variational formula of $T$ \eqref{first variation T concl1} with this vector field $\zeta$.  \eqref{AH halfspheres H eqn1} follows immediately from~\cite{Alm86} with the following two minor changes.  Firstly, we replace the rectifiable varifold $V = \mathbf{v}(S,\vartheta+1/m,\tau)$ by the varifold $V$ associated with $T$ (as a consequence of using a different isoperimetric ratio from Almgren).  Secondly, we note that the variational vector field $\zeta$ is constructed in~\cite{Alm86} so that $\op{spt} V \cap \op{spt} \zeta$ is contained in a small neighborhood of $y_0$ and thus $\zeta$ can be taken to vanish on an open neighborhood of $\mathcal{K}$.  Now by \eqref{AH halfspheres H eqn1}, $0 \leq a_i \leq m$ for all $i = 1,2,\ldots,m$.  Hence by \eqref{AH halfspheres xi jac}, 
\begin{equation} \label{AH halfspheres xi jac2}
	(2/3)^m \leq J_m \xi_{\A}(x_0) \leq 1, 
\end{equation}
which is what we will need to apply the coarea formula~\cite[Corollary 3.2.22]{Fed69} below. 

Observe that $\B_s \cap \xi_{\A}^{-1}(y_0) = y_0 + s \,N_{y_0} T$ for all $y_0 \in \B_0$ and $s > 0$.  By the variational argument of~\cite[Proposition 6(2)]{Alm86} with the changes mentioned above and by~\cite[Proposition 7]{Alm86}, for $\mathcal{H}^m$-a.e.~$y_0 \in \B_0 \cap \op{spt} T \setminus \K$ and $\mathcal{H}^{n-m}$-a.e.~$\nu \in N_{y_0} T$ there exists constants $0 \leq b_1 \leq b_2 \leq \cdots \leq b_m < \infty$ (depending on $y_0$ and $\nu$ and independent of $s$) such that for every $s > 0$ the Gauss map $\nu_{\A}$ is differentiable at $x_0 = y_0 + s\,\nu$, the principal curvatures $a_1 \leq a_2 \leq \cdots \leq a_n$ of $\B_s$ at $x_0 = y_0 + s\,\nu$ ($a_i$ depending on $s$) are given by 
\begin{equation} \label{AH halfspheres H eqn2}
	a_i(s) = \frac{b_i}{1 + s\,b_i} \text{ for } i = 1,2,\ldots,m, \quad a_i = \frac{1}{s} \text{ for } i = m+1,\ldots,n, 
\end{equation}
and either $b_1 = b_2 = \cdots = b_m = 0$ or 
\begin{equation} \label{AH halfspheres H eqn3}
	0 < b_1 + b_2 + \cdots + b_m \leq \mathbf{H}_T(y_0) \cdot \nu = |\mathbf{H}_T(y_0)| \,\cos(\beta(y_0,\nu)), 
\end{equation}
where $\beta(y_0,\nu) \in [0,\pi/2]$ is the angle such that $\mathbf{H}_T(y_0) \cdot \nu = |\mathbf{H}_T(y_0)| \,\cos(\beta(y_0,\nu))$.  Note that $y_0$ is chosen so that $T$ has an approximate tangent plane $P_{y_0}$ at $y_0$.  Thus $a_i = 1/s$ for $i = m+1,\ldots,n$ since $\B_s \cap \xi_{\A}^{-1}(y_0) \subseteq y_0 + \partial B_s(0) \cap P_{y_0}^{\perp}$.  By \eqref{AH halfspheres nu jac}, \eqref{AH halfspheres xi jac}, and \eqref{AH halfspheres H eqn2}, at $x_0 = y_0 + s\,\nu$ we have 
\begin{align}
	J_m \xi_{\A}(x_0) &= \prod_{i=1}^m (1 - s a_i) = \prod_{i=1}^m \frac{1}{1 + s b_i}, \nonumber \\
	\label{AH halfspheres H eqn4} J_n \nu_{\A}(x_0) &= a_1 a_2 \cdots a_n = \prod_{i=1}^m \frac{b_i}{1 + s b_i} \cdot s^{-(n-m)} 
		= b_1 b_2 \cdots b_m \cdot s^{-(n-m)} \,J_m \xi_{\A}(x_0). 
\end{align}
For $\mathcal{H}^m$-a.e.~$y_0$ we have $|\mathbf{H}_T(y_0)| \leq m$ and thus using the AM-GM inequality and \eqref{AH halfspheres H eqn3}, 
\begin{equation} \label{AH halfspheres H eqn5}
	b_1 b_2 \cdots b_m \leq \left(\frac{b_1+b_2+\cdots+b_m}{m}\right)^m \leq \left(\frac{|\mathbf{H}_T(y_0)|}{m}\right)^m \cos^m(\beta(y_0,\nu)) 
		\leq \cos^m(\beta(y_0,\nu)) 
\end{equation}
with equality if and only if 
\begin{equation} \label{AH halfspheres H eqn6}
	b_1 = b_2 = \cdots = b_m = \cos(\beta(y_0,\nu)), \quad |\mathbf{H}_T(y_0)| = m. 
\end{equation}
Therefore, by \eqref{AH halfspheres H eqn4} and \eqref{AH halfspheres H eqn5}, for $\mathcal{H}^m$-a.e.~$y_0 \in \B_0 \cap \op{spt} T \setminus \K$, $\mathcal{H}^{n-m}$-a.e.~$\nu \in N_{y_0} T$, and all $s > 0$ at $x_0 = y_0 + s\,\nu$ 
\begin{equation} \label{AH halfspheres nu jac2}
	J_n \nu_{\A}(x_0) \leq \cos^m(\beta(y_0,\nu)) \,s^{-(n-m)} \,J_m \xi_{\A}(x_0) 
\end{equation}
with equality if and only if \eqref{AH halfspheres H eqn6} holds true.  

\begin{proof}[Proof of Theorem D]
By Lemma \ref{restricted normal lemma2},  
\begin{equation} \label{AH halfspheres eqn1}
	\mathcal{H}^n(RNT) \geq \frac{1}{2} \,\mathcal{H}^n(\mathbb{S}^n). 
\end{equation}
Our goal is to bound $\mathcal{H}^n(RNT)$ above in terms of $\Mass(T)$. 

Fix $0 < s < 1/(3m+3)$.  Observe that 
\begin{equation} \label{AH halfspheres eqn2}
	RNT = \nu_{\A}( \B_s \cap \xi_{\A}^{-1}(\B_0 \cap \op{spt} T \setminus \K) ) .
\end{equation}
Let $\B^*_s$ be the set of all $x \in \B_s$ such that $\nu_{\A}$ is differentiable at $x$ and $J_n \nu_{\A}(x) > 0$.  Similarly, for each $y \in \B_0 \cap \op{spt} T \setminus \K$, recall that $\B_s \cap \xi_{\A}^{-1}(y) = y + s \,N_y T$ and let $N^*_y T$ be the set of all $\nu \in N_y T$ such that $\nu_{\A}$ is differentiable at $y+ s\nu$ and $J_n \nu_{\A}(y+s\nu) > 0$.  ($N^*_y T$ is independent of $s$ as $\nu \in N^*_y T$ if and only if $\nu_{\A}$ is differentiable at $y+t\nu$ and $J_n \nu_{\A}(y+t\nu) > 0$ for all $t > 0$.)  By \eqref{AH halfspheres eqn2} and the area formula~\cite[Corollary 3.2.20]{Fed69},
\begin{equation} \label{AH halfspheres eqn3}
	\mathcal{H}^n(RNT) = \int_{\B_s \cap \xi_{\A}^{-1}(\B_0 \cap \op{spt} T \setminus \K)} J_n \nu_{\A}(x) \,d\mathcal{H}^n(x) 
		= \int_{\B^*_s \cap \xi_{\A}^{-1}(\B_0 \cap \op{spt} T \setminus \K)} J_n \nu_{\A}(x) \,d\mathcal{H}^n(x). 
\end{equation}
By \eqref{AH halfspheres xi jac2}, $(2/3)^m \leq J_m \xi_{\A}(x) \leq 1$ for $\mathcal{H}^n$-a.e.~$x \in \B^*_s \cap \xi_{\A}^{-1}(\op{spt} T)$ and thus we can apply the coarea formula~\cite[Corollary 3.2.22]{Fed69} using \eqref{AH halfspheres nu jac2} and $\B^*_s \cap \xi_{\A}^{-1}(y) = y + s \,N^*_y T$ for all $y \in \B_0 \cap \op{spt} T \setminus \K$ to obtain 
\begin{align} \label{AH halfspheres eqn4}
	\mathcal{H}^n(RNT) &\leq \int_{\B^*_s \cap \xi_{\A}^{-1}(\B_0 \cap \op{spt} T \setminus \K)} 
		\cos^m(\beta(\xi_{\A}(x),\nu_{\A}(x))) \,s^{-(n-m)} \,J_m \xi_{\A}(x) \,d\mathcal{H}^n(x)
	\\&= \int_{\B_0 \cap \op{spt} T \setminus \mathcal{K}} \int_{\B^*_s \cap \xi_{\A}^{-1}(y)} 
		\cos^m(\beta(y,\nu_{\A}(x))) \,s^{-(n-m)} \,d\mathcal{H}^{n-m}(x) \,d\mathcal{H}^m(y) \nonumber 
	\\&= \int_{\B_0 \cap \op{spt} T \setminus \mathcal{K}} \int_{N^*_y T} 
		\cos^m(\beta(y,\nu)) \,d\mathcal{H}^{n-m}(\nu) \,d\mathcal{H}^m(y), \nonumber 
\end{align}
where $\beta(y,\nu) \in [0,\pi/2]$ is the angle such that $\mathbf{H}_T(y) \cdot \nu = |\mathbf{H}_T(y)| \,\cos(\beta(y,\nu))$ for $\mathcal{H}^m$-a.e.~$y \in \B_0 \cap \op{spt} T \setminus \mathcal{K}$ and $\mathcal{H}^{n-m}$-a.e.~$\nu \in N^*_y T$.  Equality holds true in \eqref{AH halfspheres eqn4} if and only if \eqref{AH halfspheres H eqn6} holds true for $\mathcal{H}^m$-a.e.~$y_0 \in \op{spt} T \setminus \mathcal{K}$ and $\mathcal{H}^{n-m}$-a.e.~$\nu \in N^*_{y_0} T$, where $b_i$ are as in \eqref{AH halfspheres H eqn2} and depend on $y_0$ and $\nu$.  

By \eqref{AH halfspheres H eqn3}, $\mathbf{H}_T(y) \cdot \nu > 0$ for $\mathcal{H}^m$-a.e.~$y \in \B_0 \cap \op{spt} T \setminus \mathcal{K}$ and $\mathcal{H}^{n-m}$-a.e.~$\nu \in N^*_y T$.  Also, $N_y T$ is orthogonal to the approximate tangent plane $P_y$ of $T$ at $\mathcal{H}^m$-a.e.~$y \in \B_0 \cap \op{spt} T \setminus \mathcal{K}$.  Hence for $\mathcal{H}^m$-a.e.~$y \in \B_0 \cap \op{spt} T \setminus \mathcal{K}$ 
\begin{equation} \label{AH halfspheres eqn5}
	\mathcal{H}^{n-m}\big( N^*_y T \setminus (\mathbb{S}^n \cap \{ \nu : \mathbf{H}_T(y) \cdot \nu > 0 \} \cap P_y^{\perp}) \big) = 0. 
\end{equation}
Recall from Theorem C that $\mathbf{H}_T(y)$ is orthogonal to $P_y$ for $\mathcal{H}^m$-a.e.~$y \in \op{spt} T \setminus \mathcal{K}$.  Thus if $\mathbf{H}_T(y) \neq 0$, after an orthogonal change of coordinates we can take $P_y = \{0\} \times \mathbb{R}^m$ and $\mathbf{H}_T(y) = |\mathbf{H}_T(y)| \,e_{n-m+1}$, where $e_{n-m+1}$ is the $(n-m+1)$-th unit coordinate vector, and then use \eqref{AH halfspheres eqn5} to obtain 
\begin{equation} \label{AH halfspheres eqn6}
	\int_{N^*_y T} \cos^m(\beta(y,\nu)) \,d\mathcal{H}^{n-m}(\nu) \,d\mathcal{H}^m(y) 
	\leq \int_{\mathbb{S}^{n-m} \cap \mathbb{R}^{n-m+1}_+} (x \cdot e_{n-m+1})^m \,d\mathcal{H}^{n-m}(x) 
\end{equation}
for $\mathcal{H}^m$-a.e.~$y \in \B_0 \cap \op{spt} T \setminus \mathcal{K}$ with equality if and only if 
\begin{equation} \label{AH halfspheres eqn7}
	\mathcal{H}^{n-m}\big( (\mathbb{S}^n \cap \{ \nu : \mathbf{H}_T(y) \cdot \nu > 0 \} \cap P_y^{\perp}) \setminus N^*_y T \big) = 0. 
\end{equation}
If instead $\mathbf{H}_T(y) = 0$, then \eqref{AH halfspheres eqn5} implies $\mathcal{H}^{n-m}(N^*_y T) = 0$ and thus \eqref{AH halfspheres eqn6} holds true with a strict inequality.  Therefore, for $\mathcal{H}^m$-a.e.~$y \in \B_0 \cap \op{spt} T \setminus \mathcal{K}$, \eqref{AH halfspheres eqn6} holds true with equality if and only if $\mathbf{H}_T(y) \neq 0$ and \eqref{AH halfspheres eqn7} holds true.  By integrating \eqref{AH halfspheres eqn6} over $y \in \B_0 \cap \op{spt} T \setminus \mathcal{K}$, 
\begin{align} \label{AH halfspheres eqn8}
	&\int_{\B_0 \cap \op{spt} T \setminus \mathcal{K}} \int_{N^*_y T} \cos^m(\beta(y,\nu)) \,d\mathcal{H}^{n-m}(\nu) \,d\mathcal{H}^m(y) 
	\\&\hspace{10mm} \leq \mathcal{H}^m(\B_0 \cap \op{spt} T) \int_{\mathbb{S}^{n-m} \cap \mathbb{R}^{n-m+1}_+} 
		(x \cdot e_{n-m+1})^m \,d\mathcal{H}^{n-m}(x) \nonumber 
\end{align}
with equality if and only if $\mathbf{H}_T(y) \neq 0$ and \eqref{AH halfspheres eqn7} holds true for $\mathcal{H}^m$-a.e.~$y \in \B_0 \cap \op{spt} T \setminus \mathcal{K}$. 

We claim that $\mathcal{H}^m$-a.e.~$y \in \B_0 \cap \op{spt} T \setminus \mathcal{K}$ with $\mathbf{H}_T(y) \neq 0$, \eqref{AH halfspheres eqn7} holds true if and only if $\mathcal{H}^{n-m}(N_y T \setminus N^*_y T) = 0$ and 
\begin{equation} \label{AH halfspheres eqn9}
	N_y T = \mathbb{S}^n \cap \{ \nu : \mathbf{H}_T(y) \cdot \nu \geq 0 \} \cap P_y^{\perp}. 
\end{equation}
In fact, we will show that \eqref{AH halfspheres eqn7} $\Rightarrow$ \eqref{AH halfspheres eqn9} as then by \eqref{AH halfspheres eqn9}, \eqref{AH halfspheres eqn7} $\Leftrightarrow$ $\mathcal{H}^{n-m}(N_y T \setminus N^*_y T) = 0$.  Suppose $\mathbf{H}_T(y) \neq 0$ and \eqref{AH halfspheres eqn7} holds true.  Since $N^*_y T \subseteq N_y T$ and $N_y T$ is a closed set, by \eqref{AH halfspheres eqn7} 
\begin{equation*}
	\mathbb{S}^n \cap \{ \nu : \mathbf{H}_T(y) \cdot \nu \geq 0 \} \cap P_y^{\perp} \subseteq N_y T. 
\end{equation*}
We know that $N_y T \subseteq \mathbb{S}^n \cap P_y^{\perp}$.  Thus either \eqref{AH halfspheres eqn9} holds true or $\op{spt} T \subseteq P_y$ and $N_y T = \mathbb{S}^n \cap P_y^{\perp}$.  Since $T \in \mathbf{I}_m(\R^{n+1} \setminus \K)$ with compact support and $\partial T = 0$ in $\R^{n+1} \setminus \K$, by the constancy theorem~\cite[Theorem 26.27]{Sim83} $\op{spt} T \subseteq P_y$ is impossible.  Therefore, \eqref{AH halfspheres eqn9} holds true, proving the claim.  As a consequence of the above claim, equality holds true in \eqref{AH halfspheres eqn8} if and only if $\mathbf{H}_T(y) \neq 0$, $\mathcal{H}^{n-m}(N_y T \setminus N^*_y T) = 0$, and \eqref{AH halfspheres eqn9} holds true for $\mathcal{H}^m$-a.e.~$y \in \B_0 \cap \op{spt} T \setminus \mathcal{K}$.  Note that if $\mathbf{H}_T(y) \neq 0$ and \eqref{AH halfspheres eqn9} holds true for some $y \in \B_0 \cap \op{spt} T \setminus \mathcal{K}$, then $\op{spt} T$ is contained in the $(m+1)$-dimensional affine plane $y + \op{span}(\mathbf{H}_T(y)) \oplus P_y$. 

To bound the term $\mathcal{H}^m(\B_0 \cap \op{spt} T)$ in \eqref{AH halfspheres eqn8} we have  
\begin{equation} \label{AH halfspheres eqn10}
	\mathcal{H}^m(\B_0 \cap \op{spt} T) \leq \mathcal{H}^m(\op{spt} T) \leq \Mass(T)
\end{equation}
with equality if and only if $\op{spt} T \subseteq \B_0$ and $T$ is a multiplicity one integral current. 

Now combining \eqref{AH halfspheres eqn1}, \eqref{AH halfspheres eqn4}, \eqref{AH halfspheres eqn8}, and \eqref{AH halfspheres eqn10}, we have shown that 
\begin{equation} \label{AH halfspheres eqn11}
	\frac{1}{2} \,\mathcal{H}^n(\mathbb{S}^n) 
		\leq \Mass(T) \int_{\mathbb{S}^{n-m} \cap \mathbb{R}^{n-m+1}_+} (x \cdot e_{n-m+1})^m \,d\mathcal{H}^{n-m}(x). 
\end{equation}
We have also shown that equality holds true in \eqref{AH halfspheres eqn11} if and only if 
\begin{enumerate}
	\item[(i)] $\mathcal{H}^n(RNT) = \tfrac{1}{2} \,\mathcal{H}^n(\mathbb{S}^n)$, 
	\item[(ii)] $\op{spt} T \subseteq \B_0$, 
	\item[(iii)] $T$ is a multiplicity one integral current, 
	\item[(iv)] \eqref{AH halfspheres H eqn6} holds true for $\mathcal{H}^m$-a.e.~$y_0 \in \op{spt} T \setminus \mathcal{K}$ and $\mathcal{H}^{n-m}$-a.e.~$\nu \in N_{y_0} T$, and 
	\item[(v)] \eqref{AH halfspheres eqn9} holds true for $\mathcal{H}^m$-a.e.~$y \in \op{spt} T \setminus \mathcal{K}$ and in particular the support of $T$ lies in an $(m+1)$-dimensional affine plane. 
\end{enumerate}
In particular, (i) is simply equality in \eqref{AH halfspheres eqn1}.  (ii) and (iii) hold true if and only if equality holds true in \eqref{AH halfspheres eqn10}.  In light of (ii), equality holds true in \eqref{AH halfspheres eqn4} and \eqref{AH halfspheres eqn8} if and only if for $\mathcal{H}^m$-a.e.~$y \in \op{spt} T \setminus \mathcal{K}$ we have 
\begin{enumerate}
	\item[(a)] \eqref{AH halfspheres H eqn6} holds true with $y_0 = y$ for $\mathcal{H}^{n-m}$-a.e.~$\nu \in N^*_y T$, 
	\item[(b)] $\mathbf{H}_T(y) \neq 0$, 
	\item[(c)] $\mathcal{H}^{n-m}(N_y T \setminus N^*_y T) = 0$, and 
	\item[(d)] \eqref{AH halfspheres eqn9} holds true. 
\end{enumerate}
However, if $\nu \in N_y T$ satisfies \eqref{AH halfspheres H eqn6} with $y_0 = y$, then $|\mathbf{H}_T(y)| = m \neq 0$ and, provided $\beta(y,\nu) \neq \pi/2$, $b_1 = b_2 = \cdots b_m = \cos(\beta(y,\nu)) > 0$ and so $\nu \in N^*_y T$.  Using this one readily checks that (iv)--(v) $\Leftrightarrow$ (a)--(d) for $\mathcal{H}^m$-a.e.~$y \in \op{spt} T \setminus \mathcal{K}$.  Therefore equality holds true in \eqref{AH halfspheres eqn4} and \eqref{AH halfspheres eqn8} if and only if (iv) and (v) hold true.  

In the special case that $B_2(0) \setminus \mathcal{K} = B_2(0) \cap \R^{n+1}_+$ and $T = \llbracket (\{0\} \times \mathbb{S}^m) \cap \R^{n+1}_+ \rrbracket$ is a multiplicity one $m$-dimensional hemisphere in $\R^{n+1}_+$, obviously (i)--(v) all hold true and consequently equality holds true in \eqref{AH halfspheres eqn11}.  Thus we can substitute $\Mass(T) = \frac{1}{2} \,\mathcal{H}^n(\mathbb{S}^n)$ into \eqref{AH halfspheres eqn11} and cancel terms to obtain 
\begin{equation*}
	1 = \int_{\mathbb{S}^{n-m} \cap \mathbb{R}^{n-m+1}_+} (x_{n-m+1})^m \,d\mathcal{H}^{n-m}(x).  
\end{equation*}
Therefore, for general $\mathcal{K}$ and $T$, \eqref{AH halfspheres eqn11} gives us 
\begin{equation} \label{AH halfspheres eqn12}
	\frac{1}{2} \,\mathcal{H}^n(\mathbb{S}^m) \leq \Mass(T) 
\end{equation}
with equality if and only if conditions (i)--(v) above hold true. 

Now suppose equality holds true in \eqref{AH halfspheres eqn12}.  By (v), after translating and rotating we may assume $\op{spt} T \subseteq \R^{m+1} \times \{0\}$.  As we discussed above, the support of $T$ cannot be contained in an $m$-dimensional plane by the constancy theorem, so $\R^{m+1} \times \{0\}$ is the smallest affine plane (with respect to set inclusion) containing $\A$.  By (v), $\op{spt} T$ must lie on the relative boundary of $\A$ in $\R^{m+1} \times \{0\}$.  Moreover, by Lemma \ref{isoper lower bound lemma} $T$ bounds some $(m+1)$-dimensional integral current $Q$ in $(\R^{m+1} \times \{0\}) \setminus \K$ with compact support.  Since $\K$ and $\A$ are bounded convex sets and $T$ has compact support, $\R^{n+1} \setminus (\K \cup \A)$ is connected and thus by the constancy theorem~\cite[Theorem 26.27]{Sim83} $\op{spt} Q \subseteq \A \setminus \op{int} \K$.  Since the support of $T$ lies on the relative boundary of $\A$ in $\R^{n+1}$ and $T$ satisfies (iii), by the constancy theorem for every $y \in \op{spt} T \setminus \K$ and $0 < \delta < \op{dist}(y,\K)$ 
\begin{equation} \label{AH halfspheres eqn13}
	T = \pm \partial \llbracket \mathcal{A} \rrbracket \text{ in } B_{\delta}(y) \cap (\R^{m+1} \times \{0\}),
\end{equation}
where the sign $\pm$ is determined by the orientation of $T$ (and depends on the ball $B_{\delta}(y)$). 

Take any point $y \in \op{spt} T \setminus \K$ and any tangent cone $C$ of $T$ at $y$.  By blowing up both $\mathcal{A}$ and $T$ at $y$ and using \eqref{AH halfspheres eqn13}, we conclude that $C$ is a multiplicity one integral current and the support of $C$ lies in an $(m+1)$-dimensional half-space $H \subset \R^{m+1} \times \{0\}$ with $0 \in \partial H$, which by~\cite[Theorem 36.5]{Sim83} implies $\op{spt} C \subseteq \partial H$.  Therefore, $C$ is a multiplicity one $m$-dimensional plane.  Since $y$ and $C$ were arbitrary, we may apply the Allard regularity theorem~\cite{All72} to conclude that $\op{spt} T \setminus \K$ is a locally $m$-dimensional $C^{1,\mu}$-submanifold for all $\mu \in (0,1)$.  Moreover, by (iv), $\op{spt} T \setminus \K$ is a codimension one submanifold of $\R^{m+1} \times \{0\}$ with constant scalar mean curvature $m$ and thus by elliptic regularity $\op{spt} T \setminus \K$ is a smooth $m$-dimensional submanifold. 

Notice that since $\op{spt} T \setminus \K$ is a smooth planar submanifold that $b_1,\ldots,b_m$ as in \eqref{AH halfspheres H eqn2} are the principal curvatures of $\op{spt} T \setminus \K$ at $y_0 \in \op{spt} T \setminus \K$.  Thus (iv) implies that $\op{spt} T \setminus \K$ is a totally umbilical submanifold with principal curvature one at each point.  Hence by the Nabelpunktsatz theorem, see~\cite[Lemma 1, p. 8]{Spi99} or~\cite[Theorem 26, p. 75]{Spi99}, each connected component of $\op{spt} T \setminus \K$ is a subset of an $m$-dimensional unit sphere of $\R^{m+1} \times \{0\}$.  In fact, by the constancy theorem (see Remark \ref{first variation T rmk}), for each $y \in \op{spt} T \setminus \K$ there exists $\delta > 0$ and $z \in \R^{m+1} \times \{0\}$ such that $\op{spt} T \cap B_{\delta}(y) = (z + \mathbb{S}^m \times \{0\}) \cap B_{\delta}(y)$.  It follows that each connected component of $\op{spt} T \setminus \K$ is equal to a connected component of $(z + \sphere \times \{0\}) \setminus \K$ for some $z \in \R^{m+1} \times \{0\}$.  If the closure of two or more connected components of $\op{spt} T \setminus \K$ intersect at $y \in \partial \K$, then by Lemma \ref{boundary rectifiability lemma4} the closure of components intersect both each other and $\partial \K$ transversely at $y$.  Thus the tangent cone $C$ to $T$ at $y$ is a sum of two or more distinct $m$-dimensional half-planes meeting along a common boundary, contradicting $C$ being relatively area minimizing in $\R^{n+1} \setminus \K_{y,0}$.  Hence the connected components of $\op{spt} T \setminus \K$ have mutually disjoint closures.  Thus by Remark \ref{first variation T rmk} each connected component of $\op{spt} T \setminus \K$ meets $\partial \K$ orthogonally.  Let $M$ be any connected component of $\op{spt} T \setminus \K$.  After translating assume that $M$ is a connected component of $(\sphere \times \{0\}) \setminus \K$.  Since $\Mass(T) = \frac{1}{2} \,\mathcal{H}^m(\sphere)$, $\overline{M} \cap \partial \K \neq \emptyset$.  Fix $y \in \overline{M} \cap \partial \K$.  Since $\K$ is convex, $\K$ is contained in the half-space $\K_0 = \{ x : \nu_{\K}(y) \cdot (x-y) \leq 0 \}$.   Since $M$ meets $\partial \K$ orthogonally at $y$, $M \setminus \K_0 = (\sphere \times \{0\}) \setminus \K_0$ is a hemisphere.  Thus 
\begin{equation*}
	\frac{1}{2} \,\mathcal{H}^m(\sphere) 
	= \mathcal{H}^m(M \setminus \K_0) 
	\leq \mathcal{H}^m(M) \leq \Mass(T) = \frac{1}{2} \,\mathcal{H}^m(\sphere), 
\end{equation*}
so $\op{spt} T \setminus \K$ has precisely one connected component and $T = \pm \llbracket (\sphere \times \{0\}) \setminus \K_0 \rrbracket$ as a multiplicity one $m$-dimensional hemisphere, where the sign $\pm$ is determined by the orientation of $T$. 
\end{proof}

\section{Proof of the sharp relative isoperimetric inequality} \label{sec:main proof sec}

Finally, here we prove Theorem A.  We will break the proof up into three smaller proofs covering the following cases:
\begin{enumerate}
	\item[(i)] $\K$ is bounded and has $C^2$-boundary, 
	\item[(ii)] $\K$ is bounded but does not have $C^2$-boundary, and 
	\item[(iii)] $\K$ is unbounded.
\end{enumerate}
The first case will follow easily from Theorems B, C, and D and the latter two cases will follow via approximation arguments and the prior cases. 

\begin{proof}[Proof of Theorem A in the case that $\K$ is bounded and has $C^2$-boundary]
Suppose $\K$ is a bounded proper convex subset of $\R^{n+1}$ with $C^2$-boundary.  We want to show that \eqref{rel iso ineq} always holds true.  Equivalently, we want to show that 
\begin{equation} \label{rel iso1 eqn1}
	\gamma_{m,n}(\K) = 2^{-\frac{1}{m}} \,\frac{\mathcal{H}^m(\partial \ball)^{\frac{m+1}{m}}}{\mathcal{H}^{m+1}(\ball)}.
\end{equation}
Suppose to the contrary that 
\begin{equation} \label{rel iso1 eqn2}
	\gamma_{m,n}(\K) < 2^{-\frac{1}{m}} \,\frac{\mathcal{H}^m(\partial \ball)^{\frac{m+1}{m}}}{\mathcal{H}^{m+1}(\ball)}.
\end{equation}
By Theorem B and \eqref{rel iso1 eqn2}, there exists a relative isoperimetric minimizer $(T,R)$ in $\R^{n+1} \setminus \K$.  By Theorem C, $T$ has distributional mean curvature $\mathbf{H}_T \in L^{\infty}(\|T\|,\R^{n+1})$ in the sense that \eqref{first variation T concl1} holds true and $|\mathbf{H}_T(y)| \leq H_0$, where $H_0$ is given by \eqref{H0 defn}.  By Corollary \ref{T R cpt spt cor}, $T$ has compact support.  Rescale so that 
\begin{equation} \label{rel iso1 eqn3}
	\Mass(R) = \frac{1}{2} \,\mathcal{H}^{m+1}(\ball). 
\end{equation}
By \eqref{rel iso1 eqn3}, \eqref{rel iso1 eqn2}, and \eqref{H0 defn},    
\begin{equation} \label{rel iso1 eqn4}
	\Mass(T) < \frac{1}{2} \,\mathcal{H}^m(\sphere), \quad
	H_0 = \frac{m}{m+1} \,\frac{\Mass(T)}{\Mass(R)} < m. 
\end{equation}
However, by Theorem D, \eqref{rel iso1 eqn4} is impossible.  Therefore we must have \eqref{rel iso1 eqn1}.

Next suppose $\K$ is a bounded proper convex subset of $\R^{n+1}$ with $C^2$-boundary and suppose $T \in \mathbf{I}_m(\R^{n+1} \setminus \K)$ and $R \in \mathbf{I}_{m+1}(\R^{n+1} \setminus \K)$ such that $R$ is relatively area minimizing with $\partial R = T$ in $\R^{n+1} \setminus \K$ and equality holds true in \eqref{rel iso ineq}.  In other words, suppose $(T,R)$ is relative isoperimetric minimizing in $\R^{n+1} \setminus \K$.  By Theorem C, $T$ has distributional mean curvature $\mathbf{H}_T \in L^{\infty}(\|T\|,\R^{n+1})$ in the sense that \eqref{first variation T concl1} holds true and $|\mathbf{H}_T(y)| \leq H_0$, where $H_0$ is given by \eqref{H0 defn}.  By Corollary \ref{T R cpt spt cor}, $T$ has compact support.  Rescale so that \eqref{rel iso1 eqn3} holds true and thus by \eqref{rel iso1 eqn1} and \eqref{H0 defn} 
\begin{equation*}
	\Mass(T) = \frac{1}{2} \,\mathcal{H}^m(\sphere), \quad
	H_0 = \frac{m}{m+1} \,\frac{\Mass(T)}{\Mass(R)} = m. 
\end{equation*}
By Theorem D, $T$ is a multiplicity one $m$-dimensional hemisphere which meets $\partial \K$ orthogonally.  

We need to show that $R$ is a multiplicity one $(m+1)$-dimensional flat disk bounded by $T$ in $\R^{n+1} \setminus \K$.  Let $D \in \mathbf{I}_m(\R^{n+1})$ be the multiplicity one $m$-dimensional flat unit disk with $\partial D = -\partial T$ in $\R^{n+1}$ and $R_0 \in \mathbf{I}_{m+1}(\R^{n+1})$ be the multiplicity one $(m+1)$-dimensional flat half-disk such that $\partial R_0 = T + D$ in $\R^{n+1}$.  We want to show that $R = R_0$.  Since $\K$ is convex and $\op{spt} \partial T \subseteq \partial \K$, $\op{spt} D \subseteq \K$.  $R_0$ satisfies $\partial R_0 = T$ in $\R^{n+1} \setminus \K$ and 
\begin{equation*}
	\Mass(R_0 \llcorner \R^{n+1} \setminus \K) \leq \Mass(R_0) = \frac{1}{2} \,\mathcal{H}^{m+1}(\ball) = \Mass(R), 
\end{equation*}
which since $R$ is relatively area minimizing with $\partial R = T$ in $\R^{n+1} \setminus \K$ implies that  $\|R_0\|(\K) = 0$ and $R_0$ is relatively area minimizing in $\R^{n+1} \setminus \K$.  In particular, $\op{spt} D \subset \partial \K$.  By Remark \ref{first variation R rmk}, since $R_0$ is relatively area minimizing in $\R^{n+1} \setminus \K$, $R_0$ meets $\partial \K$ orthogonally.  Hence, fixing any $y \in \op{spt} D$, $\nu_{\mathcal{K}}(y)$ lies in the $(m+1)$-dimensional affine plane passing through $R_0$.  Since $\K$ is convex, $\K$ lies in the half-space $\K_0 = \{ x \in \R^{n+1} : \nu_{\mathcal{K}}(y) \cdot (x-y) \leq 0 \}$.  Thus $R$ satisfies $\partial R = T$ in $\R^{n+1} \setminus \K_0$ and 
\begin{equation*}
	\Mass(R \llcorner \R^{n+1} \setminus \K_0) \leq \Mass(R) = \frac{1}{2} \,\mathcal{H}^{m+1}(\ball) = \Mass(R_0), 
\end{equation*}
which since $R_0$ is relatively area minimizing with $\partial R_0 = T$ in $\R^{n+1} \setminus \K_0$ implies that  $\|R\|(\K_0) = 0$ and $R$ is relatively area minimizing in $\R^{n+1} \setminus \K_0$.  By the sharp isoperimetric inequality in the half-space $\R^{n+1} \setminus \K_0$, $R = R_0$ in $\R^{n+1}$. 
\end{proof}

\begin{proof}[Proof of Theorem A in the case that $\K$ is bounded but does not have $C^2$-boundary]
Suppose $\K$ is a bounded proper convex subset of $\R^{n+1}$ but $\K$ does not have $C^2$-boundary.  Let $T \in \mathbf{I}_m(\R^{n+1} \setminus \K)$ and $R \in \mathbf{I}_{m+1}(\R^{n+1} \setminus \K)$ such that $R$ is relatively area minimizing with $\partial R = T$ in $\R^{n+1} \setminus \K$. 

We claim that for every $\delta > 0$ there exists a bounded proper convex set $\K_{\delta} \subset \R^{n+1}$ such that $\K_{\delta}$ has a smooth boundary and 
\begin{equation} \label{rel iso2 eqn1}
	\K \subset \K_{\delta} \subset \{ x \in \R^{n+1} : d_{\K}(x) < \delta \} .
\end{equation}
We will construct $\K_{\delta}$ as a level set of a smooth approximation $f$ of the distance function $d_{\K}$.  Let $\phi \in C^{\infty}_c(B_1(0))$ be a nonnegative function such that $\int \phi = 1$ and for each $\sigma > 0$ let $\phi_{\sigma}(x) = \sigma^{-n-1} \,\phi(x/\sigma)$.  Set $\sigma = \delta/16$ and define $f : \R^{n+1} \rightarrow [0,\infty)$ by the convolution $f = d_{\K} \ast \phi_{\sigma}$.  Set 
\begin{equation*}
	\K_{\delta} = \{ x \in \R^{n+1} : f(x) \leq \delta/2 \}. 
\end{equation*}
By the properties of convolution, $f$ is a smooth function.  Since $\K$ is a convex set, $d_{\K}$ is a convex function.  It follows that for each $x,y \in \K$ and $t \in [0,1]$ 
\begin{align*}
	f((1-t)\,x+t\,y) &= \int_{B_{\sigma}(0)} d_{\K}((1-t)\,x+t\,y-z) \,\phi_{\sigma}(z) \,dz 
	\\&\leq \int_{B_{\sigma}(0)} \left( (1-t) \,d_{\K}(x-z) + t \,d_{\K}(y-z) \right) \phi_{\sigma}(z) \,dz 
	\\&= (1-t) \,f(x) + t \,f(y) 
\end{align*}
and thus $f$ is a convex function.  Hence $\K_{\delta}$ is a convex set.  Since $d_{\K}$ is Lipschitz with $\op{Lip} d_{\K} = 1$ and $\sigma = \delta/16$, 
\begin{equation*}
	|f(x) - d_{\K}(x)| \leq \int_{B_{\sigma}(0)} |d_{\K}(x-z) - d_{\K}(x)| \,\phi_{\sigma}(z) \,dz 
		\leq \int_{B_{\sigma}(0)} |z| \,\phi_{\sigma}(z) \,dz \leq \sigma = \frac{\delta}{16}
\end{equation*}
for all $x \in R^{n+1}$.  It follows that \eqref{rel iso2 eqn1} holds true.  In particular, $0 \in \op{int} \K_{\delta}$, so $\K_{\delta}$ is a proper convex set.  Also, $\K_{\delta}$ is bounded.  Finally, we know that $d_{\K} \in C^{1,1}_{\rm loc}(\R^{n+1} \setminus \K)$ with $\nabla d_{\K} = \nu_{\K}$ in $\R^{n+1} \setminus \K$ and $\op{Lip}_{\{d_{\K} \geq s\}} \nu_{\K} \leq 3/s$ for each $s > 0$.  Thus, using $\sigma = \delta/16$, 
\begin{equation*}
	\|\nabla f(x) - \nu_{\K}(x)\| \leq \int_{B_{\sigma}(0)} |\nu_{\K}(x-z) - \nu_{\K}(x)| \,\phi_{\sigma}(z) \,dz 
		\leq \frac{8}{\delta} \int_{B_{\sigma}(0)} |z| \,\phi_{\sigma}(z) \,dz \leq \frac{8 \sigma}{\delta} = \frac{1}{2} 
\end{equation*}
for all $x \in \partial \K_{\delta}$, where we used $d_{\K}(x-z) \geq f(x) - 2\sigma = 3\delta/8$ for all $x \in \partial \K_{\delta}$ and $z \in B_{\sigma}(0)$.  Hence since $\nu_{\K}$ is a unit vector, $\|\nabla f(x)\| \geq 1/2$ for all for all $x \in \partial \K_{\delta}$.  Therefore, recalling that $f$ is smooth and using the implicit function theorem, $\K_{\delta}$ has a smooth boundary. 

Next for each integer $j \geq 1$, let $\delta_j = 2^{-j}$ and $T_j = T \llcorner \K_{\delta_j} \in \mathbf{I}_m(\R^{n+1} \setminus \K_{\delta_j})$.  Let $R_j \in \mathbf{I}_{m+1}(\R^{n+1} \setminus \K_{\delta_j})$ such that $R_j$ is relatively area minimizing with $\partial R_j = T_j$ in $\R^{n+1} \setminus \K_{\delta_j}$.  Obviously $T_j \rightarrow T$ in the mass norm topology in $\R^{n+1}$ and in particular 
\begin{equation} \label{rel iso2 eqn3}
	\Mass(T) = \lim_{j \rightarrow \infty} \Mass(T_j). 
\end{equation}
Since $\partial R_j = \partial R = T$ in $\R^{n+1} \setminus \K_{\delta_j}$ and $R_j$ is relatively area minimizing in $\R^{n+1} \setminus \K_{\delta_j}$, $\Mass(R_j) \leq \Mass(R)$ for all $j$.  Thus by the Federer-Fleming compactness theorem after passing to a subsequence $R_j \rightarrow Q$ weakly in $\R^{n+1} \setminus \K$ for some current $Q \in \mathbf{I}_{m+1}(\R^{n+1} \setminus \K)$ with $\partial Q = T$ in $\R^{n+1} \setminus \K$.  By the semi-continuity of mass, 
\begin{equation} \label{rel iso2 eqn4}
	\Mass(Q) \leq \liminf_{j \rightarrow \infty} \Mass(R_j) \leq \limsup_{j \rightarrow \infty} \Mass(R_j) \leq \Mass(R)
\end{equation}
and in particular $\Mass(Q) \leq \Mass(R)$.  But $R$ is relatively area minimizing in $\R^{n+1} \setminus \K$, so $Q$ is also relatively area minimizing with $\partial Q = T$ in $\R^{n+1} \setminus \K$ and $\Mass(Q) = \Mass(R)$.  In particular, \eqref{rel iso2 eqn4} gives us 
\begin{equation} \label{rel iso2 eqn5}
	\Mass(Q) = \lim_{j \rightarrow \infty} \Mass(R_j) = \Mass(R). 
\end{equation}
For each $j \geq 1$, since $\K_{\delta_j}$ is a bounded proper convex set with a smooth boundary and $R_j$ is relatively area minimizing with $\partial R_j = T_j$ in $\R^{n+1} \setminus \K_{\delta_j}$, by the relative isoperimetric inequality in $\R^{n+1} \setminus \K_{\delta_j}$
\begin{equation*}
	\frac{\Mass(T_j)^{\frac{m+1}{m}}}{\Mass(R_j)} \geq 2^{-\frac{1}{m}} \,\frac{\mathcal{H}^m(\sphere)^{\frac{m+1}{m}}}{\mathcal{H}^{m+1}(\ball)}. 
\end{equation*}
Thus by \eqref{rel iso2 eqn3} and \eqref{rel iso2 eqn5}
\begin{equation*}
	\frac{\Mass(T)^{\frac{m+1}{m}}}{\Mass(R)} 
	= \lim_{j \rightarrow \infty} \frac{\Mass(T_j)^{\frac{m+1}{m}}}{\Mass(R_j)} 
	\geq 2^{-\frac{1}{m}} \,\frac{\mathcal{H}^m(\sphere)^{\frac{m+1}{m}}}{\mathcal{H}^{m+1}(\ball)}. \qedhere
\end{equation*}
\end{proof}

\begin{proof}[Proof of Theorem A in the case that $\K$ is unbounded]
Suppose $\K$ is an unbounded proper convex subset of $\R^{n+1}$, with no assumptions on boundary regularity.  By translating, assume $0 \in \op{int} \K$.  Let $T \in \mathbf{I}_m(\R^{n+1} \setminus \K)$ and $R \in \mathbf{I}_{m+1}(\R^{n+1} \setminus \K)$ such that $R$ is relatively area minimizing with $\partial R = T$ in $\R^{n+1} \setminus \K$.  Let $\varepsilon > 0$.  Since $\Mass(T) + \Mass(R) < \infty$, there exists $\rho \in [1,\infty)$ such that 
\begin{equation*} 
	\Mass(T \llcorner \R^{n+1} \setminus B_{\rho}(0)) + \Mass(R \llcorner \R^{n+1} \setminus B_{\rho}(0)) < \varepsilon 
\end{equation*}
and thus by slicing theory using \eqref{meas_good_slices} with $\vartheta = 1/2$ there exists $\rho^* \in (\rho,\rho+1)$ such that $T \llcorner B_{\rho^*}(0), T \llcorner \R^{n+1} \setminus B_{\rho^*}(0) \in \mathbf{I}_m(\R^{n+1} \setminus \K)$, $R \llcorner B_{\rho^*}(0), R \llcorner \R^{n+1} \setminus B_{\rho^*}(0)  \in \mathbf{I}_{m+1}(\R^{n+1} \setminus \K)$, and $\langle R, |\cdot|, \rho^* \rangle \in \mathbf{I}_m(\R^{n+1} \setminus \K)$ with 
\begin{gather} 
	\label{rel iso3 eqn1} \partial (R \llcorner B_{\rho^*}(0)) 
		= T \llcorner B_{\rho^*}(0) + \langle R, |\cdot|, \rho^* \rangle \text{ in } \R^{n+1} \setminus \K, \\
	\label{rel iso3 eqn2} \partial (R \llcorner \R^{n+1} \setminus B_{\rho^*}(0)) 
		= T \llcorner \R^{n+1} \setminus B_{\rho^*}(0) - \langle R, |\cdot|, \rho^* \rangle \text{ in } \R^{n+1} \setminus \K, \\
	\label{rel iso3 eqn3} \Mass(T \llcorner \R^{n+1} \setminus B_{\rho^*}(0)) + \Mass(R \llcorner \R^{n+1} \setminus B_{\rho^*}(0)) < \varepsilon , \\
	\label{rel iso3 eqn4} \Mass(\langle R, |\cdot|, \rho^* \rangle) \leq 2 \,\Mass(R \llcorner \R^{n+1} \setminus B_{\rho}(0)) < 2\varepsilon . 
\end{gather}
Set 
\begin{equation*}
	\widetilde{T} = T \llcorner B_{\rho^*}(0) + \langle R, |\cdot|, \rho^* \rangle, \quad 
	\widetilde{R} = R \llcorner B_{\rho^*}(0), \quad 
	\widetilde{\K} = \K \cap \overline{B_{2\rho^*}(0)}.
\end{equation*}
Clearly $\widetilde{\K}$ is a bounded proper convex set with $0 \in \op{int} \widetilde{\K}$, $\widetilde{T} \in \mathbf{I}_m(\R^{n+1} \setminus \widetilde{\K})$, and $\widetilde{R} \in \mathbf{I}_{m+1}(\R^{n+1} \setminus \widetilde{\K})$.  By \eqref{rel iso3 eqn1}, $\partial \widetilde{R} = \widetilde{T}$ in $\R^{n+1} \setminus \widetilde{\K}$.  By \eqref{rel iso3 eqn3} and \eqref{rel iso3 eqn4},  
\begin{equation} \label{rel iso3 eqn5}
	\Mass(T - \widetilde{T}) + \Mass(R - \widetilde{R}) < 3\varepsilon . 
\end{equation}

We claim that $\widetilde{R}$ is relatively area minimizing in $\R^{n+1} \setminus \widetilde{K}$.  Suppose $\widetilde{Q} \in \mathbf{I}_{m+1}(\R^{n+1} \setminus \widetilde{\K})$ with $\partial \widetilde{Q} = \widetilde{T}$ in $\R^{n+1} \setminus \widetilde{\K}$.  Set $Q = \widetilde{Q} + R \llcorner \R^{n+1} \setminus B_{\rho^*}(0)$ so that by $\partial \widetilde{Q} = \widetilde{T}$ in $\R^{n+1} \setminus \widetilde{\K}$, \eqref{rel iso3 eqn2}, and the definition of $\widetilde{T}$ 
\begin{equation*}
	\partial Q = \widetilde{T} + T \llcorner \R^{n+1} \setminus B_{\rho^*}(0) - \langle R, |\cdot|, \rho^* \rangle = T 
\end{equation*}
in $\R^{n+1} \setminus \K$.  Since $R$ is relatively area minimizing in $\R^{n+1} \setminus \K$, $\Mass(R) \leq \Mass(Q)$.  Thus 
\begin{equation*}
	\Mass(\widetilde{R}) + \Mass(R \llcorner \R^{n+1} \setminus B_{\rho^*}(0)) = \Mass(R) 
		\leq \Mass(Q) \leq \Mass(\widetilde{Q}) + \Mass(R \llcorner \R^{n+1} \setminus B_{\rho^*}(0)). 
\end{equation*}
Hence $\Mass(\widetilde{R}) \leq \Mass(\widetilde{Q})$.  Therefore, $\widetilde{R}$ is relatively area minimizing with $\partial \widetilde{R} = \widetilde{T}$ in $\R^{n+1} \setminus \widetilde{\K}$. 

Now by \eqref{rel iso3 eqn5} and the sharp relative isoperimetric inequality for $\R^{n+1} \setminus \widetilde{K}$, 
\begin{equation} \label{rel iso3 eqn6}
	\frac{(\Mass(T) + 3\varepsilon)^{\frac{m+1}{m}}}{\Mass(R) - 3\varepsilon} 
	\geq \frac{\Mass(\widetilde{T})^{\frac{m+1}{m}}}{\Mass(\widetilde{R})} 
	\geq 2^{-\frac{1}{m}} \,\frac{\mathcal{H}^{m+1}(\ball)^{\frac{m+1}{m}}}{\mathcal{H}^m(\sphere)}. 
\end{equation}
Letting $\varepsilon \downarrow 0$ in \eqref{rel iso3 eqn6} yields \eqref{rel iso ineq}. 
\end{proof}

\end{document}